\documentclass[10pt,oneside,reqno]{article}
\usepackage[margin=1.3in]{geometry}
\usepackage[title]{appendix}
\usepackage{titlesec}
\usepackage{amsmath,amsthm,amssymb,color,bm}
\usepackage[authoryear]{natbib}
\usepackage{booktabs}

\newcommand{\dd}{\,\mathrm d}
\newcommand{\calW}{\mathcal W}
\newcommand{\Dkl}{D}
\newcommand{\Ip}{I_p}

\usepackage{tikz}
\RequirePackage{amsfonts,mathrsfs,dsfont,mathtools,thmtools}
\RequirePackage[colorlinks,citecolor=blue,urlcolor=blue,linkcolor=blue]{hyperref}
\usepackage[capitalize]{cleveref}
\RequirePackage{graphicx}

\crefname{equation}{}{}
\crefformat{equation}{\textup{(#2#1#3)}}
\crefrangeformat{equation}{\textup{(#3#1#4)--(#5#2#6)}}
\crefdefaultlabelformat{#2\textup{#1}#3}
\crefname{lemma}{Lemma}{Lemmas}
\crefname{page}{p.}{pp.}
\usepackage[normalem]{ulem}
\usepackage{bbm}
\usepackage{enumitem}
\usetikzlibrary{arrows, automata}
\usetikzlibrary{calc}
\usetikzlibrary{positioning}
\usepackage{cancel}

\numberwithin{equation}{section}
\allowdisplaybreaks[4]

\theoremstyle{plain}
\newtheorem{theorem}{Theorem}[section]
\newtheorem{assumption}{Assumption}[section]
\newtheorem{proposition}{Proposition}[section]
\newtheorem{lemma}{Lemma}[section]
\newtheorem{corollary}{Corollary}[section]

\theoremstyle{definition}

\newtheorem{remark}{Remark}[section]

\makeatletter

\newcount\minute
\newcount\hour
\newcount\hourMins
\def\now{%
\minute=\time%
\hour=\time \divide \hour by 60%
\hourMins=\hour \multiply\hourMins by 60%
\advance\minute by -\hourMins%
\zeroPadTwo{\the\hour}:\zeroPadTwo{\the\minute}%
}
\def\zeroPadTwo#1{\ifnum #1<10 0\fi#1}

\renewcommand{\cite}{\citet}

\def\^#1{\ifmmode {\mathaccent"705E #1} \else {\accent94 #1} \fi}
\def\~#1{\ifmmode {\mathaccent"707E #1} \else {\accent"7E #1} \fi}

\edef\-#1{\noexpand\ifmmode {\noexpand\bar{#1}} \noexpand\else \-#1\noexpand\fi}
\def\>#1{\vec{#1}}
\def\.#1{\dot{#1}}

\def\atop{\@@atop}
\def\*#1{\mathscr{#1}}

\renewcommand{\leq}{\leqslant}
\renewcommand{\le}{\leqslant}
\renewcommand{\geq}{\geqslant}
\renewcommand{\ge}{\geqslant}
\newcommand{\eps}{\varepsilon}

\newcommand{\IE}{\mathbbm{E}}
\newcommand{\IP}{\mathbbm{P}}
\newcommand{\Var}{\mathop{\mathrm{Var}}\nolimits}

\newcommand{\e}{{\mathrm{e}}}
\newcommand{\IR}{\mathbb{R}}

\def\ben#1{\begin{equation}#1\end{equation}}

\def\besn#1{\begin{equation}\begin{split}#1\end{split}\end{equation}}

\def\norm#1{\Vert#1\Vert}

\def\abs#1{\vert#1\vert}

\def\mid{\vert}

\def\beqn#1\eeqn{\begin{align}#1\end{align}}
\def\beq#1\eeq{\begin{align*}#1\end{align*}}

\usepackage{latexsym}
\usepackage{amscd}
\usepackage{epsf}

\DeclareMathOperator*{\argsup}{argsup}
\def\E{{\IE}}
\def\P{{\IP}}

\newcommand{\mcl}[1]{\mathcal{#1}}

\renewcommand\section{\@startsection {section}{1}{\z@}%
{-3.5ex \@plus -1ex \@minus -.2ex}%
{1.3ex \@plus.2ex}%
{\center\small\sc\mathversion{bold}}}
\def\subsection#1{\@startsection {subsection}{2}{0pt}%
{-3.5ex \@plus -1ex \@minus -.2ex}%
{1ex \@plus.2ex}%
{\bf\mathversion{bold}}{#1}}

\def\subsubsection#1{\@startsection{subsubsection}{3}{0pt}%
{\medskipamount}%
{-10pt}%
{\normalsize\itshape}{\kern-2.2ex. #1.}}

\def\blfootnote{\xdef\@thefnmark{}\@footnotetext}

\makeatother

\begin{document}
\author{Xiao Fang$^1$, Song-Hao Liu$^2$, Xiaolin Wang$^1$}
\date{\it $^1$The Chinese University of Hong Kong, $^2$Dalian University of Technology \\[2ex]  }
\title{Mixing Time of Conditional Two Star Exponential Random Graphs}
\maketitle%%%%%%%%%%%%%%%%%%%%%%%%%%%%%%%%%%%%%%%%%%%%%%
%% Please use \tableofcontents for articles %%
%% with 50 pages and more                   %%
%%%%%%%%%%%%%%%%%%%%%%%%%%%%%%%%%%%%%%%%%%%%%%
%\tableofcontents
Two star exponential random graph models (ERGMs) are an interesting special case of both general ERGMs and mean-field Ising models. In this paper, we study two star ERGMs conditioning on the edge density $p$. We begin with an analytic characterization of the replica symmetric region, where the conditional model is close in cut distance to the Erd\H{o}s--R\'enyi random graph $G(n,p)$. We prove that this region is twice as large as the corresponding region for the unconditional model. To study refined properties of the conditional model, we then analyze the global Kawasaki algorithm for sampling from it. Within the replica symmetric region, and under the additional condition that $4\beta p(1-p)<0.5$ when $|p-1/2|\lesssim 0.4632$, where $\beta$ is the model parameter, we prove metastable fast mixing of the Kawasaki algorithm. As corollaries, we obtain a weak Poincar\'e inequality and use it to deduce concentration inequalities of optimal order for conditional subgraph counts.
The additional condition $4\beta p(1-p)<1/2$ if $|p-1/2|\lesssim 0.4632$ comes from our proof technique of contractive coupling. This bottleneck did not appear in previous studies applying the contractive coupling technique to unconditional ERGMs.
%The additional contraction condition is likely a proof artifact and identifies the present limitation of available coupling techniques.

\section{Introduction}
\label{sec:intro}
%\red{to change $p$ to $p$? (watch out for $W_p$ and $(p)^2$); to use $p\mathbf{1}$ for the constant graphon?}
Exponential random graph models (ERGMs), which specify network laws through selected graph statistics, are frequently used as parametric models in network analysis.
They were suggested for directed networks by \cite{holland1981exponential} and for undirected networks by \cite{frank1986markov}.
A general development of the models is presented in \cite{wasserman1994social}. 
In this paper, we consider undirected dense ERGMs, where the resulting graphs generally contain $\asymp n^2$ edges, with
$n$ representing the number of vertices.
We do not consider sparse ERGMs (models that have $\asymp n$ edges), which are harder to study than their dense counterparts. For important discoveries about sparse ERGMs, see, for example, \cite{chatterjeeDembo2016} and \cite{Nicholas2024}.
%(see also \cite{chakraborty2024,Nicholas2024} for more details). more reference?
Alternative model specifications are also studied in the literature. For example, \cite{snijders2006new} proposed curved exponential random graph models which include statistics such as geometrically weighted degree counts.

Although there is a large literature on (unconditional) ERGMs (see, e.g., \cite{bhamidi2011mixing,chatterjee2013estimating,reinert2019,mukherjee2023statistics,bianchi2024limit,fang2024,winstein2025,winstein2025quantitative}), conditional models given the observed number of edges are less well studied. 
It is of interest to study the conditional model due to its relevance in statistical hypothesis testing problems (\cite{bresler2018optimal}) and in higher-order phase transitions in statistical mechanics models (\cite{Bauerschmidt2025}).
%, we consider two star ERGMs with a fixed edge density $\tilde p\in [\delta, 1-\delta]$.

We focus on the two star ERGM defined as follows.  
Let $\mathcal{G}_n$ denote the space of all simple graphs\footnote{All graphs considered herein are undirected, without self-loops or multiple edges.} on $n$ labeled vertices. For a graph $G \in \mathcal{G}_n$, let $\mathcal{V}(G)$ and $\mathcal{E}(G)$ denote its vertex set and edge set, respectively, with $|\mathcal{V}(G)| = n$ and $e(G)=|\mathcal{E}(G)|$ denoting the number of edges. 
Let $\alpha\in \IR$ and $\beta\geq 0$ be two real parameters. The two star ERGM assigns to each $G\in \mathcal G_n$ the probability 
\ben{\label{eq:twostarmodel2}
\mu_{\alpha, \beta}(G)=\frac{1}{Z_{n}(\alpha, \beta)}\exp\left(2\alpha e(G)+\frac{2\beta t(G)}{n}\right),
}
where $t(G)$ denotes the number of two stars (the graph with three vertices and two edges connecting them) in $G$ and $Z_{n}(\alpha, \beta)$ is the normalizing constant. The general ERGM is defined by adding more scaled subgraph counts in the exponent in \cref{eq:twostarmodel2}; we refer to \cite{chatterjee2013estimating} for details.
%Let $e(G)$ denote the number of edges and let $t(G)$ denote the number of two stars in a graph $G$.  The model is defined by
%\begin{align}\label{eq:twostarmodel}
    %\mu_{\beta}(G)=\frac{1}{Z_{n}(\alpha,\beta)}\exp\left(2\alpha  e(G)+\frac{2\beta t(G)}{n}\right).
%\end{align}
Although our method may extend to general ERGMs, we have chosen to focus on this special case where we can provide more definite answers. 
The two star ERGMs showcase the typical result we expect for general ERGMs and entail many of the technical difficulties. Besides, it is also an interesting special case of mean-field Ising models. 

We begin by studying the macroscopic behavior of the conditional two star ERGM with given edge density $p$.  Using the result of \cite{chatterjee2011large}, \cite{dembo2018large} and \cite{kenyon2017asymptotics} gave a variational characterization of the conditional model as follows.
Let \(\mathcal W\) be the space of symmetric measurable functions from \([0,1]^2\) to  \([0,1]\), and let
\[
       {\mathcal W}_{p}
        =
        \left\{W\in\mathcal W:
        \int_{[0,1]^2}W(x,y)\,dx\,dy=p
        \right\}.
\]
These results show that the conditional model is close in cut distance to the set of graphon maximizers (see the precise definitions in \cref{sec:pre})
\ben{\label{eq:variational}
W_p^*:=\argsup_{W\in \mcl W_p} \mathcal F_\beta(W),\quad \mathcal F_\beta(W):=\beta t(W)-I_p(W),
}
where
\ben{\label{eq:tW}
t(W)=\int_0^1 \big[\int_0^1 W(x,y)\dd y \big]^2\dd x,
}
\[
\Ip(W)=
\frac12\int_{[0,1]^2}\Dkl(W(x,y)\|p)\dd x\dd y,
\]
and
\[
\Dkl(q\|p)
=
q\log\frac{q}{p}
+
(1-q)\log\frac{1-q}{1-p},
\]
with the usual convention \(0\log 0=0\). 

We say that the model is in the \emph{replica symmetric region} following the terminology of \cite{chatterjee2013estimating} if the variational problem \cref{eq:variational} is maximized at the constant function $p\mathbf{1}$, where $\mathbf{1}(x,y)\equiv 1$. This means that, at first order, the conditional model is close to the Erd\H{o}s--R\'enyi random graph $G(n,p)$, in which each edge is present independently with probability $p$.
\cite[Section~4.1]{zhu2017asymptotic} proved that there is a critical value $c_p$ such that when $\beta<c_p$, the maximizer of \cref{eq:variational} is unique and is a constant and when $\beta>c_p$, the variational problem has a maximizer that is not constant. However, they did not obtain the value of $c_p$ except when $p=1/2$. Our first main result is an analytic characterization of the replica symmetric region.

\begin{proposition}
\label{prop:1}
Let \(p\in(0,1)\). Let
\[
c_p
:=
\begin{cases}
\dfrac{\log\frac{p}{1-p}}{2p-1}, & p\ne \dfrac12, \\[1.2em]
2, & p=\dfrac12.
\end{cases}
\]
\begin{enumerate}
\item If \(\beta \le c_p\), then the constant graphon \(W\equiv p\)
is the unique maximizer of \(\mathcal F_\beta\) over \(\mathcal W_p\).

\item If \(\beta >c_p\), then the constant graphon \(W\equiv p\) is not a maximizer of \(\mathcal{F}_{\beta}\) over \(\calW_p\).
\end{enumerate}
\end{proposition}

The proof of \cref{prop:1} is given in \cref{sec:proof-main}. Our key observation is
\[
c_p=\inf_{\substack{q\in[0,1]\\ q\ne p}}
\frac{D(q\|p)}{(q-p)^2}.
\]
For \(W\in\mathcal W_p\), this yields
\[
\mathcal F_\beta(W)-\mathcal F_\beta(p\mathbf{1})
\le
(\beta-c_p)\int_0^1
\left(\int_0^1(W(x,y)-p)\,\dd y\right)^2
\dd x,
\]
which leads to optimality of the constant graphon and uniqueness when $\beta\le c_p$;
when $\beta>c_p$, a small two-block perturbation preserving the edge density
gives a strictly larger value than $\mathcal F_\beta(p\mathbf 1)$.

To contrast the above result, we compare it with the replica symmetric
region of the unconditional model.  We show in Appendix~\ref{others} that,
for a prescribed density \(p\ne1/2\), there exists a parameter
\(\alpha\) for which the constant graphon \(p\mathbf{1}\) is the unique global
maximizer of the unconditional variational problem (see \cref{eq:variationUnconditional} below) if and only if
\[
    \beta<\frac{c_p}{2}.
\]
At \(p=1/2\), the boundary value \(\beta=1=c_p/2\) also has the constant
graphon \(\frac{1}{2}\mathbf{1}\) as its unique maximizer.
The factor-of-two separation between the two thresholds has a structural
origin.  In the unconditional variational problem, the edge density is free
to fluctuate, so constant perturbations \(W\equiv q\), which change the
overall density from \(p\) to \(q\), are admissible.  Even after \(\alpha \)
is chosen so that \(W\equiv p\) is the candidate homogeneous phase, this
global-density mode $q$ may become energetically favorable once
\(\beta>c_p/2\).  Conditioning on the number of edges removes precisely
this unstable direction: every admissible perturbation \(U=W-p\) must have
zero average, and the two star gain can arise only through fluctuations of
the degree profile.  The corresponding entropy cost is strong enough to
stabilize the constant graphon all the way up to \(\beta=c_p\).

Consequently, in the intermediate region
\[
    \frac{c_p}{2}<\beta\leq c_p,
\]
the conditional model remains replica symmetric, whereas no choice of \(\alpha \) makes the corresponding
unconditional model concentrate around the graphon \(p\mathbf{1}\).  Thus
conditioning does not merely select a fixed-density sector of the
unconditional model; it genuinely changes the phase diagram and creates a
strictly larger homogeneous phase.  This distinction also suggests that
methods that pass through the unconditional model cannot access the full
replica symmetric region of the conditional model.  See
\cref{fig:1}.

%Note that the region $\beta<c_p$ is larger than the counterpart in the unconditional model where given $p$, we can find $\beta$ such that the constant graphon $p$ is the unique global maximizer of the unconditional ERGM model (the area below the green line in \cref{fig:1}). This shows a different phase transition curse between the conditional and unconditional model.

\begin{figure}[ht]
    \centering
    \includegraphics[width=0.5\textwidth]{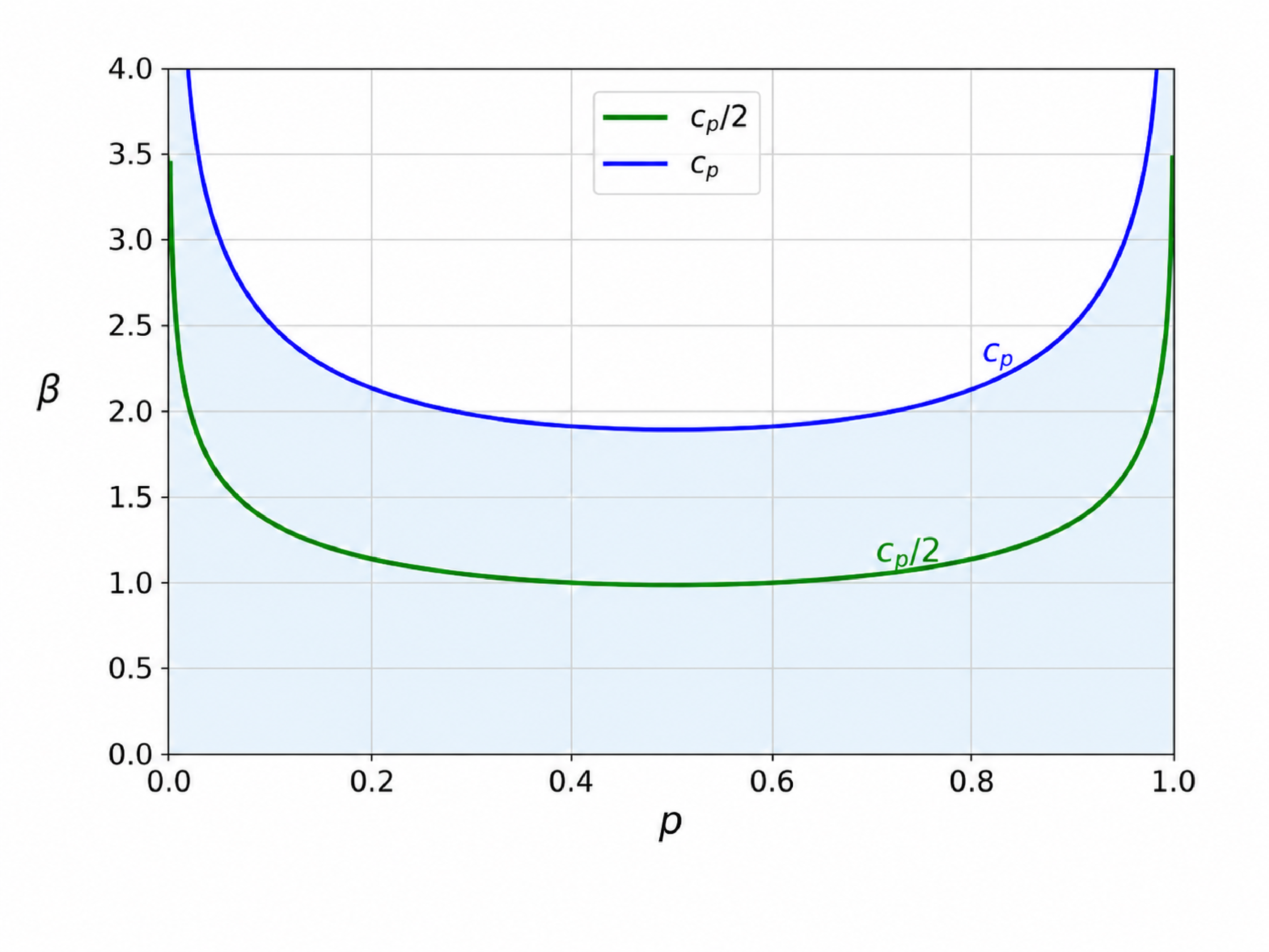}
    \caption{The replica symmetric region of the conditional model (below the blue line) is twice as large as the corresponding region for the unconditional model (below the green line).}
    \label{fig:1}
\end{figure}

As an illustration of the previous remark, \cite{fang2025conditional} obtained a conditional central limit theorem (CLT), which describes Gaussian fluctuations after normalization, for two star counts given the number of edges in general ERGMs.  Their method, however, relies on properties of the unconditional model and works at best up to the green line in \cref{fig:1}.  We instead analyze the conditional model directly.  Although we do not pursue this direction here, the results developed below may help extend the conditional CLT to a larger parameter region.

Large-deviation results describe the first-order behavior of the model, but they do not determine its microscopic law; indeed, the relevant total-variation distance may tend to one as the number of vertices grows.  For the unconditional model, microscopic properties such as concentration inequalities and CLTs can be obtained through Glauber dynamics.  Under a fixed-edge constraint, single-edge Glauber updates are unavailable.  A natural sampler is therefore the global Kawasaki dynamics, which proposes deleting one present edge and adding one absent edge, followed by a heat-bath update conditional on the remaining edges (see \cref{sec:pre}).  Our second main result proves metastable fast mixing of this dynamics in a subregion of the replica symmetric phase.

\begin{theorem}[\cref{coupling}, informal version]\label{thm:1informal}
Consider the two star ERGM \cref{eq:twostarmodel2} conditioned to have edge density $p$.  Assume $0\leq \beta<c_p$ and also $4\beta p(1-p)<1/2$ when $|p-1/2|\lesssim 0.4632$.  From every initial state in a regular set of overwhelming stationary probability, the Kawasaki dynamics mixes in $O(n^2\log n)$ steps, up to an exponentially small error.
\end{theorem}

The reason why it is necessary to study metastable mixing is that the Kawasaki algorithm may be trapped at certain points for an exponentially long time. See Appendix~\ref{sec:GPT1}.
However, these points form a set with exponentially small probability (see \cref{coupling}) and metastable mixing suffices for all purposes. This is similar to the result of \cite{bresler2024metastable} for the unconditional model. See \cite{winstein2025quantitative}, who used metastable mixing to prove the CLT for the unconditional ERGM in the supercritical region.

The proof of \cref{thm:1informal} is by a basin-restricted contraction argument. This approach has been successfully used by \cite{bhamidi2011mixing} and \cite{bresler2024metastable} to prove optimal results for unconditional ERGMs. However, additional technicality and bottleneck arise for the conditional model as we explain below. We first
identify the regular basin \(\Gamma_{p,\varepsilon}\), on which every local
edge density is uniformly close to \(p\).  Combining the graphon-scale
concentration estimate with the cavity and switching argument, we show in
\cref{le-1} that the conditional measure assigns all but
exponentially small mass to \(\Gamma_{p,\varepsilon}\).
The second ingredient is dynamical stability of this basin.  Near either
boundary of \(\Gamma_{p,\varepsilon}\), the extremal local edge densities
have a uniform drift pointing back toward \(p\).  \cref{co-1} shows that a chain started in a smaller core remains in the
regular basin for exponentially long times and is repeatedly driven away
from its boundary.
The third ingredient is a local contraction mechanism.
Using a coupling argument, we show in \cref{le-5} the one-step contraction in Hamming distance of Kawasaki paths starting from two different states inside
\(\Gamma_{p,\varepsilon}\) under the additional assumption that $4\beta p(1-p)<1/2$ when $|p-1/2|\lesssim 0.4632$. 
Finally, the stopped path-coupling estimate
\cref{le-21} balances this contraction against the exponentially unlikely
escape event to give the claimed
\(O\!\left(n^2\log(n^2/\delta)\right)\) metastable mixing bound.

The additional condition $4\beta p (1-p)<0.5$ in \cref{thm:1informal} corresponds to the dotted red line in \cref{fig:2} (which is optimal for those $p$ sufficiently close to 0 or 1) and comes from our proof technique of contractive coupling (see \cref{le-5}). 
For \(p\) sufficiently close to \(0\) or \(1\), this condition is implied
by the replica symmetric condition and therefore does not reduce the
admissible region.  
Alternatively, one may try to use the method
in \cite{kuchukova2025fast}. However, their method relies on a local CLT for edge counts for the
unconditional model and seems at best works up to the green line in \cref{fig:1}. 
%Moreover, their mixing speed is not optimal and is not sufficient to obtain refined properties of the conditional model (as we explain in the next paragraph). 
Another alternative method is mixing through spectral gap (cf. \cite{bauerschmidt2019very}). In particular, the conditional version of the spectral gap method by \cite{Bauerschmidt2025} suggests a threshold of $\beta<1$, which is still a subset of the replica symmetric region of the unconditional model. It would be interesting to further develop their method to prove metastable fast mixing throughout the replica symmetric region of the conditional model.
We also remark that both of the above two alternative methods
do not directly apply  to general conditional ERGMs with higher-order interactions.

%\red{By contrast, the spectral condition of \cite[Theorem~1.3]{Bauerschmidt2025}, under the exact Ising reparametrization in \cref{rem:bau}, is uniform for every fixed \(\beta<1\).  The two criteria are not nested: the spectral condition covers a larger range near \(p=1/2\), whereas the present coupling condition permits \(\beta>1\) when \(p\) is sufficiently close to \(0\) or \(1\).  }

\begin{figure}[ht]
    \centering
    \includegraphics[width=0.5\textwidth]{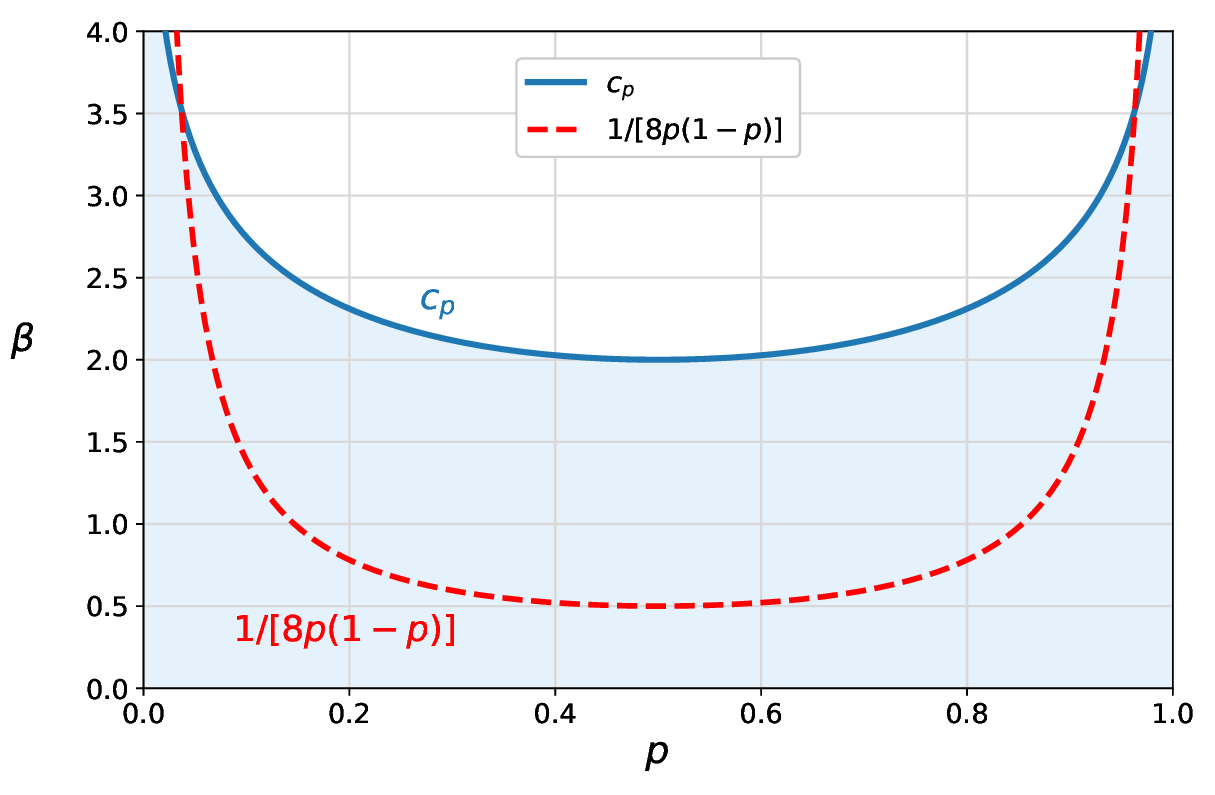}
    \caption{We conjecture that fast metastable mixing holds up to the blue line. However, we have only proved it up to the dotted red line using the contractive coupling argument. The intersection points of the two curves are $p\approx 0.5\pm 0.4632$.}
    \label{fig:2}
\end{figure}

To illustrate how metastable fast mixing yields refined properties of the conditional ERGM, we establish a weak Poincar\'e inequality with an exponentially small remainder (\cref{cor:full-defective-poincare}) and derive higher-order concentration inequalities (\cref{thm:higher-order-concentration-log}).  In particular, \cref{cor:triangle-variance} shows that, under the assumptions of \cref{thm:1informal}, the variance of the triangle count is $O(n^3)$, compared with order $O(n^4)$ in the unconditional model.  The optimal mixing speed is essential for this improvement.

%\red{This work leaves several natural questions.  First, can metastable fast mixing, and the resulting weak Poincar\'e and concentration inequalities, be extended to the whole replica symmetric region?
%, including \(4\beta p(1-p)\ge1/2\)?  For the two star model, the spectral condition of \cite{Bauerschmidt2025}, together with the global-exchange Dirichlet-form comparison in \cref{rem:bau}, already covers the range \(\beta<1\).  Extending that approach beyond the spectral range would require verifying the covariance condition introduced in \cite{Bauerschmidt2025}.  This route, however, appears specific to the Ising representation.} 
This work leaves several natural questions.  First, can metastable fast
mixing, and the resulting weak Poincar\'e and concentration inequalities,
be extended to the whole replica symmetric region?
%, including \(4\beta p(1-p)\ge 1/2\)?  The present restriction comes from the contractive coupling and is likely technical; an alternative may be to adapt the methods of \cite{Bauerschmidt2025}, although these appear specific to the two star Ising representation.  
Second, beyond the blue curve in \cref{fig:1} one
should analyze the nonconstant phase, including
mixing within individual basins and transitions between them.  Third, it
would be interesting to prove quantitative conditional CLTs directly under
the fixed-edge measure, potentially beyond the range accessible through the
unconditional model.  Finally, one would like to extend the theory to
general conditional ERGMs, where new notions of local regularity and new
coupling or functional-inequality arguments will be needed.  We leave these
questions for future investigation.

%This work leaves open the following interesting questions: (a) studying refined properties of the conditional ERGM in the replica symmetric region when $4\beta p(1-p)\geq 0.5$ (the quesiton mark in \cref{fig:2}); It might be possible to extend the techniques in \cite{Bauerschmidt2025} to prove weak Poincar\'e inequalities for conditional Ising models and address this question, although this technique will be restricted to the two ERGM;  (b) refined properties of the conditional model beyond the blue curve (bipartite graphon); (c) CLTs; (d) general conditional ERGMs. We leave them for future investigations.

In \cref{sec:pre}, we provide preparatory materials and give the formal statement of our main results. In \cref{sec:proof-main}, we prove our main results, and defer the proofs of some lemmas to \cref{sec:confinement} and \cref{sec:path-coupling}. Some additional technical arguments are given in the appendices.

\paragraph{Acknowledgements.} 
Fang X. was partially supported by Hong Kong RGC GRF 14304822, 14303423, 14302124, 14304125 and a CUHK direct grant. Liu S.H. was partially supported by the Fundamental Research Funds for the Central Universities DUT25RC(3)133.

\paragraph{Statement on artificial-intelligence assistance.} 
The main results and the ideas underlying their proofs were developed without any AI assistance. GPT 5.5 Pro and GPT 5.6 Pro assisted in carrying out the computations for the cavity argument in Appendix~\ref{sec:degree-to-local}, finding the counterexample in Appendix~\ref{sec:GPT1}, and improving the exposition of the manuscript.

\section{Preliminaries and main results}\label{sec:pre}
Recall the definition of the two star ERGM in \cref{eq:twostarmodel2}.
Building on \cite{chatterjee2011large}, \cite{chatterjee2013estimating} established a large deviation principle for general ERGMs on the space of graphons. 
To prepare for the statement of their results,
recall that \(\mathcal W\) denotes the space of symmetric measurable functions \(W:[0,1]^2\to[0,1]\). A finite graph is identified with a graphon through its adjacency matrix. 
 
%For every fixed \(H\),
%\[
        %t(H,x)=t(H,W_x)+O_H(n^{-1}).
%\]
The cut norm of an integrable function \(U:[0,1]^2\to\mathbb R\) is
\[
        \|U\|_\square
        =
        \sup_{S,T\subseteq[0,1]}
        \left|\int_{S\times T}U(x,y)\,dx\,dy\right|,
\]
where the supremum is over measurable sets \(S,T\).  The cut distance between graphons is
\[
        d_\square(U,W)=\inf_\sigma \|U-W^\sigma\|_\square,
\]
where the infimum is over measure-preserving bijections \(\sigma: [0,1]\to [0,1]\) and \(W^\sigma(x,y)=W(\sigma(x),\sigma(y))\).

For a finite graph \(H\) with vertices $\{1,\dots, v(H)\}$ and edge set $\mcl{E}(H)$, define its graphon homomorphism density to be
\[
        t(H,W)=
        \int_{[0,1]^{v(H)}}
        \prod_{\{i,j\}\in \mcl{E}(H)} W(x_i,x_j)\,
        dx_1\cdots dx_{v(H)} .
\]
When $H$ is a two star, $t(H, W)$ reduces to $t(W)$ in \cref{eq:tW}.

Specializing the large-deviation result of \cite{chatterjee2013estimating} to two star ERGMs shows that the model \cref{eq:twostarmodel2} is close in cut distance to the set of graphons $W$ that maximize the functional
\ben{\label{eq:variationUnconditional}
\alpha e(W)+\beta t(W)-I(W),
}
where 
\[
d_W(x):=\int_0^1 W(x,y)\dd y,
\quad  e(W)=\int_0^1 d_W(x)dx   ,\quad
t(W):=\int_0^1 d_W(x)^2\dd x,
\]
and
\ben{\label{eq:I(W)}
        I(W)=
        \frac12\int_{[0,1]^2}
        \bigl[W\log W+(1-W)\log(1-W)\bigr]dxdy,
}
with the convention \(0\log0=0\). It is also known that the maximizer of \cref{eq:variationUnconditional} is unique and a constant function almost surely on the two-dimensional space of parameters $\alpha\in \IR,\beta\geq 0$ with respect to the Lebesgue measure.

Now we turn to ERGMs conditioning on the number of edges $k$ such that $|k/{n\choose 2}-p|\to 0$.
%DELETE? We focus on the ERGMs conditioned on a fixed edges. Let $\beta_{2},\dots, \beta_{m}$ be positive parameters and let $ p\in (0,1)$. Let $\E_{(\beta_{1},\beta_{2},\dots, \beta_{m})}Y_{ij}=p$, where $Y_{ij}$ is the edge indicator of a random graph with vertex set $\{1,\dots, n\}$ and the expectation is with respect to the ERGM \eqref{eq:ERGM1} with parameters $(\beta_{1},\beta_{2},\dots, \beta_{m})$.
The large deviation principle for conditional ERGMs was given by \cite{dembo2018large} and \cite{kenyon2017asymptotics}. For two star ERGMs, the conditional model is close in cut distance to the set of maximizers $W_p^*$ in \cref{eq:variational}.
Note that the variational functional \cref{eq:variational} does not depend on $\alpha$. This is because the conditional model is independent of $\alpha$ (see \cite{fang2025conditional} for example). \cref{prop:1} characterizes the replica symmetric region of \cref{eq:variational}.

In what follows, we exclude the critical case $\beta=c_p$ and simply refer
to $\{(p,\beta): 0<p<1,\ 0\leq\beta<c_p\}$ as the replica symmetric
region. In many statistical physics models, the behavior at criticality
may differ from that in the subcritical region; see, for example,
\cite{chatterjee2011nonnormal} and \cite{mukherjee2023statistics}.

We consider a sequence of two star ERGMs indexed by $n\geq 1$, with parameters $\alpha=\alpha_n$ and $\beta=\beta_n$ that may depend on $n$. Let $N=n(n-1)/2$, let $k=k_n$ be a sequence of positive integers, and set $p=p_n=k/N$. Let $\mu^k_\beta=\mu^k_{\beta,n}$ be the conditional model given exactly $k$ edges. We fix a compact subset $K$ of the replica symmetric region $\{(p,\beta): 0<p<1, 0\leq \beta<c_p\}$ and suppose that $(p,\beta)\in K$ for all sufficiently large $n$.

Following the proofs of \cite[Theorem~3.2]{chatterjee2013estimating} and \cite[Theorem~2]{kenyon2017asymptotics}, we show in the following proposition that, in the replica symmetric region, the conditional measure $\mu_\beta^k$ is exponentially concentrated in every fixed cut-metric neighborhood of $p\mathbf 1$. In the following, we say a triple $\{n\geq 1, p\in (0,1), k\in \{0,\dots, N\}\}$ is \emph{admissible} if $k=Np$.

\begin{proposition}
\label{theo of appro}
For every fixed \(\eta>0\), there exist constants \(n_{K,\eta},c_{K,\eta},C_{K,\eta}>0\), depending only on \(\eta\) and \(K\), such that, uniformly for all \((p,\beta)\in K\), all admissible \(k=Np\in\mathbb Z\), and all \(n\geq n_{K,\eta}\), we have
\ben{\label{eq:prop21-1}
        \mu_\beta^k\left(d_\square(W_X,p\mathbf 1)>\eta\right)
        \le C_{K,\eta} e^{-c_{K,\eta} n^2},
}
where $X\sim \mu^k_\beta$ and $W_X$ denotes the graphon defined through the adjacency matrix of $X$.    
\end{proposition}

%To study refined properties of the conditional model, we consider the global Kawasaki algorithm to sample from it. \red{To add a short description of the Kawasaki algorithm.}

To study refined properties of the conditional model, we consider the global
Kawasaki dynamics on the fixed-edge slice
\[
    \Omega_k
    :=
    \left\{
        x\in\{0,1\}^{\mathcal I}:
        \sum_{e\in\mathcal I}x_e=k
    \right\},
\]
where $\mathcal I:=\binom{[n]}2$ denotes the set of $N$ potential edges in a graph on $n$ vertices labeled by $[n]:=\{1,\dots, n\}$.
For \(X\in\Omega_k\), define the finite two star count and Hamiltonian by
\ben{\label{eq:Hamilton}
    S_2(X):=\sum_{v=1}^n\binom{\deg_X(v)}2,
    \qquad
    \mathcal H_\beta(X):=\frac{2\beta}{n}S_2(X),
}
where $\deg_X(v)$ denotes the degree of the vertex $v$ in the graph $X$.
The conditional measure is therefore
\[
    \mu_\beta^k(X)
    =
    \frac{e^{\mathcal H_\beta(X)}}
         {\sum_{Y\in\Omega_k}e^{\mathcal H_\beta(Y)}}.
\]
For the standard graphon representation (recall \cref{eq:tW}),
\[
        t(W_X)=n^{-3}\sum_v\deg_X(v)^2,
        \qquad
        \mathcal H_\beta(X)=n^2\beta t(W_X)-\frac{2\beta k}{n}.
\]
Thus, on \(\Omega_k\), the conditional measure equivalently has weights
\(\exp\{n^2\beta t(W_X)\}\).
For \(x\in\Omega_k\), write
\[
    E_1(x):=\{e\in\mathcal I:x_e=1\},
    \qquad
    E_0(x):=\{e\in\mathcal I:x_e=0\}.
\]
For \(x,y\in\Omega_k\), define the Kawasaki distance
\ben{\label{eq:Kawadist}
        \rho(x,y)
        :=|E_1(x)\setminus E_1(y)|
        =|E_1(y)\setminus E_1(x)|.
}
A single exchange can remove at most one mismatched occupied edge, while
successively exchanging the mismatched occupied and unoccupied edges gives a
path of this length; thus \(\rho\) is the graph distance on \(\Omega_k\).
Given the current state \(X_m=x\), choose an occupied edge-position
\(s\in E_1(x)\) and an unoccupied edge-position \(t\in E_0(x)\),
independently and uniformly.  Let \(x^{s\leftrightarrow t}\) be the graph
obtained by deleting \(s\) and adding \(t\).  Conditional on this proposal,
the chain moves from \(x\) to \(y=x^{s\leftrightarrow t}\) with the
two-state heat-bath probability
\[
    \alpha(x,y)
    :=
    \frac{\mu_\beta^k(y)}
         {\mu_\beta^k(x)+\mu_\beta^k(y)},
\]
and otherwise remains at \(x\).  Thus the transition kernel is
\ben{\label{eq:KawaTran}
    P(x,y)
    =
    \frac{1}{k(N-k)}
    \frac{\mu_\beta^k(y)}
         {\mu_\beta^k(x)+\mu_\beta^k(y)},
    \qquad
    y=x^{s\leftrightarrow t}\ne x,
}
where \(x_s=1\), \(x_t=0\), and
\[
    P(x,x)
    =
    1-\sum_{\substack{y\in\Omega_k\\y\ne x}}P(x,y).
\]
The proposal is called global because \(s\) and \(t\) may be arbitrary
occupied and unoccupied edge-positions, rather than neighboring sites in the
line graph.  Moreover,
\[
    \mu_\beta^k(x)P(x,y)
    =
    \frac{1}{k(N-k)}
    \frac{\mu_\beta^k(x)\mu_\beta^k(y)}
         {\mu_\beta^k(x)+\mu_\beta^k(y)}
    =
    \mu_\beta^k(y)P(y,x),
\]
so the chain is reversible with stationary distribution \(\mu_\beta^k\).
Since every legal exchange has positive probability, the chain is
irreducible on \(\Omega_k\); the positive holding probability also makes it
aperiodic.

For later use, write
\[
        \square_\eta(p)=
        \{x\in\mathcal G_n: d_\square(W_x,p\mathbf 1)\le\eta\},
        \qquad
        \square_\eta^k(p)=\square_\eta(p)\cap\Omega_k .
\]
For two edge-positions \(e,f\in\mathcal I\), write \(f\sim e\) if
\(f\ne e\) and \(f\cap e\ne\varnothing\) (that is, they share exactly one vertex).  For \(X\in\Omega_k\), define
\ben{\label{eq:rex}
        d(e,X):=|\{f\in\mathcal I:X_f=1,\ f\sim e\}|,
        \qquad
        r(e,X):=\frac{d(e,X)}{2n}.
}
We write
\ben{\label{eq:Gamma}
        \Gamma_{p,\varepsilon}
        =
        \{X\in\Omega_k:r(e,X)\in [p-\varepsilon,p+\varepsilon]
        \ \text{for every edge-position }e\in\mathcal I\},
}
and
\[
        r_{\max}(X):=\max_{e\in\mathcal I} r(e,X),
        \qquad
        r_{\min}(X):=\min_{e\in\mathcal I} r(e,X).
\]

The variational threshold in \cref{prop:1} describes the
replica symmetric region of the fixed edge two star model.  The dynamical
arguments below require, in addition, a quantitative contraction margin for
the Kawasaki coupling.  We therefore introduce the region
\[
    \mathscr R_{\mathrm{RS}}
    :=
    \{(p,\beta):0<p<1,\ 0\leq\beta<c_p\},
\]
and
\[
    \mathscr R_{\mathrm{Kaw}}
    :=
    \left\{
    (p,\beta)\in\mathscr R_{\mathrm{RS}}:
    0\leq 4\beta p(1-p)<\frac12
    \right\}.
\]
%The blue curve in Figure~\ref{fig:2} corresponds to the boundary \(\beta=c_p\), while the dotted red curve in Figure~\ref{fig:2} corresponds to \(4\beta p(1-p)=1/2\).  
The mixing and concentration results proved
below are not asserted on the whole replica symmetric region.  They are proved
inside the smaller region \(\mathscr R_{\mathrm{Kaw}}\).

Throughout the remaining part of the paper, we fix a compact set
$B\subset \mathscr R_{\mathrm{Kaw}} .$ In particular, we assume that there exists a constant \(\kappa_B>0\) such that, for every
\((p,\beta)\in B\),
\ben{\label{eq:kappaB}
    \kappa_B\le p\le 1-\kappa_B,
    \qquad
    \beta\le c_p-\kappa_B,
    \qquad
    \frac12-4\beta p(1-p)\ge \kappa_B .
}
All constants in the estimates below may depend on \(B\), but not on the particular choice of \((p,\beta)\in B\).  When the model parameter $\beta=\beta_n$ and the edge density $p=p_n$ are
allowed to vary with \(n\), we assume that \(k=k_n=Np_n\in\mathbb Z\)
and \((p_n,\beta_n)\in B\).  To keep the notation light, we will drop the subscript $n$.  
%In particular, Assumption~\ref{ass:conditional-RS} holds by Proposition~\ref{threshold}, and there is a uniform positive margin
%\[
%    \frac12-4\beta p(1-p)\ge \kappa_B .
%\]

As will be shown in Appendix~\ref{sec:GPT1}, it is in general not possible to prove fast mixing with arbitrary initial state. Therefore, following \cite{bresler2024metastable}, we consider metastable mixing.

Let \(P\) denote the global Kawasaki kernel defined in \cref{eq:KawaTran}, whose stationary
distribution is \(\mu_\beta^k\).  
For any \(\varepsilon>0\), set
\(
\pi_\varepsilon=\mu_\beta^k(\,\cdot\mid\Gamma_{p,\varepsilon}).
\) The restricted kernel \(P_\varepsilon\) is defined as the hard-wall restriction of the full
kernel to \(\Gamma_{p,\varepsilon}\): for \(x,y\in\Gamma_{p,\varepsilon}\), \(x\ne y\),
\[
        P_\varepsilon(x,y)=P(x,y),
\]
and
\[
        P_\varepsilon(x,x)
        =
        P(x,x)+P(x,\Gamma_{p,\varepsilon}^c)
        =
        1-
        \sum_{\substack{y\in\Gamma_{p,\varepsilon}\\ y\ne x}}P(x,y).
\]
Equivalently, the chain uses the same Kawasaki proposal and the same
acceptance probability as the full chain, but any move which would leave
\(\Gamma_\varepsilon\) is turned into a holding move.  Since \(P\) is reversible
with respect to \(\mu_\beta^k\), \(P_\varepsilon\) is reversible with respect to
\(\pi_\varepsilon\).

\begin{theorem}\label{coupling}
Fix a compact set \(B\subset\mathscr R_{\mathrm{Kaw}}\).  There exists
\(\varepsilon_0=\varepsilon_0(B)>0\), depending only on \(B\), such that the following holds
for all \((p,\beta)\in B\) and all \(0<\varepsilon\le\varepsilon_0\).  There are constants
\(\alpha_{B,\varepsilon}>0\), \(c_{B,\varepsilon}>0\), and
\(C_{B,\varepsilon}<\infty\), and \(n_0=n_0(B,\varepsilon)<\infty\), such that,
for every \(n\ge n_0\),
\ben{\label{eq:regularregion}
\mu_\beta^k(\Gamma_{p,\varepsilon/4}^c)\leq C_{B,\varepsilon} e^{-c_{B,\varepsilon} n}
}
and for every
\(x\in\Gamma_{p,\varepsilon/4}\), we have

\smallskip
\noindent
\textup{(i)} Let \((X_t)_{t\ge0}\) be the full Kawasaki chain with kernel
\(P\), started from \(X_0=x\), and let
\((\widetilde X_t)_{t\ge0}\) be the restricted Kawasaki chain with kernel
\(P_\varepsilon\), started from \(\widetilde X_0=x\).  For every
\(T\le e^{\alpha_{B,\varepsilon} n}\), the two chains can be coupled so that
\[
        \mathbb P\left(
        (X_0,\ldots,X_T)=(\widetilde X_0,\ldots,\widetilde X_T)
        \right)
        \ge
        1-C_{B,\varepsilon}e^{-c_{B,\varepsilon}n}.
\]

\smallskip
\noindent
\textup{(ii)} Whenever \(e^{-c_{B,\varepsilon}n}\le\delta\le1\), the
restricted chain satisfies
\[
        d_{\mathrm{TV}}\bigl(xP_\varepsilon^T,
        \mu_\beta^k(\cdot\mid\Gamma_{p,\varepsilon})\bigr)
        \le \delta
\]
for
\[
        T=\lceil C_{B,\varepsilon}n^2\log(n^2/\delta)\rceil.
\]
For the same value of \(T\), the full Kawasaki chain satisfies
\[
        d_{\mathrm{TV}}\bigl(xP^T,\mu_\beta^k\bigr)
        \le \delta .
\]
\end{theorem}

Just as (metastable) fast mixing of the Glauber dynamics for the unconditional ERGM implies a (weak) Poincar\'e inequality
(\cite{ganguly2019}, \cite{bresler2024metastable}), metastable mixing of the Kawasaki algorithm obtained in \cref{coupling} can be used to prove the weak Poincar\'e inequality for the restricted kernel (\cref{lem:defective-poincare}) as well as for the full Kawasaki kernel (\cref{cor:full-defective-poincare}).

\begin{theorem}
\label{lem:defective-poincare}
Fix \(0<\varepsilon\le\varepsilon_0(B)\) as in \cref{coupling}.  
Set
\(\pi_\varepsilon:=\mu_\beta^k(\,\cdot\mid \Gamma_{p,\varepsilon})\).  
Let \(\mathcal E_\varepsilon\) be
the Dirichlet form of the restricted Kawasaki kernel \(P_\varepsilon\):
\[
\mathcal E_\varepsilon(f,f)
=
\frac12
\sum_{x,y\in \Gamma_{p, \varepsilon}}
\pi_{\varepsilon}(x)P_\varepsilon(x,y)\bigl(f(y)-f(x)\bigr)^2 .
\]
Then there exist constants \(C_{B,\varepsilon}<\infty\) and
\(c_{B,\varepsilon}>0\), and \(n_0=n_0(B,\varepsilon)<\infty\), depending only
on \(B\) and \(\varepsilon\), such that, uniformly over
\((p,\beta)\in B\), for every  \(n\ge n_0\) and every
bounded \(f:\Gamma_\varepsilon\to\mathbbm R\),
\begin{equation}\label{eq:no-log-defective-poincare}
    \Var_{\pi_\varepsilon}(f)
    \le
    C_{B,\varepsilon}n^2\mathcal E_\varepsilon(f,f)
    +
    C_{B,\varepsilon}e^{-c_{B,\varepsilon}n}\norm{f}_\infty^2.
\end{equation}
\end{theorem}

For a bounded function $f$ on \(\Omega_k\), define the Dirichlet form of the Kawasaki kernel:
\[
        \mathcal E_P(f,f)
        =\frac12\sum_{x,y\in\Omega_k}
        \mu_\beta^k(x)P(x,y)\bigl(f(y)-f(x)\bigr)^2.
\]

\begin{corollary}\label{cor:full-defective-poincare}
Under the assumptions of \cref{coupling}, there exist constants
\(C_B<\infty\), \(c_B>0\), and \(n_0=n_0(B)<\infty\) such that,
uniformly over \((p,\beta)\in B\), for every \(n\ge n_0\) and every
bounded \(f:\Omega_k\to\mathbbm R\),
\[
        \Var_{\mu_\beta^k}(f)
        \le C_Bn^2\mathcal E_P(f,f)
        +C_Be^{-c_Bn}\norm{f}_\infty^2.
\]
\end{corollary}

It is known that (weak) Poincar\'e inequality for the Glauber dynamics can be used to deduce concentration inequalities for unconditional ERGMs; see \cite{Sam2020}, \cite{fang2025conditional}, and \cite{winstein2025}. 
It turns out that the weak Poincar\'e inequality for the Kawasaki algorithm yields concentration
inequalities for the conditional two star ERGM. 

We use the following Hoeffding-type decomposition under weak dependence as in \cite{Sam2020} and \cite{fang2025conditional}.
For \(e\in\mathcal I\), write
\(
        \xi_e=X_e-p .
\)
Note that $\E_{\mu^k_\beta}\xi_e=0$.
For an ordered tuple \(I=(i_1,\ldots,i_m)\in\mathcal I^m\), let
\(\mathcal P(I)\) denote the set of partitions of the set of positions
\([m]=\{1,\ldots,m\}\).  If \(J\subseteq [m]\), set
\(
        m_J(I)
        =
        \E_{\mu_\beta^k}
        \prod_{r\in J}\xi_{i_r}.
\)
For a partition \(P\in\mathcal P(I)\), let
\(
        M(P)=|\{J\in P: |J|\ge2\}|.
\)
If \(J=\{r\}\) is a singleton block, we write
\(
        \xi_{i_J}=\xi_{i_r}.
\)
Define \(g_\varnothing:=1\).  For \(I=(i_1,\ldots,i_m)\), \(m\ge1\), define
\begin{align}\label{hoff}
    g_I(X)
    =&
    \sum_{P\in\mathcal P(I)}
    (-1)^{M(P)}M(P)!
    \prod_{\substack{J\in P\\ |J|=1}}
    \xi_{i_J}
    \prod_{\substack{J\in P\\ |J|\ge2}}
    m_J(I).
\end{align}
It can be verified as for the unconditional ERGM case in \cite{Sam2020} and \cite{fang2025conditional} that
\[
        \E_{\mu_\beta^k} g_I(X)=0,
        \qquad I\ne\varnothing ,
\]
\ben{\label{eq:gRecursion}
        g_I
        =
        \prod_{r=1}^m \xi_{i_r}
        -
        \sum_{\substack{B\subset[m]\\ |B|\ge 2}}
        m_B(I)\,g_{I\setminus B}.
}
Moreover, if \(I=(i_1,\ldots,i_m)\) has pairwise distinct
coordinates, then, for every collection of distinct positions
\(r_1,\ldots,r_j\in[m]\),
\begin{align}\label{eq:appell-properties}
      D_{i_{r_1}}\cdots D_{i_{r_j}}g_I
        =
        g_{I\setminus\{r_1,\ldots,r_j\}}, 
\end{align}
where 
$D_e F(x):=F(x^{e\leftarrow 1})-F(x^{e\leftarrow 0})$, $x^{e\leftarrow a}$ is obtained by replacing $x_e$ by $a$, and 
\(I\setminus\{r_1,\ldots,r_j\}\) denotes the ordered tuple obtained from
\(I\) by deleting the coordinates in the indicated positions. 

For \(d\ge1\), let \(A\) be a \(d\)-tensor indexed by $\mathcal I^d$ and set
\(
        \norm{A}_2
        =\left(\sum_{I\in\mathcal I^d}A_I^2\right)^{1/2}.
\)
Moreover, assume that \(A_I=0\) whenever two coordinates of \(I\) coincide.  
Define
\ben{\label{eq:fdA}
f_{d,A}(X)
    = \sum_{I\in\mathcal I^d} A_I g_I(X).
}
Since all moments $m_J(I)$ in
\eqref{hoff} are bounded by $1$,
\begin{equation}\label{eq:supnorm-basic}
    \norm{f_{d,A}}_\infty
    \le
    C_dN^{d/2}\norm{A}_2
    \le
    C_dn^d\norm{A}_2,
\end{equation}
where $C_d$ is a constant depending only on $d$.
The following result provides concentration inequalities for $f_{d,A}(X)$ under $\mu^k_\beta$.

\begin{theorem}\label{thm:higher-order-concentration-log}
Assume the hypotheses of Theorem~\ref{coupling}, and set \(\mu:=\mu_\beta^k\). Let \(f_{d,A}\) be defined by
\eqref{eq:fdA}.  The following estimates hold.

\smallskip
\noindent
\textup{(i)} For every fixed $d\ge1$ and $M>0$, there are constants
$C_{d,M}=C_{d,M}(B)$, $c_{d,M}=c_{d,M}(B)>0$ and
$n_0=n_0(d,M,B)<\infty$, depending only on
$d,M$, and $B$, such that, uniformly over
$(p,\beta)\in B$, for every  \(n\ge n_0\) and
\(
    2\le q\le c_{d,M}n/ \log n,
\)
one has
\begin{equation}\label{eq:moment-bound}
    \norm{f_{d,A}}_{L^q(\mu)}
    \le
    C_{d,M} q^d\norm{A}_2
    +
    C_{d,M}n^{-M}\norm{f_{d,A}}_\infty.
\end{equation}
%In particular, for every fixed $q\ge1$,
%\[
%    \norm{f_{d,A}}_{L^q(\mu)}
%    \le
%    C_{d,q,M}\norm{A}_2
%    +
%    C_{d,q,M}n^{-M}\norm{f_{d,A}}_\infty.
%\]

\smallskip
\noindent
\textup{(ii)} Consequently, for every $t\ge0$,
\begin{equation}\label{eq:tail-bound}
    \mu\left(
    \abs{f_{d,A}}
    \ge
    t+C_{d,M}n^{-M}\norm{f_{d,A}}_\infty
    \right)
    \le
    2\exp\left\{
    -c_{d,M}
    \min\left[
    \left(\frac{t}{\norm{A}_2}\right)^{1/d},
    \frac{n}{\log n}
    \right]
    \right\}.
\end{equation}
When $\norm{A}_2=0$, the right-hand side is interpreted as $0$ for $t>0$.
\end{theorem}

Using the concentration inequality for $f_{d,A}(X)$ above, we can deduce optimal concentration inequalities for subgraph counts. We use triangle counts as an illustration in the next result.

\begin{corollary}\label{cor:triangle-variance}
Let \(  T_\triangle(X)=\sum_{1\le u<v<w\le n}X_{uv}X_{uw}X_{vw}\) be the number of triangles.
Under the assumptions of \cref{coupling}, there is a constant \(C_B<\infty\)
such that, uniformly over \((p,\beta)\in B\) and all sufficiently large
\(n\),
\[
        \Var_{\mu_\beta^k}\bigl(T_\triangle\bigr)\le C_Bn^3.
\]
\end{corollary}

\begin{proof}
%Let \(Q\) and \(R\) be the sums of \(\xi_e\xi_f\) over unordered adjacent edge-pairs and of \(\xi_e\xi_f\xi_g\) over triangles, respectively.  
A direct computation using
\(\sum_e\xi_e=0\) and \(X_e=p+\xi_e\), together with
the recursion \cref{eq:gRecursion}, gives
\[
T_\triangle-\E T_\triangle=pf_{2,A^{(2)}}+f_{3,A^{(3)}},
\]
where
\[
A^{(2)}_{e,f}=\frac12\mathbf1_{\{e\sim f\}},\qquad
A^{(3)}_{e,f,g}=\frac16\mathbf1_{\{e,f,g\text{ form a triangle}\}}.
\]
We can also calculate
\(\norm{A^{(2)}}_2^2=N(n-2)/2\) and
\(\norm{A^{(3)}}_2^2=\binom n3/6\).  Hence
\cref{thm:higher-order-concentration-log,eq:supnorm-basic}, with \(q=2\)
and \(M=4\), yield
\(\norm{T_\triangle-\E T_\triangle}_2\le C_Bn^{3/2}\), proving the claim.
\end{proof}

\section{Proof of main results}\label{sec:proof-main}
In this section, we prove the main results in \cref{sec:pre}.
We first establish the variational threshold and the corresponding
cut-metric concentration of the conditional model.  We then turn to the
Kawasaki dynamics.  The proof of the metastable mixing theorem uses several
auxiliary estimates, whose statements are collected below and whose proofs are
deferred to later sections.  Finally, we derive the weak Poincar\'e
inequality and the higher-order concentration bound from the mixing theorem.

To prove \cref{prop:1}, we need the following lemma.

\begin{lemma}\label{lem:sharp}
For every \(p\in(0,1)\) and \(q\in[0,1]\),
\[
    D(q\|p)\ge c_p(q-p)^2,
\]
where \(c_p\) is defined in \cref{prop:1}.  The constant is optimal: equality holds at \(q=p\), and, when \(p\ne1/2\), also at \(q=1-p\).
\end{lemma}

\begin{proof}
For \(p=1/2\), the function \(F(q):=D(q\|1/2)-2(q-1/2)^2\) satisfies \(F(1/2)=F'(1/2)=0\) and
\(
    F''(q)=1/[q(1-q)]-4\ge0,
\)
so \(F\ge0\).  By the symmetry
\(D(q\|p)=D(1-q\|1-p)\), it remains to consider \(p<1/2\).

Set \(c=c_p\) and \(F(q):=D(q\|p)-c(q-p)^2\). With
\(x=(1-p)/p>1\), the elementary bounds
\(
    2(x-1)/(x+1)<\log x<(x-x^{-1})/2
\)
give
\(
    2<c<1/(2p(1-p)).
\)
Consequently, the equation \(F''(q)=0\), where
\(
    F''(q)=1/(q(1-q))-2c,
\)
has two roots \(q_-<q_+\) satisfying
\(
    p<q_-<\frac12<q_+<1-p.
\)
Moreover,
\[
    F'(p)=F'(1/2)=F'(1-p)=0,
    \qquad
    F(p)=F(1-p)=0.
\]
Since \(F''>0\) on \((0,q_-)\cup(q_+,1)\) and \(F''<0\) on \((q_-,q_+)\), it follows that
\[
    F'<0\ \text{on }(0,p),\quad
    F'>0\ \text{on }(p,1/2),\quad
    F'<0\ \text{on }(1/2,1-p),\quad
    F'>0\ \text{on }(1-p,1).
\]
Together with \(F(p)=F(1-p)=0\), this proves \(F(q)\ge0\) for all \(q\in[0,1]\).
\end{proof}

\begin{proof}[Proof of \cref{prop:1}]
We first prove the first assertion. Fix \(W\in\calW_p\), and write
\[
U(x,y)=W(x,y)-p,
\qquad
g(x)=\int_0^1 U(x,y)\dd y.
\]
We have
\[
\int_0^1 g(x)\dd x
=
\int_{[0,1]^2} U(x,y)\dd x\dd y
=0.
\]
Moreover,
\(
d_W(x)=p+g(x),
\)
and therefore
\(
t(W)
=
\int_0^1 (p+g(x))^2\dd x
=
p^2+\int_0^1 g(x)^2\dd x.
\)
Since \(\Ip(p\mathbf{1})=0\) and \(t(p\mathbf{1})=p^2\), it follows that
\[
\mathcal{F}_{\beta}(W)-\mathcal{F}_{\beta}(p\mathbf{1})
=
\beta \int_0^1 g(x)^2\dd x-\Ip(W).
\]

It remains to compare the entropy term with the degree fluctuation
\(\int_0^1 g(x)^2\dd x\). Define
\[
R(x,y):=U(x,y)-g(x)-g(y).
\]
Because \(W\) is symmetric and \(\int_0^1g=0\), we have
\[
\int_0^1 R(x,y)\dd y=0,
\qquad
\int_0^1 R(x,y)\dd x=0.
\]
Hence
\[
\begin{aligned}
\int_{[0,1]^2} U(x,y)^2\dd x\dd y
&=
\int_{[0,1]^2}
\bigl(g(x)+g(y)+R(x,y)\bigr)^2
\dd x\dd y  \\
&=
2\int_0^1 g(x)^2\dd x
+
\int_{[0,1]^2} R(x,y)^2\dd x\dd y  \\
&\ge
2\int_0^1 g(x)^2\dd x.
\end{aligned}
\]
By \cref{lem:sharp}, we therefore get
\begin{equation}\label{eq:entropy-degree-bound}
\Ip(W)
=
\frac12\int_{[0,1]^2}\Dkl(W(x,y)\|p)\dd x\dd y
\ge
\frac{c_p}{2}\int_{[0,1]^2} U(x,y)^2\dd x\dd y
\ge
c_p\int_0^1 g(x)^2\dd x.
\end{equation}
Consequently,
\begin{equation}\label{eq:variational-gap}
\mathcal{F}_{\beta}(W)-\mathcal{F}_{\beta}(p\mathbf{1})
\le
(\beta-c_p)\int_0^1 g(x)^2\dd x.
\end{equation}
If \(\beta \le c_p\), this proves \(\mathcal{F}_{\beta}(W)\le \mathcal{F}_{\beta}(p\mathbf{1})\) for all
\(W\in\calW_p\). Thus \(W\equiv p\) is a maximizer. 

If \(\beta<c_p\) and
\(\mathcal F_\beta(W)=\mathcal F_\beta(p\mathbf 1)\), then
\eqref{eq:variational-gap} implies \(g=0\) almost everywhere, and hence
$$
\mathcal F_\beta(W)-\mathcal F_\beta(p\mathbf 1)=-I_p(W).
$$
Thus \(W\equiv p\) is the unique maximizer. It remains to consider \(\beta=c_p\). If
\(\mathcal F_{c_p}(W)=\mathcal F_{c_p}(p\mathbf 1)\), then
\eqref{eq:entropy-degree-bound} must hold with equality throughout.
In particular,
$$
D(W(x,y)\|p)=c_p(W(x,y)-p)^2
$$
for almost every \((x,y)\). By Lemma~\ref{lem:sharp}, this implies
\(W(x,y)\in\{p,1-p\}\) almost everywhere when \(p\ne1/2\), while
\(W=1/2\) almost everywhere when \(p=1/2\). Since \(W\in\mathcal W_p\),
the former case also forces \(W=p\) almost everywhere.

We now prove the second assertion. Suppose that \(\beta >c_p\). By the
definition of \(c_p\), there exists \(q\in(0,1)\), \(q\ne p\), such that
\(
\beta (q-p)^2-\Dkl(q\|p)>0.
\)
Write
\(
\delta=q-p. %a=\beta \delta^2-\Dkl(q\|p)>0.
\)
For \(\alpha\in(0,1)\), split \([0,1]\) into
\(
A=[0,\alpha], B=(\alpha,1],
\)
and define
\(
r_\alpha
=
p-\frac{2\alpha}{1-\alpha}\delta.
\)
For all sufficiently small \(\alpha>0\), we have \(r_\alpha\in(0,1)\).
Define the block graphon \(W_\alpha\) by
\[
W_\alpha(x,y)
=
\begin{cases}
p, & x,y\in A, \\[0.3em]
q, & x\in A,\ y\in B \text{ or } x\in B,\ y\in A, \\[0.3em]
r_\alpha, & x,y\in B.
\end{cases}
\]
Then \(W_\alpha\) satisfies the fixed-edge constraint. Indeed,
\[
\begin{aligned}
e(W_\alpha)
&=
\alpha^2p
+
2\alpha(1-\alpha)q
+
(1-\alpha)^2r_\alpha  \\
&=
\alpha^2p
+
2\alpha(1-\alpha)(p+\delta)
+
(1-\alpha)^2
\left(p-\frac{2\alpha}{1-\alpha}\delta\right) \\
&=p.
\end{aligned}
\]
Thus \(W_\alpha\in\calW_p\).

Next, we compute \(t(W_\alpha)\). If \(x\in A\), then
\[
d_{W_\alpha}(x)=\alpha p+(1-\alpha)q
=p+(1-\alpha)\delta.
\]
If \(x\in B\), then
\[
d_{W_\alpha}(x)
=
\alpha q+(1-\alpha)r_\alpha
=
p-\alpha\delta.
\]
Therefore
\[
\begin{aligned}
t(W_\alpha)
&=
\alpha\bigl(p+(1-\alpha)\delta\bigr)^2
+
(1-\alpha)\bigl(p-\alpha\delta\bigr)^2  \\
&=
p^2+\alpha(1-\alpha)\delta^2.
\end{aligned}
\]

The entropy term is
\[
\begin{aligned}
\Ip(W_\alpha)
&=
\frac12
\left[
\alpha^2\Dkl(p\|p)
+
2\alpha(1-\alpha)\Dkl(q\|p)
+
(1-\alpha)^2\Dkl(r_\alpha\|p)
\right] \\
&=
\alpha(1-\alpha)\Dkl(q\|p)
+
\frac{(1-\alpha)^2}{2}\Dkl(r_\alpha\|p).
\end{aligned}
\]
Since \(r_\alpha\to p\) as \(\alpha\downarrow0\), Taylor's expansion of
\(r\mapsto\Dkl(r\|p)\) at \(r=p\) gives
\[
\Dkl(r_\alpha\|p)=O\bigl((r_\alpha-p)^2\bigr)=O(\alpha^2).
\]
Hence
\[
\begin{aligned}
\mathcal{F}_{\beta}(W_\alpha)-\mathcal{F}_{\beta}(p\mathbf{1})
&=
\beta \bigl(t(W_\alpha)-p^2\bigr)-\Ip(W_\alpha) \\
&=
\alpha(1-\alpha)
\bigl[\beta \delta^2-\Dkl(q\|p)\bigr]
-
\frac{(1-\alpha)^2}{2}\Dkl(r_\alpha\|p) \\
&=
\alpha(1-\alpha)\bigl[\beta \delta^2-\Dkl(q\|p)\bigr]+O(\alpha^2).
\end{aligned}
\]
Since \(\beta \delta^2-\Dkl(q\|p)>0\), the last expression is strictly positive for all sufficiently
small \(\alpha>0\). Thus, for such \(\alpha\),
\[
\mathcal{F}_{\beta}(W_\alpha)>\mathcal{F}_{\beta}(p\mathbf{1}).
\]
Therefore the constant graphon \(W\equiv p\) is not a maximizer whenever
\(\beta >c_p\).
\end{proof}

Next, we prove \cref{theo of appro}.  The argument is a
consequence of the fixed-edge large deviation principle and the
uniqueness of the variational maximizer.

\begin{proof}[Proof of \cref{theo of appro}]
For \((p,\beta)\in K\), write
\(
        \Psi(p,\beta)
        =
        \sup_{W\in\mathcal W_p}
        \{\mathcal F_\beta(W)\}.
\)
By \cref{prop:1}, the constant graphon \(p\mathbf 1\) is the unique
maximizer, and hence
\(
        \Psi(p,\beta)=\beta p^2.
\)
For \(r>0\), set
\(
        A_r(p)
        =
        \{W\in\mathcal W_p:d_\square(W,p\mathbf 1)\ge r\}
\)
and
\[
        \Delta_r(p,\beta)
        =
        \Psi(p,\beta)
        -
        \sup_{W\in A_r(p)}
        \mathcal F_\beta(W),
\]
with the convention that the supremum over the empty set is \(-\infty\).
We first show that
$\Delta_r(p,\beta)$ is bounded away from zero uniformly for \((p,\beta)\in K\).

Since \(p\mapsto c_p\) is continuous and \(K\) is a compact subset of the
replica symmetric region, there is
\[
        \kappa_K
        :=
        \min_{(p,\beta)\in K}(c_p-\beta)>0.
\]
Fix \((p,\beta)\in K\) and \(W\in\mathcal W_p\).  Put
\[
        U(x,y)=W(x,y)-p,
        \qquad
        g(x)=\int_0^1U(x,y)\,\mathrm dy.
\]
As in the proof of \cref{prop:1},
\[
        t(W)-p^2=\int_0^1g(x)^2\,\mathrm dx,
        \qquad
        \int_{[0,1]^2}U(x,y)^2\,\mathrm dx\,\mathrm dy
        \ge 2\int_0^1g(x)^2\,\mathrm dx,
\]
and the definition of \(c_p\) gives
\[
        I_p(W)
        \ge
        \frac{c_p}{2}
        \int_{[0,1]^2}U(x,y)^2\,\mathrm dx\,\mathrm dy.
\]
Consequently, using \(\beta\ge0\),
\begin{align}
        \Psi(p,\beta)
        -\mathcal F_\beta(W)
        &=
        I_p(W)-\beta\int_0^1g(x)^2\,\mathrm dx \ge
        \frac{c_p-\beta}{2}
        \lVert W-p\mathbf 1\rVert_2^2.                 \label{eq:prop21-strong-gap}
\end{align}
The constant graphon is invariant under relabeling, and therefore
\[
        d_\square(W,p\mathbf 1)
        =\lVert W-p\mathbf 1\rVert_\square
        \le \lVert W-p\mathbf 1\rVert_2.
\]
It follows that, uniformly for \((p,\beta)\in K\),
\begin{equation}
        \Delta_r(p,\beta)
        \ge
        \frac{\kappa_K r^2}{2}.                       \label{eq:prop21-uniform-gap}
\end{equation}

We now use the large-deviation principle to prove \cref{eq:prop21-sequential-rate} along any subsequence with converging parameters. After that, we will prove the desired result \cref{eq:prop21-1} by contradiction.  Let \(n_j\to\infty\),
let
\[
        N_j=\binom{n_j}{2},
        \qquad
        k_j=N_jp_j\in\mathbb Z,
        \qquad
        (p_j,\beta_j)\in K.
\]
After passing to a subsequence, compactness of \(K\) gives
\(
        (p_j,\beta_j)\longrightarrow(p,\beta)\in K.
\)
For every event \(E\subseteq\Omega_{k_j}\), we have
\begin{equation}
        \mu_{\beta_j}^{k_j}(E)
        =
        \frac{
        \displaystyle\sum_{X\in E}
        \exp\{n_j^2\beta_jt(W_X)\}}
        {
        \displaystyle\sum_{X\in\Omega_{k_j}}
        \exp\{n_j^2\beta_jt(W_X)\}}.                 \label{eq:prop21-tilted-ratio}
\end{equation}

Let \(\nu_{n_j,k_j}\) be the uniform measure on \(\Omega_{k_j}\).  
The fixed-edge graphon large deviation principle
(\cite[Theorem~1.1]{dembo2018large}) applies for \(\nu_{n_j,k_j}\) to every sequence
\(k_j/N_j\to p\).
%.  Notice that, under the usual zero-diagonal adjacency matrix convention,
%\[
%        \int_{[0,1]^2}W_X
%        =\frac{2k_j}{n_j^2}
%        =\left(1-\frac1{n_j}\right)p_j,
%\]
%so the closed and open sets below are taken in the ambient graphon space; the
%rate function is \(I_p\) on \(\mathcal W_p\) and \(+\infty\) outside it.
Moreover,
\[
        \sup_W|\beta_jt(W)-\beta t(W)|
        \le |\beta_j-\beta|\longrightarrow0.
\]
Hence the large deviation upper bound, together with the cut-continuity of
\(t\), gives, for every closed set \(F\),
\begin{align}
        \limsup_{j\to\infty}\frac1{n_j^2}
        \log
        \sum_{\substack{X\in\Omega_{k_j}\\ W_X\in F}}
        \exp\{n_j^2\beta_jt(W_X)\}
        \le
        C_p
        +
        \sup_{W\in F\cap\mathcal W_p}
        \{\beta t(W)-I_p(W)\},                       \label{eq:prop21-upper}
\end{align}
where
\[
        C_p
        =
        \lim_{j\to\infty}\frac1{n_j^2}\log|\Omega_{k_j}|
        =\frac12h(p),
        \qquad
        h(p)=-p\log p-(1-p)\log(1-p).
\]
Likewise, if \(V\) is any open neighborhood of \(p\mathbf 1\), then the
large deviation lower bound gives
\besn{\label{eq:LDlower}
        \liminf_{j\to\infty}\frac1{n_j^2}
        \log
        \sum_{X\in\Omega_{k_j}}
        \exp\{n_j^2\beta_jt(W_X)\}
        &\ge
        C_p
        +
        \sup_{W\in V\cap\mathcal W_p}
        \{\beta t(W)-I_p(W)\} \\
        &\ge
        C_p+\beta p^2.
}

Fix the \(\eta\) in the statement of \cref{theo of appro} and define the closed set
\(
        F_{\eta/2}(p)
        =
        \{W:d_\square(W,p\mathbf 1)\ge\eta/2\}.
\)
Since
\(
        d_\square(p_j\mathbf 1,p\mathbf{1})=|p_j-p|,
\)
for all sufficiently large \(j\),
\(
        \{d_\square(W_X,p_j\mathbf 1)>\eta\}
        \subseteq
        \{W_X\in F_{\eta/2}(p)\}.
\)
Using \cref{eq:prop21-tilted-ratio}, \cref{eq:prop21-upper} and \cref{eq:LDlower}, we obtain
\begin{align}
        \limsup_{j\to\infty}\frac1{n_j^2}
        \log
        \mu_{\beta_j}^{k_j}
        \bigl(d_\square(W_X,p_j\mathbf 1)>\eta\bigr)
        &\le
        \sup_{\substack{W\in\mathcal W_p\\
        d_\square(W,p\mathbf 1)\ge\eta/2}}
        \{\beta t(W)-I_p(W)\}
        -\beta p^2\le
        -\frac{\kappa_K\eta^2}{8},                  \label{eq:prop21-sequential-rate}
\end{align}
where the last inequality is \cref{eq:prop21-uniform-gap} with
\(r=\eta/2\).

Finally, suppose that 
there would exist a subsequence such that
\[
        \mu_{\beta_j}^{k_j}
        \bigl(d_\square(W_X,p_j\mathbf 1)>\eta\bigr)
        >
        \exp\left\{-\frac{\kappa_K\eta^2}{16}n_j^2\right\}
\]
for every \(j\).  After extracting a convergent subsequence, this contradicts
\cref{eq:prop21-sequential-rate}.  Therefore there exists
\(n_{K,\eta}<\infty\) such that, for every \(n\ge n_{K,\eta}\), every
\((p,\beta)\in K\), and every admissible \(k=Np\),
\[
        \mu_\beta^k
        \bigl(d_\square(W_X,p\mathbf 1)>\eta\bigr)
        \le
        \exp\left\{-\frac{\kappa_K\eta^2}{16}n^2\right\}.
\]
Thus the proposition holds, for example, with
\(
        c_{K,\eta}=\kappa_K\eta^2/16,
        C_{K,\eta}=1.
\)
\end{proof}
The proof of the mixing result, \cref{coupling}, relies
on four auxiliary estimates.  First, cut metric closeness to the constant
graphon implies local regularity with probability at least \(1-Ce^{-cn}\) (\cref{le-1}).
Second, a chain started from a sufficiently regular configuration remains in a
regular region for an exponentially long time with high probability (\cref{co-1}).  Third,
inside the regular region, the Kawasaki path coupling is contractive in one
step (\cref{le-5}).  Fourth, an abstract stopped-contraction lemma (\cref{le-21}) converts this local
contraction into a total-variation mixing bound.  We state these estimates
below and defer their proofs to later sections and Appendix \ref{sec:degree-to-local}.
\begin{lemma}[Degree concentration]\label{le-1}
Let \(K\subset\mathscr R_{\mathrm{RS}}\) be compact.  For every fixed
\(\varepsilon>0\), there exist constants
\(\eta_0=\eta_0(K,\varepsilon)>0\), \(c=c(K,\varepsilon)>0\), and
\(C=C(K,\varepsilon)<\infty\) such that, for every fixed
\(0<\eta<\eta_0\), there exists
\(n_0=n_0(K,\varepsilon,\eta)<\infty\) for which, uniformly over
\((p,\beta)\in K\), every \(n\ge n_0\) satisfies
%\[
%    \mu^k_{\beta}\bigl(\Gamma_{p,\varepsilon}^c
%    \,\bigm|\, \square_{\eta}^k(p)\bigr)
%    \le C e^{-cn}.
%\]
\[
    \mu^k_{\beta}\bigl(\Gamma_{p,\varepsilon}
    \,\bigm|\, \square_{\eta}^k(p)\bigr)
    \ge 1-Ce^{-cn}.
\]
\end{lemma}

\begin{lemma}[Confinement]\label{co-1}
Under the standing assumption \((p,\beta)\in B\subset
\mathscr R_{\mathrm{Kaw}}\), there exists \(\varepsilon_0>0\),
depending only on \(B\), such that, for every
\(0<\varepsilon\le\varepsilon_0\), there exist constants
\(\alpha_{B,\varepsilon}>0\) and \(c_{B,\varepsilon}>0\), and \(n_0=n_0(B,\varepsilon)<\infty\), such that
the following holds uniformly over \((p,\beta)\in B\) for every 
\(n\ge n_0\).  If
\((X_t)_{t\ge0}\) is the full Kawasaki chain and
\(
        X_0\in\Gamma_{p,\varepsilon/4},
\)
then
\[
        \mathbb P\left(
        X_t\in\Gamma_{p,\varepsilon/2}\ \text{for all }
        0\le t\le e^{\alpha_{B,\varepsilon} n}
        \right)
        \ge 1-e^{-c_{B,\varepsilon} n}.
\]
Moreover, if \(X_0\in\Gamma_{p,\varepsilon}\), then there are
\(C=C(B,\varepsilon)<\infty\) and an integer \(T_n\le Cn^2\) such that
\[
        \mathbb P\left(X_{T_n}\in
        \Gamma_{p,\varepsilon/2}\right)
        \ge 1-e^{-c_{B,\varepsilon} n}.
\]
\end{lemma}

\begin{lemma}[Contractive coupling]\label{le-5}
Under the standing assumption \((p,\beta)\in B\subset
\mathscr R_{\mathrm{Kaw}}\), there exist \(\varepsilon_0=\varepsilon_0(B)>0\),
\(c=c(B)>0\) and
\(n_0=n_0(B)<\infty\) such that, uniformly over \((p,\beta)\in B\), the following holds for every
 \(n\ge n_0\).  Define
\(
        \Lambda=\Gamma_{p,\varepsilon_0}\times\Gamma_{p,\varepsilon_0}
\)
and let \(X',Y'\) be obtained from \(x,y\) by the coupling constructed in Section \ref{sec:path-coupling}  for
the full Kawasaki dynamics. Then, for every \((x,y)\in\Lambda\),
\[
        \mathbb E \rho(X',Y')
        \le
        \left(1-\frac{c}{n^2}\right)\rho(x,y).
\]
\end{lemma}

\begin{lemma}[{\citealp[Lemma~4.1 and Corollary~4.2]
{bresler2024metastable}}]\label{le-21}
Let \((\Sigma,d)\) be a finite metric space, let \(P\) be a Markov kernel on \(\Sigma\), and let \(\Lambda\subseteq\Sigma\times\Sigma\).  Assume that there exists \(\gamma\in(0,1]\) such that,
for every \((x,y)\in\Lambda\), there is a coupling
\(Q_{x,y}\) of \(P(x,\cdot)\) and \(P(y,\cdot)\) satisfying,
for \((X,Y)\sim Q_{x,y}\),
\[
    \mathbb E d(X,Y)\le(1-\gamma)d(x,y).
\]
Define
\(
    \bar d=\max_{x,y\in\Sigma}d(x,y)
\) and
  \(  \underline d=\min_{x\ne y}d(x,y).
\)
Then, for any \((X_0,Y_0)\in\Sigma\times\Sigma\), there is a coupling of two \(P\)-chains such that
\[
    \E d(X_{k+1},Y_{k+1})
    \le (1-\gamma)\E d(X_k,Y_k)
    +\bar d\,\P((X_k,Y_k)\in\Lambda^c).
\]
Consequently,
\[
    \E d(X_t,Y_t)
    \le \bar d\left[(1-\gamma)^t+
    \frac{\sup_{k\le t}\P((X_k,Y_k)\in\Lambda^c)}{\gamma}\right],
\]
and
\[
    d_{\mathrm{TV}}(\mathcal L(X_t),\mathcal L(Y_t))
    \le \frac{\bar d}{\underline d}\left[(1-\gamma)^t+
    \frac{\sup_{k\le t}\P((X_k,Y_k)\in\Lambda^c)}{\gamma}\right].
\]
\end{lemma}

\begin{proof}[Proof of \cref{coupling}]
Let $\varepsilon_0(B)$ be the minimum of the corresponding constants in 
\cref{co-1,le-5}.
Fix $0<\varepsilon\le\varepsilon_0(B)$, and write
\[
    \mu:=\mu_\beta^k,
    \qquad G:=\Gamma_{p,\varepsilon/4},
    \qquad \Gamma:=\Gamma_{p,\varepsilon},
    \qquad \pi:=\mu(\,\cdot\mid\Gamma).
\]
All constants below depend only on $B$ and $\varepsilon$.

\smallskip
\noindent\textbf{Step 1: stationary mass of the regular core.}
Apply \cref{le-1} with $K=B$ and radius $\varepsilon/4$, and fix
\[
    \eta:=\frac12\eta_0(B,\varepsilon/4)>0.
\]
Combining that lemma
with \cref{theo of appro}, also with $K=B$, gives
\[
\begin{aligned}
    \mu(G^c)
    &\le \mu\bigl(G^c\mid\square_\eta^k(p)\bigr)
       +\mu\bigl((\square_\eta^k(p))^c\bigr)\\
    &\le A_{\mathrm{deg}}e^{-b_{\mathrm{deg}}n}
       +A_{\square}e^{-b_{\square}n^2}
\end{aligned}
\]
for all sufficiently large $n$. 
Consequently, with
\[
    A_s:=\max\{1,A_{\mathrm{deg}}+A_{\square}\},
    \qquad b_s:=\min\{b_{\mathrm{deg}},b_{\square}\}>0,
\]
there is $n_s=n_s(B,\varepsilon)$ such that, for every $n\ge n_s$,
\begin{equation}\label{eq:mix21-core-mass}
    \mu(G^c)\le A_s e^{-b_s n},
    \qquad \mu(\Gamma^c)\le A_s e^{-b_s n},
\end{equation}
\begin{equation}\label{eq:mix21-conditioning-tv}
    d_{\mathrm{TV}}(\mu,\pi)
    =\mu(\Gamma^c)\le A_s e^{-b_s n}.
\end{equation}
This proves \cref{eq:regularregion}.

\smallskip
\noindent\textbf{Step 2: comparison of the full and restricted paths.}
From \cref{co-1}, there exist positive constants $a_f, b_f, n_f$ such that, for every $x\in G$, every $n\ge n_f$, and every integer
$T\le e^{a_f n}$, we have
\begin{equation}\label{eq:mix21-confinement}
    \mathbb P_x\bigl(X_t\notin\Gamma_{p,\varepsilon/2}
                 \text{ for some }0\le t\le T\bigr)
    \le e^{-b_f n}.
\end{equation}
For a legal exchange deleting $s$ and adding $t$, the definition of
$r(e,\cdot)$ in \cref{eq:rex} gives the exact identity
\[
    r(e,z^{s\leftrightarrow t})-r(e,z)
    =\frac{\mathbf 1_{\{t\sim e\}}-
            \mathbf 1_{\{s\sim e\}}}{2n}.
\]
Thus a single exchange changes each local density by at most $1/(2n)$.
If $n\ge\varepsilon^{-1}$, every proposal from
$\Gamma_{p,\varepsilon/2}$ stays in $\Gamma$, and hence
\begin{equation}\label{eq:kernels-agree-interior}
    P_\varepsilon(z,\cdot)=P(z,\cdot),
    \qquad z\in\Gamma_{p,\varepsilon/2}.
\end{equation}
Start the full and restricted chains from the same $x\in G$, and use the
same proposals and acceptance while their states agree.
On the complement of the event in \eqref{eq:mix21-confinement}, their
paths agree through time $T$. Therefore
\begin{equation}\label{eq:path-comparison-X-new}
    \mathbb P\bigl((X_0,\ldots,X_T)\ne
                   (\widetilde X_0,\ldots,\widetilde X_T)\bigr)
    \le e^{-b_f n},
    \qquad T\le e^{a_f n}.
\end{equation}
In particular,
$d_{\mathrm{TV}}(xP^T,xP_\varepsilon^T)\le e^{-b_f n}$.
This proves \textup{(i)}.

\smallskip
\noindent\textbf{Step 3: contraction against a stationary full chain.}
From \cref{le-5}, there exist
$\lambda_B\in(0,1]$ and $n_c=n_c(B)$ such that for $n\geq n_c$, there is a one-step
coupling $X', Y'$ of the full Kawasaki dynamics starting from $z,w$ satisfying
\[
    \mathbb E\rho(X',Y')
    \le\left(1-\frac{\lambda_B}{n^2}\right)\rho(z,w),
    \qquad (z,w)\in\Lambda:=\Gamma\times\Gamma.
\]
Couple the full chain with $X_0=x\in G$ to a full chain with
$Y_0\sim\mu$ using \cref{le-21}. 
%Its application to this random initial state follows by conditioning on $Y_0$ and using the same statewise coupling kernels. 
Stationarity gives $Y_t\sim\mu$ for every $t$.
Set
\[
    b_0:=\min\{b_s,b_f\}>0.
\]
For every integer $T\le e^{a_f n}$, the marginal estimates
\eqref{eq:mix21-core-mass} and \eqref{eq:mix21-confinement} imply
\[
\begin{aligned}
    \sup_{0\le t\le T}\mathbb P\bigl((X_t,Y_t)\notin\Lambda\bigr)
    &\le \sup_{0\le t\le T}\mathbb P(X_t\notin\Gamma)
          +\mu(\Gamma^c)\\
    &\le e^{-b_f n}+A_s e^{-b_s n}
     \le (1+A_s)e^{-b_0 n}.
\end{aligned}
\]
%Only the one-time marginals of $Y_t$ are used here; no confinement estimate for the stationary chain is required.

The minimum nonzero Kawasaki distance is $1$, and its diameter is
$\min\{k,N-k\}\le N\le n^2$. Applying \cref{le-21} with
$d=\rho$ and $\gamma=\lambda_B/n^2$ therefore yields
\begin{equation}\label{eq:full-chain-tv-new}
\begin{aligned}
    d_{\mathrm{TV}}(xP^T,\mu)
    &\le n^2\left(1-\frac{\lambda_B}{n^2}\right)^T
       +\frac{1+A_s}{\lambda_B}n^4e^{-b_0 n}\\
    &\le n^2e^{-\lambda_B T/n^2}
       +\frac{1+A_s}{\lambda_B}n^4e^{-b_0 n}.
\end{aligned}
\end{equation}
Using \eqref{eq:mix21-conditioning-tv} and
\eqref{eq:path-comparison-X-new}, the triangle inequality also gives
\[
    d_{\mathrm{TV}}(xP_\varepsilon^T,\pi)
    \le d_{\mathrm{TV}}(xP^T,\mu)
       +e^{-b_f n}+A_s e^{-b_s n}.
\]
Consequently, if
\[
    D:=(1+A_s)(1+\lambda_B^{-1}),
\]
then, for every integer $T\le e^{a_f n}$,
\begin{equation}\label{eq:restricted-mixing-bound-new}
    \max\bigl\{d_{\mathrm{TV}}(xP^T,\mu),
               d_{\mathrm{TV}}(xP_\varepsilon^T,\pi)\bigr\}
    \le n^2e^{-\lambda_B T/n^2}+Dn^4e^{-b_0 n}.
\end{equation}
Both $b_0$ and $D$ depend only on $B$ and $\varepsilon$.

\smallskip
\noindent\textbf{Step 4: choice of the theorem constants and mixing time.}
Choose
\[
    \alpha_{B,\varepsilon}:=a_f,
    \qquad c_{B,\varepsilon}:=\frac{b_0}{2},
    \qquad C_{B,\varepsilon}:=
        \max\{1,A_s,2/\lambda_B\}.
\]
Take $n_0=n_0(B,\varepsilon)$ at least
\[
    \max\{2,n_s,n_f,n_c,\lceil\varepsilon^{-1}\rceil\},
\]
and sufficiently large that, for every $n\ge n_0$,
\begin{equation}\label{eq:mix21-threshold-conditions}
\begin{split}
    Dn^4e^{-c_{B,\varepsilon}n}&\le\frac12,\\
    C_{B,\varepsilon}n^2
       \bigl(2\log n+c_{B,\varepsilon}n\bigr)+1
       &\le e^{\alpha_{B,\varepsilon}n}.
\end{split}
\end{equation}
%Such a threshold exists because $c_{B,\varepsilon}$ and $\alpha_{B,\varepsilon}$ are positive. Every quantity entering these requirements depends only on $B$ and $\varepsilon$.

For $e^{-c_{B,\varepsilon}n}\le\delta\le1$, put
\[
    T:=\left\lceil C_{B,\varepsilon}n^2
                      \log(n^2/\delta)\right\rceil.
\]
Since $\log(1/\delta)\le c_{B,\varepsilon}n$, the second inequality in
\eqref{eq:mix21-threshold-conditions} ensures
$T\le e^{\alpha_{B,\varepsilon}n}$, uniformly over the stated range of
$\delta$. Moreover, $\lambda_B C_{B,\varepsilon}\ge2$ gives
\[
    n^2e^{-\lambda_B T/n^2}
    \le n^2\left(\frac{\delta}{n^2}\right)^2
    =\frac{\delta^2}{n^2}\le\frac\delta2.
\]
Since $b_0=2c_{B,\varepsilon}$, the first inequality in
\eqref{eq:mix21-threshold-conditions} gives
\[
    Dn^4e^{-b_0 n}
    \le\frac12 e^{-c_{B,\varepsilon}n}
    \le\frac\delta2.
\]
Substitution into \eqref{eq:restricted-mixing-bound-new} proves both
mixing estimates in \textup{(ii)}.
%Finally, \eqref{eq:mix21-core-mass} gives \eqref{eq:regularregion}, while \eqref{eq:path-comparison-X-new} gives part~\emph{(i)}, by the choices \(C_{B,\varepsilon}\ge A_s\), \(c_{B,\varepsilon}\le b_s\wedge b_f\), and \(C_{B,\varepsilon}\ge1\).  This completes the proof.
\end{proof}

We next convert the metastable mixing estimate into a weak
Poincar\'e inequality by a finite-time spectral truncation.
The argument is a weak analogue of the 
argument used in the proofs of \cite[Lemma~13.7]{levin2009markov} and  \cite[Theorem~2.1]{ganguly2019}.

\begin{proof}[Proof of \cref{lem:defective-poincare}]
All constants in this proof depend
only on \(B\) and \(\varepsilon\).
As in the proof of \cref{coupling}, write
\(
    \Gamma=\Gamma_{p,\varepsilon},
 G=\Gamma_{p,\varepsilon/4},
\mu=\mu_\beta^k,
    \pi=\mu(\,\cdot\mid\Gamma).
\)
From \cref{coupling}, there are constants
\(C_{\mathrm{mix}},c_{\mathrm{mix}},C,c>0\) such that, for every
\(x\in G\) and every \(\delta\ge e^{-c_{\mathrm{mix}}n}\),
\begin{equation}\label{eq:mixing-input}
    d_{\mathrm{TV}}(xP_{\varepsilon}^T,\pi)\le\delta,
    \qquad
    T=\lceil C_{\mathrm{mix}}n^2\log\frac{n^2}{\delta}\rceil,
\end{equation}
and
\begin{equation}\label{eq:good-set-mass-input}
    \pi(G^c)\le Ce^{-cn},
    \qquad
    \mu(\Gamma^c)\le Ce^{-cn}.
\end{equation}
Choose \(a\in(0,c_{\mathrm{mix}})\), put \(\delta=e^{-an}\), and set
\[
    \tau
    :=
    \left\lceil
    C_{\mathrm{mix}}n^2\log(n^2e^{an})
    \right\rceil .
\]
Then \(\tau\le Kn^3\) for a constant \(K<\infty\).  Let
\(h:=f-\pi f\).  For \(x\in G\), since \(\pi h=0\), \eqref{eq:mixing-input}
gives
\[
    \abs{P_{\varepsilon}^\tau h(x)}
    =
    \abs{\int h\,\dd(xP_\varepsilon^\tau-\pi)}
    \le
    2e^{-an}\norm{h}_\infty .
\]
On \(G^c\), we use only \(\abs{P_\varepsilon^\tau h}\le\norm{h}_\infty\).
Since \(\norm{h}_\infty\le2\norm{f}_\infty\), \eqref{eq:good-set-mass-input}
yields, after changing constants,
\begin{equation}\label{eq:long-time-L2-small}
    \norm{P_{\varepsilon}^\tau h}_{L^2(\pi)}^2
    \le
    Ce^{-bn}\norm{f}_\infty^2
\end{equation}
for some \(b>0\).

Since $\Gamma$ is finite and $P_\varepsilon$ is reversible
with respect to $\pi$, there is an orthonormal eigenbasis
$(u_j)_{j=1}^{|\Gamma|}$ of $L^2(\pi)$, with
$P_\varepsilon u_j=\lambda_j u_j$ and $\lambda_j\in[-1,1]$.
Writing
\[
    h=\sum_j a_j u_j,
    \qquad a_j=\langle h,u_j\rangle_\pi,
\]
define the finite positive measure
\[
    \nu_h:=\sum_j a_j^2\delta_{\lambda_j}.
\]
The following identities then follow from orthogonality and
$\mathcal E_\varepsilon(f,f)
=\langle h,(I-P_\varepsilon)h\rangle_\pi$:
\[
    \Var_\pi(f)=\int 1\,\dd\nu_h(\lambda),
    \qquad
    \mathcal E_\varepsilon(f,f)
    =\int(1-\lambda)\,\dd\nu_h(\lambda),
\]
and
\[
    \norm{P_{\varepsilon}^\tau h}_{L^2(\pi)}^2
    =\int\lambda^{2\tau}\,\dd\nu_h(\lambda).
\]
Fix \(\theta>0\) and split the spectrum into
\[
    A_-:=\left\{\lambda\le1-\frac{\theta}{n^2}\right\},
    \qquad
    A_+:=\left\{\lambda>1-\frac{\theta}{n^2}\right\}.
\]
On \(A_-\),
\begin{equation}\label{eq:low-spectrum}
    \nu_h(A_-)
    \le
    \frac{n^2}{\theta}\mathcal E_\varepsilon(f,f).
\end{equation}
On \(A_+\), the eigenvalues are positive for all sufficiently large \(n\), and
\(\tau\le Kn^3\) gives
\[
    \lambda^{2\tau}
    \ge
    \left(1-\frac{\theta}{n^2}\right)^{2\tau}
    \ge
    e^{-4K\theta n}.
\]
Combining this bound with \eqref{eq:long-time-L2-small}, we obtain
\begin{equation}\label{eq:high-spectrum}
    \nu_h(A_+)
    \le
    e^{4K\theta n}\norm{P_\varepsilon^\tau h}_{L^2(\pi)}^2
    \le
    Ce^{-(b-4K\theta)n}\norm{f}_\infty^2.
\end{equation}
Taking \(\theta<b/(8K)\) and adding
\eqref{eq:low-spectrum}--\eqref{eq:high-spectrum} proves
\eqref{eq:no-log-defective-poincare}.
\end{proof}

Next, we prove \cref{cor:full-defective-poincare}, which is via \cref{eq:no-log-defective-poincare} and a comparison between the Kawasaki kernel $P$ and the restricted kernel $P_\varepsilon$ (see \cref{eq:full-cor-variance-decomposition} and \cref{eq:upperDirich}).

\begin{proof}[Proof of \cref{cor:full-defective-poincare}]
Let $\varepsilon_0(B)$ be as in \cref{coupling}.
Set
\[
    \mu:=\mu_\beta^k,
    \qquad
    \varepsilon:=\frac{\varepsilon_0(B)}{2},
    \qquad
    \Gamma:=\Gamma_{p,\varepsilon},
    \qquad
    \pi:=\mu(\,\cdot\,\mid\Gamma).
\]
Since \(\varepsilon\) is now fixed in terms of \(B\), all constants
depending on \(B\) and \(\varepsilon\) below may be regarded as depending
only on \(B\).
Applying
\cref{coupling} with the above value of \(\varepsilon\), we obtain
constants \(C_1,c_1>0\) and \(n_1<\infty\), depending only on \(B\),
such that
\begin{equation}\label{eq:full-cor-small-complement}
    \mu(\Gamma^c)\le C_1e^{-c_1n}.
\end{equation}
In particular, after increasing \(n_1\) if necessary, \(\mu(\Gamma)>0\),
so that \(\pi\) is well defined.

Let
\[
    m_\Gamma:=\pi(f).
\]
Using the fact that the variance is the minimum of the quadratic
distance from a constant, we have
\begin{align*}
    \Var_\mu(f)
    &=\inf_{a\in\mathbbm R}\int_{\Omega_k}(f-a)^2\,d\mu\\
    &\le \int_{\Omega_k}(f-m_\Gamma)^2\,d\mu\\
    &=\int_\Gamma(f-m_\Gamma)^2\,d\mu
      +\int_{\Gamma^c}(f-m_\Gamma)^2\,d\mu\\
    &=\mu(\Gamma)\Var_\pi(f)
      +\int_{\Gamma^c}(f-m_\Gamma)^2\,d\mu.
\end{align*}
Since
\[
    |m_\Gamma|
    =|\pi(f)|
    \le \norm{f}_\infty,
\]
we have \(|f-m_\Gamma|\le 2\norm{f}_\infty\).  Consequently,
\begin{equation}\label{eq:full-cor-variance-decomposition}
    \Var_\mu(f)
    \le
    \mu(\Gamma)\Var_\pi(f)
    +4\mu(\Gamma^c)\norm{f}_\infty^2.
\end{equation}

We now apply \cref{lem:defective-poincare} to the restriction of \(f\)
to \(\Gamma\).  There exist constants \(C_2,c_2>0\) and
\(n_2<\infty\), depending only on \(B\), such that
\[
    \Var_\pi(f)
    \le
    C_2n^2\mathcal E_\varepsilon(f,f)
    +C_2e^{-c_2n}\norm{f|_\Gamma}_\infty^2
    \le
    C_2n^2\mathcal E_\varepsilon(f,f)
    +C_2e^{-c_2n}\norm{f}_\infty^2.
\]
Multiplying by \(\mu(\Gamma)\le 1\) gives
\begin{equation}\label{eq:full-cor-restricted-poincare}
    \mu(\Gamma)\Var_\pi(f)
    \le
    C_2n^2\mu(\Gamma)\mathcal E_\varepsilon(f,f)
    +C_2e^{-c_2n}\norm{f}_\infty^2.
\end{equation}

It remains to compare the two Dirichlet forms.  Recall that
\(P_\varepsilon(x,y)=P(x,y)\) whenever \(x,y\in\Gamma\) and
\(x\ne y\).  The additional holding probability in the hard-wall
kernel does not contribute to the Dirichlet form.  Hence
\besn{\label{eq:upperDirich}
    \mu(\Gamma)\mathcal E_\varepsilon(f,f)
    &=
    \frac12
    \sum_{x,y\in\Gamma}
    \mu(x)P_\varepsilon(x,y)
    \bigl(f(y)-f(x)\bigr)^2\\
    &=
    \frac12
    \sum_{\substack{x,y\in\Gamma\\x\ne y}}
    \mu(x)P(x,y)
    \bigl(f(y)-f(x)\bigr)^2\\
    &\le
    \frac12
    \sum_{x,y\in\Omega_k}
    \mu(x)P(x,y)
    \bigl(f(y)-f(x)\bigr)^2\\
    &=\mathcal E_P(f,f).
}
Combining this estimate with
\eqref{eq:full-cor-small-complement},
\eqref{eq:full-cor-variance-decomposition}, and
\eqref{eq:full-cor-restricted-poincare}, we obtain
\[
    \Var_{\mu_\beta^k}(f)
    \le
    C_2n^2\mathcal E_P(f,f)
    +
    \bigl(C_2e^{-c_2n}+4C_1e^{-c_1n}\bigr)
    \norm{f}_\infty^2.
\]
Finally, by setting \(c_B:=\min\{c_1,c_2\}\), enlarging \(C_B\), and
taking \(n_0:=\max\{n_1,n_2\}\), we conclude that
\[
    \Var_{\mu_\beta^k}(f)
    \le
    C_Bn^2\mathcal E_P(f,f)
    +C_Be^{-c_Bn}\norm{f}_\infty^2,
\]
uniformly over \((p,\beta)\in B\), as claimed.
\end{proof}

%We finally use the weak Poincar\'e inequality to prove the higher-order concentration estimate.  The argument is an induction on the order of the multilinear statistic.  A single Kawasaki exchange lowers the algebraic order of the statistic, and this allows the Poincar\'e inequality to be iterated.

\begin{proof}[Proof of \cref{thm:higher-order-concentration-log}]
Set \(\varepsilon=\varepsilon_0(B)/2\),
\(\Gamma=\Gamma_{p,\varepsilon}\) and
\(\pi=\mu_\beta^k(\,\cdot\mid\Gamma)\), where $\varepsilon_0(B)$ is as in \cref{coupling}.  All constants below are
uniform over \((p,\beta)\in B\) and depend only on the displayed fixed
parameters and \(B\).
We prove \eqref{eq:moment-bound} by induction on \(d\).  The proof is first
carried out for dyadic \(q=2^j, j=1,2,\dots\).  
%The constants are chosen with a factor-two margin in the admissible range of \(q\), and 
The passage to arbitrary \(q\) is made at
the end of Step 3.  The induction is started from the convention \(f_{0,a}\equiv a\) and
\(\norm{a}_2=|a|\).

For a bounded \(H:\Gamma\to\mathbbm{R}\), define
\[
    \abs{\nabla_\varepsilon H}^2(x)
    :=
    \sum_{y\in\Gamma}P_\varepsilon(x,y)\{H(y)-H(x)\}^2.
\]
Then
\(
    \mathcal E_\varepsilon(H,H)
    =
    \frac12\int\abs{\nabla_\varepsilon H}^2\,\dd\pi.
\)

\medskip
\noindent
\textbf{Step 1: a \(q\)-dependent weak \(L^q\) Poincar\'e estimate.}
In this step, we prove that for every \(M>0\), there are constants \(C_M,c_M>0\) such that, for every
dyadic \(q\) satisfying
\(
    2\le q\le 2c_M n/\log n,
\)
and every bounded \(H:\Gamma\to\mathbbm{R}\),
\begin{equation}\label{eq:defective-Lq}
    \norm{H-\pi H}_{L^q(\pi)}
    \le
    C_Mqn\norm{\nabla_\varepsilon H}_{L^q(\pi)}
    +
    C_Mn^{-M}\norm{H}_\infty.
\end{equation}
To prove this, take \(q>2\) dyadic, put
\(
    G_H=H-\pi H,
    r=\frac q2,
    U=\abs{G_H}^{r}.
\)
Applying \eqref{eq:no-log-defective-poincare} to \(U\) and taking square roots
gives, after changing constants, 
\begin{equation}\label{eq:U-poincare}
    \norm{U-\pi U}_{L^2(\pi)}
    \le
    Cn\norm{\nabla_\varepsilon U}_{L^2(\pi)}
    +
    Ce^{-cn}\norm{U}_\infty.
\end{equation}
The elementary inequality
\(
    \bigl|\abs{u}^{r}-\abs{v}^{r}\bigr|
    \le
    r\bigl(\abs{u}^{r-1}+\abs{v}^{r-1}\bigr)\abs{u-v}
\)
and H\"older's inequality imply
\begin{equation}\label{eq:chain-rule}
    \norm{\nabla_\varepsilon U}_{L^2(\pi)}
    \le
    Cq\norm{G_H}_{L^q(\pi)}^{r-1}
    \norm{\nabla_\varepsilon H}_{L^q(\pi)}.
\end{equation}
Moreover,
\[
    \norm{G_H}_{L^q(\pi)}^r
    =\norm{U}_{L^2(\pi)}
    \le
    \norm{U-\pi U}_{L^2(\pi)}
    +\norm{G_H}_{L^{q/2}(\pi)}^r.
\]
Since \(\norm{G_H}_\infty\le2\norm{H}_\infty\), equations
\eqref{eq:U-poincare}--\eqref{eq:chain-rule} yield
\ben{\label{eq:Xr}
    X^r
    \le
    AX^{r-1}+Y^r+Z^r,
}
where
\[
    X=\norm{G_H}_{L^q(\pi)},
    \qquad
    A=Cqn\norm{\nabla_\varepsilon H}_{L^q(\pi)},
    \qquad
    Y=\norm{G_H}_{L^{q/2}(\pi)},
    \qquad
    Z=Ce^{-cn/q}\norm{H}_\infty.
\]
If \(X\le2A\), the following estimate is immediate.  Otherwise, $AX^{r-1}<X^r/2$, and \cref{eq:Xr} implies
\[
    X
    \le
    Cqn\norm{\nabla_\varepsilon H}_{L^q(\pi)}
    +
    \left(1+\frac{C}{q}\right)Y
    +
    Ce^{-cn/q}\norm{H}_\infty.
\]
Iterating over \(q,q/2,q/4,\ldots,2\), and using monotonicity of the
\(L^s\)-norms, gives
\begin{equation}\label{eq:Lq-before-cutoff}
    \norm{H-\pi H}_{L^q(\pi)}
    \le
    Cqn\norm{\nabla_\varepsilon H}_{L^q(\pi)}
    +
    C(1+\log q)e^{-cn/q}\norm{H}_\infty.
\end{equation}
The case \(q=2\) follows directly from
\eqref{eq:no-log-defective-poincare}.  If
\(q\le2c_Mn/\log n\) and \(c_M\) is sufficiently small, then
\[
    (1+\log q)e^{-cn/q}\le n^{-M},
\]
which proves \eqref{eq:defective-Lq}.

\medskip
\noindent
\textbf{Step 2: exact order lowering under a Kawasaki exchange.}
In this step, we estimate $\|\nabla_\varepsilon f_{d,A}\|_{L^q(\pi)}$.
Let \(F:=f_{d,A}\).  For an occupied edge \(s\) and an unoccupied edge \(t\),
recall that \(x^{s\leftrightarrow t}\) is obtained from \(x\) by deleting \(s\) and
adding \(t\). Write
\(
    \nabla_{s,t}F(x)=F(x^{s\leftrightarrow t})-F(x).
\)
For a multi-affine function and a state with \(x_s=1\), \(x_t=0\),
\[
    F(x^{s\leftrightarrow t})-F(x)
    =
    -D_sF(x)+D_tF(x)-D_sD_tF(x).
\]
For \(r\in[d]\) and \(s\in\mathcal I\), let \(A_{r=s}\)
be the \((d-1)\)-tensor obtained by fixing the \(r\)-th
coordinate of \(A\) at \(s\). For distinct \(r,\ell\in[d]\) and \(s,t\in\mathcal I\),
the tensor \(A_{r=s,\ell=t}\) is obtained by simultaneously
fixing the \(r\)-th and \(\ell\)-th coordinates of the
original tensor at \(s\) and \(t\), respectively. The
remaining coordinates retain their original order.

Using \eqref{eq:appell-properties}, we obtain the exact identity
\begin{equation}\label{eq:exact-order-lowering}
\begin{aligned}
    \nabla_{s,t}F
    ={}&
    -\sum_{r=1}^d f_{d-1,A_{r=s}}
    +\sum_{r=1}^d f_{d-1,A_{r=t}}  
    -\sum_{\substack{1\le r,\ell\le d\\ r\ne\ell}}
    f_{d-2,A_{r=s,\ell=t}},
\end{aligned}
\end{equation}
where the last sum is absent when \(d=1\).  The sections satisfy
\begin{equation}\label{eq:section-identities}
    \sum_{r=1}^d\sum_{s\in\mathcal I}
    \norm{A_{r=s}}_2^2
    =
    d\norm{A}_2^2,
    \qquad
    \sum_{\substack{1\le r,\ell\le d\\r\ne\ell}}
    \sum_{s,t\in\mathcal I}
    \norm{A_{r=s,\ell=t}}_2^2
    =
    d(d-1)\norm{A}_2^2.
\end{equation}
Every legal restricted Kawasaki transition satisfies
\begin{equation}\label{eq:transition-upper}
    P_\varepsilon(x,x^{s\leftrightarrow t})
    \le
    \frac{C}{k(N-k)},
    \qquad
    k\asymp N,
    \qquad
    N-k\asymp N.
\end{equation}
%Indeed, this follows directly from the hard-wall definition, since the non-diagonal transition probabilities of \(P_\varepsilon\) agree with those of the full Kawasaki kernel; the two comparisons are uniform by \(\kappa_B\le p\le1-\kappa_B\).  
Consequently,
\begin{equation}\label{eq:gradient-square}
    \abs{\nabla_\varepsilon F}^2(x)
    \le
    \frac{C}{k(N-k)}
    \sum_{s:x_s=1}\sum_{t:x_t=0}
    \abs{\nabla_{s,t}F(x)}^2.
\end{equation}

Assume inductively that \eqref{eq:moment-bound} has been proved for every
order \(m<d\).  Since \(\mu(\Gamma)\ge1-Ce^{-cn}\),
\begin{equation}\label{eq:pi-mu-comparison}
    \norm{K}_{L^q(\pi)}
    \le
    C\norm{K}_{L^q(\mu)}
\end{equation}
for every function \(K\).  Apply the induction hypothesis with a defect exponent
\(M_0\ge M+d+3\), decreasing the final constant \(c_{d,M}\) if necessary so
that all lower-order induction hypotheses are available in the present range of
\(q\).  Together with \eqref{eq:supnorm-basic}, for every \(m<d\) and every coefficient
tensor \(B\),
\begin{equation}\label{eq:clean-induction}
    \norm{f_{m,B}}_{L^q(\pi)}
    \le
    (C_{d,M}q)^m\norm{B}_2
\end{equation}
for all sufficiently large \(n\).  Indeed, the defect term is bounded by
\[
    Cn^{-M_0}\norm{f_{m,B}}_\infty
    \le
    C_m n^{-M_0+m}\norm{B}_2,
\]
and is absorbed into the first term.

Using \eqref{eq:exact-order-lowering}, Minkowski's inequality in
\(L^q(\pi;\ell^2)\), and then
\eqref{eq:section-identities}--\eqref{eq:clean-induction}, we obtain
\begin{align}
    \norm{\nabla_\varepsilon F}_{L^q(\pi)}
    &\le
    \frac{C_d}{\sqrt{k(N-k)}}
    \Bigg[
    \sum_{r=1}^d
    \left(
    N\sum_s\norm{f_{d-1,A_{r=s}}}_{L^q(\pi)}^2
    \right)^{1/2}                                      \notag\\
    &\qquad
    +\sum_{r=1}^d
    \left(
    N\sum_t\norm{f_{d-1,A_{r=t}}}_{L^q(\pi)}^2
    \right)^{1/2}                                      \notag\\
    &\qquad
    +\sum_{r\ne\ell}
    \left(
    \sum_{s,t}
    \norm{f_{d-2,A_{r=s,\ell=t}}}_{L^q(\pi)}^2
    \right)^{1/2}
    \Bigg]                                               \notag\\
    &\le
    \frac{C_{d,M} q^{d-1}}{n}\norm{A}_2.
    \label{eq:gradient-bound}
\end{align}
Here the factor \(1/n\) comes from
\(
 \sqrt N/(\sqrt{k(N-k)})\asymp 1/n.
\)
The mixed second-difference terms are smaller, of order
\(N^{-1}(Cq)^{d-2}\norm{A}_2\).

\medskip
\noindent
\textbf{Step 3: completion of the induction and transfer to the full measure.}
Applying \eqref{eq:defective-Lq} to \(H=F\) and using
\eqref{eq:gradient-bound} gives
\begin{equation}\label{eq:restricted-moment}
    \norm{F-\pi F}_{L^q(\pi)}
    \le
    C_{d,M} q^d\norm{A}_2
    +
    C_{d,M}n^{-M}\norm{F}_\infty.
\end{equation}
By the centering property following \eqref{hoff}, \(\mu F=0\).  Hence
\[
    \abs{\pi F}
    =
    \frac{\abs{\mu(F\mathbbm{1}_{\Gamma^c})}}{\mu(\Gamma)}
    \le
    Ce^{-cn}\norm{F}_\infty.
\]
Moreover,
\[
    \norm{F}_{L^q(\mu)}
    \le
    C\norm{F}_{L^q(\pi)}
    +
    Ce^{-cn/q}\norm{F}_\infty.
\]
For \(q\le2c_{d,M}n/\log n\), decreasing \(c_{d,M}\) if necessary makes the exponential terms at most
\(
    C_{d,M}n^{-M}\norm{F}_\infty.
\)
This proves \eqref{eq:moment-bound} for dyadic \(q\).

For an arbitrary \(q\) in the stated range, choose a dyadic
\(q_0\in[q,2q)\).  The dyadic argument was proved up to twice the final range,
so it applies to \(q_0\).  Monotonicity of \(L^q\)-norms and \(q_0<2q\) then
prove \eqref{eq:moment-bound} for all admissible \(q\).  
%If \(1\le q<2\), the ``in particular'' estimate follows from the case \(q=2\) and monotonicity of \(L^q\)-norms.

\medskip
\noindent
\textbf{Step 4: optimization of the moment parameter.}
If \(\norm{A}_2=0\), then \(A=0\) and \(F=0\), so the tail estimate is trivial.
Assume from now on that \(\norm{A}_2>0\).  Let
\(
    K=C_{d,M}^d\norm{A}_2,
    R=C_{d,M}n^{-M}\norm{F}_\infty,
    q_{\max}=c_{d,M}\frac{n}{\log n}.
\)
By \eqref{eq:moment-bound}, for every \(2\le q\le q_{\max}\),
\(
    \mu\{\abs{F}\ge eKq^d+eR\}
    \le e^{-q}.
\)
If \(t\ge eK2^d\), choose
\(
    q
    :=
    \min\left\{
    \left(t/(eK)\right)^{1/d},
    q_{\max}
    \right\}.
\)
Then \(q\ge2\) and \(eKq^d\le t\), so
\[
    \mu\{\abs{F}\ge t+eR\}
    \le
    \exp\left\{
    -c_{d,M}
    \min\left[
    \left(\frac{t}{\norm{A}_2}\right)^{1/d},
    \frac{n}{\log n}
    \right]
    \right\}.
\]
For \(t<eK2^d\), the desired estimate is trivial after decreasing
\(c_{d,M}\), since the right-hand side can be made at least one.  Enlarging
\(C_{d,M}\) proves
\eqref{eq:tail-bound}.
\end{proof}

\iffalse
\begin{remark}
The proof uses only the weak Poincar\'e inequality
\eqref{eq:no-log-defective-poincare}.  Its $L^q$ consequence costs a factor
$q$ at each order-lowering step, which gives
\(
    \norm{f_{d,A}}_q\lesssim q^d\norm{A}_2
\)
and hence the exponent $1/d$ in \eqref{eq:tail-bound}.  An estimate of the
form
\(
    \norm{H-\pi H}_q
    \le
    C\sqrt q\,n\norm{\nabla_\varepsilon H}_q
    +\text{defect}
\)
would instead give $\norm{f_{d,A}}_q\lesssim q^{d/2}\norm{A}_2$ and the
stretched-exponential exponent $2/d$.  Such an estimate requires an
$O(n^2)$ log-Sobolev or comparable sub-Gaussian input and does not follow from
Theorem~\ref{lem:defective-poincare} alone. Such an estimate requires an \(O(n^2)\) log-Sobolev or comparable
sub-Gaussian input.  In the spectral range \(\beta<1\), such an input
follows from \cite[Theorem~1.3]{Bauerschmidt2025}.  It does not
follow from Theorem~\ref{lem:defective-poincare} alone, and the present
weak-Poincar\'e argument is needed for the portion of
\(\mathscr R_{\mathrm{Kaw}}\) beyond that spectral range.
\end{remark}
\fi

\section{Confinement estimates and proof of \cref{co-1}}
\label{sec:confinement}

In this section we prove the confinement estimate in \cref{co-1}, which follows the idea of  \cite[Lemma~14]{bhamidi2011mixing}.  The key point is that, near the boundary of the local
regularity region $\Gamma_{p,\varepsilon}$, the local edge-density has a drift pointing back toward
\(p\).  We first prove the following lemma. The desired \cref{co-1} then follows by iteration.

%A martingale estimate gives a one-block confinement estimate, and iteration yields confinement for exponentially long times.

\begin{lemma}\label{le-3}
Under the standing assumption \((p,\beta)\in B\subset
\mathscr R_{\mathrm{Kaw}}\)\footnote{With additional technicality by considering $\Gamma_{p,\eps}\cap \square_\eta^k(p)$, we might be able to enlarge the region $\mathscr R_{\mathrm{Kaw}}$ to $\mathscr R_{\mathrm{RS}}$ in \cref{le-3}, hence in \cref{co-1}. However, since such an improvement will not improve our main result, we do not pursue it here.}, there exist constants
\(\varepsilon_0=\varepsilon_0(B)>0\), \(\delta_0=\delta_0(B)\in(0,1/10)\),
and \(A=A(B)<\infty\), such that the following holds uniformly over
\((p,\beta)\in B\).  For every
\(0<\varepsilon\le\varepsilon_0\), there exists
\(c_{B,\varepsilon}>0\) and \(n_0=n_0(B,\varepsilon)<\infty\), depending only
on \(B\) and \(\varepsilon\), such that, for every  \(n\ge n_0\), with
\(
        \delta:=\delta_0\varepsilon ,
\)
if \((X_t)_{t\ge0}\) is the full Kawasaki dynamics with stationary measure
\(\mu_\beta^k\) and \(X_0\in \Gamma_{p,\varepsilon}\), then, with
\(T=\lfloor A n^2\rfloor\),
\[
        \mathbb P\left(
        X_t\in\Gamma_{p,\varepsilon+\delta}\ \text{for all }0\le t\le T,
        \ X_T\in\Gamma_{p,\varepsilon-\delta}
        \right)
        \ge 1-e^{-c_{B,\varepsilon} n}.
\]
%Equivalently, with probability at least \(1-e^{-c_{B,\varepsilon} n}\),
%\[
%        r_{\max}(X_t)\le p+\varepsilon+\delta,
%        \qquad
%        r_{\min}(X_t)\ge p-\varepsilon-\delta,
%        \qquad 0\le t\le T,
%\]
%and
%\[
%        r_{\max}(X_T)\le p+\varepsilon-\delta,
%        \qquad
%        r_{\min}(X_T)\ge p-\varepsilon+\delta .
%\]
\end{lemma}

\begin{proof}
Throughout this proof, we use \(C,c,\gamma\) to denote constants that may only depend on \(B\),
while \(c_{B,\varepsilon}\) may also depend on \(\varepsilon\). Let $\mathcal F_t:=\sigma(X_0,\ldots,X_t)$ be the natural filtration of the chain.

The proof is organized into five steps.  In Step~1, we compute the
one-step drift of \(r(e,X_t)\) for a fixed edge-position \(e\), reducing the
dynamics to a comparison between accepted exchanges that remove and add an
edge adjacent to \(e\).  Step~2 uses the margin
\(4\beta p(1-p)<1/2\) to turn this expansion into a uniform inward drift whenever
\(r(e,X_t)\) lies near either boundary of the regularity band.  In Step~3, we
combine this drift with the bounded jump size and conditional variance to
construct stopped exponential supermartingales, which show that a local
density is exponentially unlikely to make a bad outward excursion while it
remains in a boundary strip.  Step~4 applies these excursion estimates,
together with first-crossing decompositions and union bounds over all
edge-positions, to prove that the chain remains inside the enlarged band
\(\Gamma_{p,\varepsilon+\delta}\) throughout a block of length
\(T=\lfloor An^2\rfloor\).  Finally, Step~5 chooses \(A\) sufficiently large
so that the cumulative inward drift over one block dominates the width
\(\delta\), forcing the terminal configuration into the smaller band
\(\Gamma_{p,\varepsilon-\delta}\) with exponentially high probability.

\medskip
\noindent
\textbf{Step 1: One-step drift for the local edge density.}
We first derive a one-step drift estimate for the local density around a fixed
edge-position.  For two edge-positions \(f,e\), recall that we write \(f\sim e\) if
\(f\neq e\) and \(f\cap e\neq\varnothing\).  Let \(E_1=E_1(X)\) and
\(E_0=E_0(X)\) denote the sets of present and absent edges, respectively.  For
\(e'\in E_1\) and \(e''\in E_0\), let \(X^{e'\leftrightarrow e''}\) be the
graph obtained from \(X\) by deleting \(e'\) and adding \(e''\).  Set
\[
    \alpha_X(e',e'')
    =
    \frac{\mu_\beta^k(X^{e'\leftrightarrow e''})}
    {\mu_\beta^k(X)+\mu_\beta^k(X^{e'\leftrightarrow e''})}.
\]
This is the acceptance probability of the full Kawasaki chain.

Since an accepted exchange changes \(d(e,X)\) by
\(
    d(e,X^{e'\leftrightarrow e''})-d(e,X)
    =
    -[{\bf 1}_{\{e'\sim e\}}-{\bf 1}_{\{e''\sim e\}}]
\)
and $r(e,X)=d(e,X)/2n$,
we have
\begin{align}
    &\mathbb E\left[
        r(e,X_1)-r(e,X_0)\mid X_0=X
    \right]                                                   \notag\\
    &\qquad =
    -\frac{1}{2n|E_1||E_0|}
    \left[
        \sum_{\substack{e'\in E_1,\ e'\sim e\\ e''\in E_0}}
        \alpha_X(e',e'')
        -
        \sum_{\substack{e'\in E_1\\ e''\in E_0,\ e''\sim e}}
        \alpha_X(e',e'')
    \right],
    \label{eq:local-drift-revised}
\end{align}
where
\(
    |E_1|=Np,
    |E_0|=N(1-p).
\)

We next expand the acceptance probability.  Recall from \cref{eq:Hamilton} that
\(
    \mu_\beta^k(X)\propto e^{\mathcal H_\beta(X)},
    \mathcal H_\beta(X)=\frac{2\beta}{n}S_2(X).
\)
A single exchange satisfies, uniformly in \(e'\in E_1\) and \(e''\in E_0\),
\[
    S_2(X^{e'\leftrightarrow e''})-S_2(X)
    =
    d(e'',X)-d(e',X)+O(1).
\]
Therefore
\[
    \alpha_X(e',e'')
    =
    \frac{1}{1+\exp\left\{
        \frac{2\beta}{n}
        \bigl(d(e',X)-d(e'',X)\bigr)
        +O(n^{-1})
    \right\}} .
\]
If \(X\in\Gamma_{p,h}\), then
\(
    d(f,X)=2np+O(hn)
\)
uniformly in the edge-position \(f\).  Taylor expansion gives, uniformly in
\(e'\in E_1\) and \(e''\in E_0\),
\[
    \alpha_X(e',e'')
    =
    \frac12
    -
    \frac{\beta}{2n}
    \bigl(d(e',X)-d(e'',X)\bigr)
    +
    O(h^3+n^{-1}).
\]
In particular,
\begin{equation}\label{eq:alpha-upper-lower}
    \frac12-2\beta h-C(h^3+n^{-1})
    \le
    \alpha_X(e',e'')
    \le
    \frac12+2\beta h+C(h^3+n^{-1}) .
\end{equation}

Set
\(
    q_e=r(e,X)-p .
\)
By the definition of \(r(e,X)\),
\[
    |\{e'\in E_1:e'\sim e\}|=2n(p+q_e).
\]
Since the total number of edge-positions adjacent to \(e\) is \(2(n-2)\), we
also have
\[
    |\{e''\in E_0:e''\sim e\}|
    =
    2n(1-p-q_e)+O(1).
\]
The \(O(1)\) error in this count will be absorbed into the \(Cn^{-1}\) term
below.

Define
\[
    S_-(e):=
    \sum_{\substack{e'\in E_1,\ e'\sim e\\ e''\in E_0}}
    \alpha_X(e',e''),
    \qquad
    S_+(e):=
    \sum_{\substack{e'\in E_1\\ e''\in E_0,\ e''\sim e}}
    \alpha_X(e',e'').
\]
Using \eqref{eq:alpha-upper-lower}, we obtain
\[
\begin{aligned}
    S_-(e)
    &\ge
    Nn(p+q_e)(1-p)
    -
    4\beta Nn(p+q_e)(1-p)h
    -
    C Nn(h^3+n^{-1}),                                      \\
    S_+(e)
    &\le
    Nnp(1-p-q_e)
    +
    4\beta Nnp(1-p-q_e)h
    +
    C Nn(h^3+n^{-1}).
\end{aligned}
\]
Consequently, using \(|q_e|\le h\),
\begin{equation}\label{eq:Sdiff-lower}
    S_-(e)-S_+(e)
    \ge
    Nn
    \left[
        q_e-8\beta p(1-p)h-Ch^2-Cn^{-1}
    \right].
\end{equation}
The same argument, with the inequalities reversed, gives
\begin{equation}\label{eq:Sdiff-upper}
    S_-(e)-S_+(e)
    \le
    Nn
    \left[
        q_e+8\beta p(1-p)h+Ch^2+Cn^{-1}
    \right].
\end{equation}
\medskip
\noindent
\textbf{Step 2: Boundary drift estimates.}
Let
\(
    \theta_0=1-8\beta p(1-p) .
\)
By the definition of \(\kappa_B\) in \cref{eq:kappaB},
\[
    \theta_0
    =
    2\left(\frac12-4\beta p(1-p)\right)
    \ge 2\kappa_B.
\]
Choose
\(
    0<\delta_0=\delta_0(B)\le
    \min\left\{1/10,\kappa_B/6\right\},
\)
and then choose \(\varepsilon_0=\varepsilon_0(B)>0\) sufficiently small.
Throughout the proof we take
\(
    0<\varepsilon\le\varepsilon_0,
    \delta=\delta_0\varepsilon,
    h=\varepsilon+\delta .
\)

Suppose first that \(X\in\Gamma_{p,h}\) and
\(
    r(e,X)\ge p+\varepsilon-2\delta .
\)
Then \(q_e\ge \varepsilon-2\delta\).  Writing
\(
    a=8\beta p(1-p)\le1-2\kappa_B,
\)
\eqref{eq:Sdiff-lower} gives, for
\(n\ge n_0(B,\varepsilon)\),
\[
\begin{aligned}
    S_-(e)-S_+(e)
    &\ge
    Nn\left[2\kappa_B\varepsilon-3\delta
    -C_Bh^2-C_Bn^{-1}\right] \ge \frac{\kappa_B}{2}Nn\varepsilon .
\end{aligned}
\]
In particular, after decreasing a constant depending only on \(B\),
\(
    S_-(e)-S_+(e)\ge c_BNn\varepsilon .
\)
Substituting this estimate into \eqref{eq:local-drift-revised} yields
\begin{equation}\label{eq:upper-boundary-drift}
    \mathbb E\left[
        r(e,X_1)-r(e,X_0)\mid X_0=X
    \right]
    \le
    -\frac{\gamma\varepsilon}{n^2}.
\end{equation}

Similarly, if \(X\in\Gamma_{p,h}\) and
\(
    r(e,X)\le p-\varepsilon+2\delta ,
\)
then \(q_e\le -\varepsilon+2\delta\).  The same calculation using
\eqref{eq:Sdiff-upper} gives the preceding bound with the sign reversed, and
hence
\begin{equation}\label{eq:lower-boundary-drift}
    \mathbb E\left[
        r(e,X_1)-r(e,X_0)\mid X_0=X
    \right]
    \ge
    \frac{\gamma\varepsilon}{n^2}.
\end{equation}
Thus the drift is uniformly inward near both boundaries.
\medskip

\noindent
\textbf{Step 3: Exponential bound for bad excursions.}
Set
\(
    R_+=p+\varepsilon,
    R_-=p-\varepsilon,
    T=\lfloor A n^2\rfloor ,
\)
where \(A>0\) will be chosen later below \cref{eq:choiceA}.  Define
\[
    A_t(\delta)
    :=
    \left\{
        r_{\max}(X_t)\le R_++\delta,
        \quad
        r_{\min}(X_t)\ge R_--\delta
    \right\}.
\]
For a fixed edge-position \(e\), put
\[
    D_t^+(e,\delta)
    :=
    A_t(\delta)\cap
    \{R_+-2\delta\le r(e,X_t)\le R_++\delta\},
\]
and
\[
    D_t^-(e,\delta)
    :=
    A_t(\delta)\cap
    \{R_--\delta\le r(e,X_t)\le R_-+2\delta\}.
\]
For \(0\le t_1<t_2\le T\), define
\[
    B_{t_1,t_2}^+(e,\delta)
    :=
    \left(\bigcap_{t_1\le t<t_2}D_t^+(e,\delta)\right)
    \cap
    \left\{
        r(e,X_{t_2})-r(e,X_{t_1})>\frac{\delta}{2}
    \right\}.
\]
Similarly, define
\[
    B_{t_1,t_2}^-(e,\delta)
    :=
    \left(
        \bigcap_{t_1\le t<t_2}D_t^-(e,\delta)
    \right)
    \cap
    \left\{
        r(e,X_{t_2})-r(e,X_{t_1})
        <-\frac{\delta}{2}
    \right\}.
\]

We claim that, for every fixed edge-position \(e\),
\begin{equation}\label{eq:bad-excursion-bound}
    \mathbb P\left(
        \bigcup_{0\le t_1<t_2\le T}
        B_{t_1,t_2}^+(e,\delta)
    \right)
    +
    \mathbb P\left(
        \bigcup_{0\le t_1<t_2\le T}
        B_{t_1,t_2}^-(e,\delta)
    \right)
    \le e^{-c_{B,\varepsilon} n}.
\end{equation}

We prove the estimate for the upper boundary; the lower boundary is identical
after replacing \(r(e,X_s)\) by \(-r(e,X_s)\).  On \(D_t^+(e,\delta)\), the
drift of \(r(e,X_t)\) is at most \(-\gamma\varepsilon/n^2\) from \cref{eq:upper-boundary-drift}.  Moreover, one
Kawasaki exchange changes \(r(e,\cdot)\) by at most \(C/n\).  A non-zero
change can occur only if one of the two exchanged edges is adjacent to \(e\),
which has conditional probability \(O(n^{-1})\).  Hence, uniformly in the
current state,
\ben{\label{eq:driftsquare}
    \mathbb E\left[
        \bigl(r(e,X_{t+1})-r(e,X_t)\bigr)^2
        \,\bigm|\, X_t
    \right]
    \le
    \frac{C}{n^3}.
}

The drift estimate is valid only while the trajectory remains in
\(D_s^+(e,\delta)\).  We therefore stop the process when it leaves this set.
Fix \(0\le t_1<t_2\le T\), and write
\(
    \Delta_s=r(e,X_{s+1})-r(e,X_s),
    \mu=\gamma\varepsilon/n^2.
\)
Let
\(
    \tau^+
    =
    \inf\{s\in\{t_1,\ldots,t_2-1\}:D_s^+(e,\delta)^c\}\wedge t_2
\)
and define
\(
    \Delta_s^+ = \Delta_s{\bf 1}_{\{s<\tau^+\}},t_1\le s<t_2.
\)
Then \({\bf 1}_{\{s<\tau^+\}}\) is \(\mathcal F_s\)-measurable, and by
\eqref{eq:upper-boundary-drift} and \cref{eq:driftsquare},
\[
    \mathbb E(\Delta_s^+\mid \mathcal F_s)
    \le
    -\mu{\bf 1}_{\{s<\tau^+\}},
    \qquad
    |\Delta_s^+|\le \frac Cn,
    \qquad
    \mathbb E((\Delta_s^+)^2\mid \mathcal F_s)
    \le \frac C{n^3}{\bf 1}_{\{s<\tau^+\}} .
\]

Let \(\lambda=c_0\varepsilon n\), where \(c_0>0\) is sufficiently small,
depending only on \(B\).  By decreasing \(\varepsilon_0\), we
may assume \(|\lambda\Delta_s^+|\le 1\).  Taylor expansion of the exponential
gives
\[
\begin{aligned}
    \mathbb E\left[
        \exp\left\{
            \lambda\left(
                \Delta_s^+
                +\frac{\mu}{2}{\bf 1}_{\{s<\tau^+\}}
            \right)
        \right\}
        \,\bigm|\, \mathcal F_s
    \right]
    &\le
    \exp\left\{
        \left[
            -\frac{\lambda\mu}{2}
            +C\lambda^2n^{-3}
        \right]
        {\bf 1}_{\{s<\tau^+\}}
    \right\} \le 1 .
\end{aligned}
\]
In the last step we used the choice of \(c_0\).  Therefore
\[
    M_j^+
    :=
    \exp\left\{
        \lambda
        \sum_{s=t_1}^{j-1}
        \left(
            \Delta_s^+
            +\frac{\mu}{2}{\bf 1}_{\{s<\tau^+\}}
        \right)
    \right\},
    \qquad t_1\le j\le t_2,
\]
is a non-negative supermartingale.

On the event \(B_{t_1,t_2}^+(e,\delta)\), we have \(\tau^+=t_2\) and
\(
    \sum_{s=t_1}^{t_2-1}\Delta_s>\delta/2.
\)
Hence
\(
    M_{t_2}^+
    \ge
    \exp\left\{\lambda\delta/2\right\}.
\)
By Markov's inequality,
\[
    \mathbb P\bigl(B_{t_1,t_2}^+(e,\delta)\bigr)
    \le
    \exp\left\{-\frac{\lambda\delta}{2}\right\}
    \le
    \exp\{-c\varepsilon^2 n\},
\]
because \(\delta=\delta_0\varepsilon\).  The lower-boundary estimate follows in
the same way, using \eqref{eq:lower-boundary-drift} and the stopped increments
associated with \(D_s^-(e,\delta)\).

Since \(T=O_A(n^2)\), there are at most \(O_A(n^4)\) choices of
\((t_1,t_2)\).  Thus, after a union bound over all intervals and after
decreasing the constant in the exponent,
\(
    O_A(n^4)e^{-c\varepsilon^2n}
    \le e^{-c_{B,\varepsilon} n}
\)
for all sufficiently large \(n\).  This proves
\eqref{eq:bad-excursion-bound}.
\medskip

\noindent
\textbf{Step 4: Persistence of the regularity band.}
We now show that the trajectory remains in the local regularity band with high
probability.  Let
\[
    \tau=\inf\{0\le t\le T:A_t(\delta)^c\}.
\]
Suppose first that \(\tau\le T\) and
\(
    r_{\max}(X_\tau)>R_++\delta .
\)
Choose an edge-position \(e\) for which the upper boundary is crossed, and let
\[
    t_2:=\inf\{0\le t\le \tau:r(e,X_t)>R_++\delta\}.
\]
Since \(X_0\in\Gamma_{p,\varepsilon}\), we have \(r(e,X_0)\le R_+\), and
therefore \(t_2\ge1\).

If
\(
    r(e,X_t)\ge R_+-2\delta\;
    \text{for all }0\le t<t_2,
\)
set \(t_1=0\).  Otherwise, let
\[
    s_0:=\max\{0\le t<t_2:r(e,X_t)<R_+-2\delta\},
    \qquad
    t_1:=s_0+1 .
\]
Since one step changes \(r(e,\cdot)\) by at most \(C/n\), for all sufficiently
large \(n\) we have \(t_1<t_2\).  By the definitions of \(t_1\) and \(t_2\),
the event \(D_t^+(e,\delta)\) holds for every \(t_1\le t<t_2\).  Moreover, if
\(t_1=0\), then \(r(e,X_{t_1})\le R_+\), and hence
\[
    r(e,X_{t_2})-r(e,X_{t_1})>\delta .
\]
If \(t_1>0\), then by the definition of \(s_0\) and by the one-step bound,
\[
    r(e,X_{t_1})\le R_+-2\delta+\frac Cn .
\]
Since \(r(e,X_{t_2})>R_++\delta\), it follows that
\[
    r(e,X_{t_2})-r(e,X_{t_1})
    >
    3\delta-\frac Cn
    >
    \frac{\delta}{2}
\]
for all sufficiently large \(n\).  Hence \(B_{t_1,t_2}^+(e,\delta)\) occurs.

By \eqref{eq:bad-excursion-bound}, and then by a union bound over all
edge-positions,
\[
    \mathbb P\left(
        \tau\le T,\ r_{\max}(X_\tau)>R_++\delta
    \right)
    \le e^{-c_{B,\varepsilon} n}.
\]
The lower boundary is treated in exactly the same way, using
\(B^-_{t_1,t_2}(e,\delta)\).  Therefore
\begin{equation}\label{eq:regularity-band-high-probability}
    \mathbb P\left(
        A_t(\delta)\ \text{holds for all }0\le t\le T
    \right)
    \ge 1-e^{-c_{B,\varepsilon} n}.
\end{equation}
\medskip
\noindent
\textbf{Step 5: Inward movement after time \(T\).}
It remains to prove that the upper and lower envelopes move inward after
\(T=\lfloor A n^2\rfloor\) steps.  Let
\[
    \mathcal A_T
    =
    \{A_t(\delta)\text{ holds for every }0\le t\le T\}.
\]
By \eqref{eq:regularity-band-high-probability},
\(
    \mathbb P(\mathcal A_T^c)\le e^{-c_{B,\varepsilon} n}.
\)

We first treat the upper envelope.  Fix an edge-position \(e\), and set
\[
    E_e^+=\{r(e,X_T)>R_+-\delta\}.
\]
We split \(E_e^+\) according to whether the trajectory ever enters the lower
part of the upper boundary strip:
\[
\begin{aligned}
E_e^+
\subseteq\;
\mathcal A_T^c
&\cup
\Bigl(
    \mathcal A_T\cap E_e^+
    \cap
    \{r(e,X_t)\ge R_+-2\delta\text{ for all }0\le t<T\}
\Bigr)                                                               \\
&\cup
\Bigl(
    \mathcal A_T\cap E_e^+
    \cap
    \{r(e,X_t)<R_+-2\delta\text{ for some }0\le t<T\}
\Bigr).
\end{aligned}
\]
The first event has probability at most \(e^{-c_{B,\varepsilon} n}\).

On the second event, \(D_t^+(e,\delta)\) holds for all \(0\le t<T\).  We repeat
the stopped supermartingale construction from Step 3 with \(t_1=0\) and
\(t_2=T\).  Namely, set
\[
    \Delta_s=r(e,X_{s+1})-r(e,X_s),
    \qquad
    \mu=\frac{\gamma\varepsilon}{n^2},
\]
and define
\[
    \tau_T^+
    :=
    \inf\{0\le s<T:D_s^+(e,\delta)^c\}\wedge T,
    \qquad
    \Delta_s^{+,T}:=\Delta_s{\bf 1}_{\{s<\tau_T^+\}} .
\]
The same argument as in Step 3 shows that
\[
    M_j^{+,T}
    :=
    \exp\left\{
        \lambda
        \sum_{s=0}^{j-1}
        \left(
            \Delta_s^{+,T}
            +\frac{\mu}{2}{\bf 1}_{\{s<\tau_T^+\}}
        \right)
    \right\},
    \qquad 0\le j\le T,
\]
is a non-negative supermartingale.

On the second event, \(\tau_T^+=T\).  Since
\(X_0\in\Gamma_{p,\varepsilon}\), we have \(r(e,X_0)\le R_+\).  Therefore
\ben{\label{eq:choiceA}
    \sum_{s=0}^{T-1}
    \left(
        \Delta_s+\frac{\mu}{2}
    \right)
    >
    -\delta+\frac{\mu T}{2}.
}
Choose \(A\) sufficiently large so that
\(
    \mu T/2
    =
    \gamma\varepsilon T/2n^2
    \ge 2\delta
\)
for all sufficiently large \(n\).  Since \(\delta=\delta_0\varepsilon\), this
choice of \(A\) depends only on \(B\).  The last display then
implies
\(
    M_T^{+,T}\ge e^{\lambda\delta}.
\)
By the supermartingale property and Markov's inequality,
\[
    \mathbb P(M_T^{+,T}\ge e^{\lambda\delta})
    \le e^{-\lambda\delta}
    \le e^{-c_{B,\varepsilon} n}.
\]
Thus the second event has probability at most \(e^{-c_{B,\varepsilon} n}\).

It remains to consider the third event.  On this event, let
\[
    s_0:=\max\{0\le s<T:r(e,X_s)<R_+-2\delta\},
    \qquad
    t_1:=s_0+1 .
\]
Since one Kawasaki step changes \(r(e,\cdot)\) by at most \(C/n\), for all
sufficiently large \(n\),
\[
    r(e,X_{t_1})\le R_+-2\delta+\frac Cn .
\]
Moreover, by the definition of \(s_0\) and by the event \(\mathcal A_T\),
\(D_s^+(e,\delta)\) holds for every \(t_1\le s<T\).  Since
\(r(e,X_T)>R_+-\delta\), we also have
\[
    r(e,X_T)-r(e,X_{t_1})
    >
    \delta-\frac Cn
    >
    \frac{\delta}{2}
\]
for all sufficiently large \(n\).  Hence \(B_{t_1,T}^+(e,\delta)\) occurs.
By \eqref{eq:bad-excursion-bound}, the third event also has probability at
most \(e^{-c_{B,\varepsilon} n}\).

Combining the three estimates gives
\[
    \mathbb P\{r(e,X_T)>R_+-\delta\}
    \le e^{-c_{B,\varepsilon} n}.
\]
Taking a union bound over all edge-positions and absorbing the polynomial
factor into the exponential, we obtain
\[
    \mathbb P\{r_{\max}(X_T)>R_+-\delta\}
    \le e^{-c_{B,\varepsilon} n}.
\]

The lower envelope is handled in the same way, using
\eqref{eq:lower-boundary-drift} and the bad downward excursion events
\(B^-_{t_1,t_2}(e,\delta)\).  Thus
\[
    \mathbb P\{r_{\min}(X_T)<R_-+\delta\}
    \le e^{-c_{B,\varepsilon} n}.
\]

Together with \eqref{eq:regularity-band-high-probability}, this proves that,
with probability at least \(1-e^{-c_{B,\varepsilon} n}\),
\[
    A_t(\delta)\ \text{holds for all }0\le t\le T,
\]
and
\[
    r_{\max}(X_T)\le R_+-\delta,
    \qquad
    r_{\min}(X_T)\ge R_-+\delta .
\]
This proves the claim.
\end{proof}

Iterating \cref{le-3}, we complete the proof of the confinement estimate.

\begin{proof}[Proof of \cref{co-1}]
We call $\varepsilon$ in $\Gamma_{p,\varepsilon}$ the \emph{radius}.
Apply \cref{le-3} with radius \(\varepsilon/4\), and set
\(
        T_0=\lfloor A n^2\rfloor 
\)
with $A$ as in the statement of \cref{le-3}.
If a block of the full Kawasaki chain for time duration $T_0$ starts from \(\Gamma_{p,\varepsilon/4}\), then, except on an
event of probability at most \(e^{-c_{B,\varepsilon/4}n}\), the chain stays
during the whole block in
\(
        \Gamma_{p,\,\varepsilon/4+\delta_0\varepsilon/4}
        \subset \Gamma_{p,\varepsilon/2},
\)
and the endpoint belongs to
\(
        \Gamma_{p,\,\varepsilon/4-\delta_0\varepsilon/4}
        \subset \Gamma_{p,\varepsilon/4}.
\)
Thus the estimate can be iterated.  Taking a union bound
over at most \(e^{\alpha_{B,\varepsilon} n}/T_0+1\) blocks and choosing
\(\alpha_{B,\varepsilon}<c_{B,\varepsilon/4}/2\), we obtain
\[
        \mathbb P\left(
        X_t\in\Gamma_{p,\varepsilon/2}
        \text{ for all }0\le t\le e^{\alpha_{B,\varepsilon} n}
        \right)
        \ge 1-e^{-c_{B,\varepsilon} n}.
\]

For the second assertion, start from radius \(\varepsilon\) and apply
\cref{le-3}.
After \(m\) successive blocks the radius is at most
\((1-\delta_0)^m\varepsilon\) except for an exponentially small probability.  Choose \(m\) so that this is at most
\(\varepsilon/2\).  A union bound over these finitely many blocks gives
\[
        \mathbb P\left(
        X_{mT_0}\in\Gamma_{p,\varepsilon/2}
        \right)
        \ge 1-e^{-c_{B,\varepsilon} n}.
\]
Taking \(T_n=mT_0\), we have \(T_n\le Cn^2\) for a constant
\(C=C(B,\varepsilon)\), which proves the finite-time estimate.
\end{proof}

\section{Path coupling estimates and proof of \cref{le-5}}
\label{sec:path-coupling}

In this section we prove the contractive coupling estimate in
\cref{le-5}.  The main point is to control the difference of the
acceptance probabilities directly in terms of the disagreement edges
between the two configurations.

Define the following proposal coupling for the full Kawasaki dynamics. Given
\((X_t,Y_t)\), choose \(e_1\in E_1(X_t)\) and \(e_2\in E_0(X_t)\)
independently and uniformly at random, and set
\(
        \widehat X=X_t^{e_1\leftrightarrow e_2}.
\)
If \(e_1\in E_1(X_t)\cap E_1(Y_t)\), we use the same deleted
edge \(e_1\) for \(Y_t\); otherwise we choose \(e_3\) uniformly from
\(E_1(Y_t)\setminus E_1(X_t)\). Similarly, if
\(e_2\in E_0(X_t)\cap E_0(Y_t)\), we use the same added edge \(e_2\) for
\(Y_t\); otherwise we choose \(e_4\) uniformly from
\(E_0(Y_t)\setminus E_0(X_t)\). More explicitly,
\(
        \widehat Y
        =
        Y_t^{\,f_1\leftrightarrow f_2},
\)
where
\[
        f_1=
        \begin{cases}
        e_1, & e_1\in E_1(X_t)\cap E_1(Y_t),\\
        e_3, & e_1\in E_1(X_t)\setminus E_1(Y_t),
        \end{cases}
        \qquad
        f_2=
        \begin{cases}
        e_2, & e_2\in E_0(X_t)\cap E_0(Y_t),\\
        e_4, & e_2\in E_0(X_t)\setminus E_0(Y_t).
        \end{cases}
\]
Here \(e_3\) is chosen uniformly from \(E_1(Y_t)\setminus E_1(X_t)\) whenever
it is needed, and \(e_4\) is chosen uniformly from
\(E_0(Y_t)\setminus E_0(X_t)\) whenever it is needed. If both are needed, they
are chosen independently. Since \(X_t\) and \(Y_t\) have the same number of
edges, the relevant difference sets have the same cardinality. Hence
\(\widehat Y\) has the marginal distribution of a uniform Kawasaki proposal
from \(Y_t\).

Let
\(
        p_X=\alpha(X_t,\widehat X),
        p_Y=\alpha(Y_t,\widehat Y),
\)
where \(\alpha\) denotes the  acceptance probability,
that is,
\[
        \alpha(Z,Z')
        =
        \frac{\mu_\beta^k(Z')}
        {\mu_\beta^k(Z)+\mu_\beta^k(Z')}
\]
whenever \(Z'\) is obtained from \(Z\) by one legal Kawasaki exchange. We
couple the two acceptance decisions maximally: both proposals are accepted with
probability \(\min\{p_X,p_Y\}\), both are rejected with probability
\(\min\{1-p_X,1-p_Y\}\), and on the remaining probability mass exactly one
proposal is accepted, so that the two marginals are Bernoulli\((p_X)\) and
Bernoulli\((p_Y)\), respectively.

\begin{proof}[Proof of \cref{le-5}]
All constants below are uniform over \((p,\beta)\in B\). Let
\(\varepsilon=\varepsilon_0(B)>0\) be chosen below. Fix
%\(0<\varepsilon\le\varepsilon_0\) and
\(x,y\in\Gamma_{p,\varepsilon}\). Write
\[
    D=E_1(x)\setminus E_1(y),\qquad
    \widetilde D=E_1(y)\setminus E_1(x),\qquad
    r=|D|=|\widetilde D|.
\]
If \(r=0\), there is nothing to prove. Set
\[
    C_1=E_1(x)\cap E_1(y),\qquad
    C_0=E_0(x)\cap E_0(y),
\]
so that
\[
    |C_1|=Np-r,\qquad |C_0|=N(1-p)-r.
\]

Let \((\widehat x,\widehat y)\) be the coupled proposals defined above,
and let \((X',Y')\) be the states after the coupled acceptance decisions.
Set
\[
    D'=E_1(X')\setminus E_1(Y').
\]
For fixed \(x,y,e_1,e_2\), define
\[
    b(e_1,e_2)
    :=
    \mathbb E\left[
        \min\{\alpha(x,\widehat x),\alpha(y,\widehat y)\}
        \,\middle|\,x,y,e_1,e_2
    \right]
\]
and
\[
    o(e_1,e_2)
    :=
    \mathbb E\left[
        |\alpha(x,\widehat x)-\alpha(y,\widehat y)|
        \,\middle|\,x,y,e_1,e_2
    \right],
\]
where the conditional expectations are only over the remaining random
choices \(e_3,e_4\), when needed.

If \((e_1,e_2)\in D\times C_0\), simultaneous acceptance decreases
\(|D|\) by \(1\). The same holds if
\((e_1,e_2)\in C_1\times\widetilde D\). If
\((e_1,e_2)\in C_1\times C_0\), unequal acceptances can increase
\(|D|\) by at most \(1\). Finally, for
\((e_1,e_2)\in D\times\widetilde D\), define
\[
\begin{aligned}
    \widetilde b(e_1,e_2)
    &:=
    \mathbb E_{e_3,e_4}\Big[
        \min\{\alpha(x,\widehat x),\alpha(y,\widehat y)\}\\
    &\hspace{20mm}\times
        \bigl(
        2-\mathbf 1_{\{e_3=e_2\}}
          -\mathbf 1_{\{e_4=e_1\}}
        \bigr)
    \Big],
\end{aligned}
\]
where \(e_3\) and \(e_4\) are chosen independently and uniformly from
\(\widetilde D\) and \(D\), respectively. Under simultaneous acceptance,
\[
    |D'|-|D|
    =
    -2+\mathbf 1_{\{e_3=e_2\}}
       +\mathbf 1_{\{e_4=e_1\}}.
\]
If exactly one proposal is accepted, then \(|D|\) decreases by \(1\).
Discarding this additional favorable contribution, we obtain
\begin{align}
    \mathbb E\bigl[|D'|\mid x,y\bigr]
    \le\;& r
    -\frac{1}{k(N-k)}
      \sum_{e_1\in D}\sum_{e_2\in C_0}b(e_1,e_2)
    \notag\\
    &-\frac{1}{k(N-k)}
      \sum_{e_1\in C_1}\sum_{e_2\in\widetilde D}b(e_1,e_2)
    \notag\\
    &-\frac{1}{k(N-k)}
      \sum_{e_1\in D}\sum_{e_2\in\widetilde D}
      \widetilde b(e_1,e_2)
    +\frac{1}{k(N-k)}
      \sum_{e_1\in C_1}\sum_{e_2\in C_0}o(e_1,e_2).
    \label{eq:Dt-contraction-start}
\end{align}

For a configuration \(Z\) and a legal exchange
\(Z^{s\leftrightarrow t}\), we have
\[
    \alpha(Z,Z^{s\leftrightarrow t})
    =
    \ell\left(
        \frac{2\beta}{n}
        \{d(t,Z)-d(s,Z)-\mathbf 1_{\{s\sim t\}}\}
    \right),
\]
where
\[
    \ell(u)=\frac{e^u}{1+e^u}.
\]
Since \(x,y\in\Gamma_{p,\varepsilon}\), the argument of \(\ell\)
is \(O_B(\varepsilon+n^{-1})\) for every acceptance probability
appearing above. Since \(\ell(0)=1/2\) and
\(0<\ell'(u)\le1/4\), after increasing \(C_B\) all such probabilities
are bounded below by
\[
    a_\varepsilon
    :=
    \frac12-C_B(\varepsilon+n^{-1}).
\]
Therefore, for \((e_1,e_2)\in D\times\widetilde D\),
\[
\begin{aligned}
    \widetilde b(e_1,e_2)
    &\ge
    a_\varepsilon
    \mathbb E_{e_3,e_4}
    \bigl[
        2-\mathbf 1_{\{e_3=e_2\}}
          -\mathbf 1_{\{e_4=e_1\}}
    \bigr] \\
    &=
    a_\varepsilon\left(2-\frac2r\right).
\end{aligned}
\]
Hence the total negative numerator in
\eqref{eq:Dt-contraction-start} is at least
\[
\begin{aligned}
    a_\varepsilon
    \Big[
        r\{N(1-p)-r\}
        +(Np-r)r
        +\left(2-\frac2r\right)r^2
    \Big]
    =
    a_\varepsilon r(N-2).
\end{aligned}
\]
Since \(k=Np\) and \(N-k=N(1-p)\), the negative contribution is at most
\[
    -\frac{a_\varepsilon r(N-2)}{N^2p(1-p)}
    =
    -\frac{r}{Np(1-p)}
    \left(1-\frac2N\right)a_\varepsilon.
\]
Hence, after changing \(C_B\),
\begin{equation}\label{eq:negative-contribution}
    -\frac{r}{Np(1-p)}
    \left(
        \frac12-C_B(\varepsilon+n^{-1})
    \right)
\end{equation}
is an upper bound for the negative contribution.

It remains to bound the positive term in
\eqref{eq:Dt-contraction-start}.  For
\((e_1,e_2)\in C_1\times C_0\), no additional choice
\(e_3,e_4\) is needed, and hence
\[
    o(e_1,e_2)
    =
    \left|
        \alpha(x,x^{e_1\leftrightarrow e_2})
        -
        \alpha(y,y^{e_1\leftrightarrow e_2})
    \right|.
\]
Using
\[
    \alpha(Z,Z^{e_1\leftrightarrow e_2})
    =
    \ell\left(
        \frac{2\beta}{n}
        \{d(e_2,Z)-d(e_1,Z)-\mathbf 1_{\{e_1\sim e_2\}}\}
    \right),
    \qquad
    \ell(u)=\frac{e^u}{1+e^u},
\]
and \(0<\ell'(u)\le1/4\), we obtain
\[
\begin{aligned}
    o(e_1,e_2)
    \le
    \frac{\beta}{2n}
    \Big(
        |d(e_1,x)-d(e_1,y)|
        +
        |d(e_2,x)-d(e_2,y)|
    \Big).
\end{aligned}
\]

For every edge-position \(h\),
\[
\begin{aligned}
    d(h,x)-d(h,y)
    &=
    \sum_{f\in D}\mathbf 1_{\{f\sim h\}}
    -
    \sum_{f\in\widetilde D}\mathbf 1_{\{f\sim h\}},
\end{aligned}
\]
and hence
\[
    |d(h,x)-d(h,y)|
    \le
    \sum_{f\in D\cup\widetilde D}\mathbf 1_{\{f\sim h\}}.
\]
Since \(C_1\subseteq E_1(x)\) and \(x\in\Gamma_{p,\varepsilon}\),
\[
\begin{aligned}
    \sum_{e_1\in C_1}|d(e_1,x)-d(e_1,y)|
    &\le
    \sum_{f\in D\cup\widetilde D}
    |\{e_1\in C_1:e_1\sim f\}| \\
    &\le
    \sum_{f\in D\cup\widetilde D} d(f,x)
    \le
    4rn(p+\varepsilon).
\end{aligned}
\]
Similarly, since \(C_0\subseteq E_0(x)\),
\[
\begin{aligned}
    \sum_{e_2\in C_0}|d(e_2,x)-d(e_2,y)|
    &\le
    \sum_{f\in D\cup\widetilde D}
    |\{e_2\in C_0:e_2\sim f\}| \\
    &\le
    \sum_{f\in D\cup\widetilde D}
    \{2(n-2)-d(f,x)\} \\
    &\le
    4rn(1-p+\varepsilon).
\end{aligned}
\]
Therefore, using
\(|C_1|\le Np\) and \(|C_0|\le N(1-p)\),
\[
\begin{aligned}
    \sum_{e_1\in C_1}\sum_{e_2\in C_0}o(e_1,e_2)
    &\le
    \frac{\beta}{2n}
    \Bigg[
        |C_0|
        \sum_{e_1\in C_1}|d(e_1,x)-d(e_1,y)| \\
    &\hspace{23mm}
        +
        |C_1|
        \sum_{e_2\in C_0}|d(e_2,x)-d(e_2,y)|
    \Bigg] \\
    &\le
    Nr\bigl(
        4\beta p(1-p)+2\beta\varepsilon
    \bigr).
\end{aligned}
\]
Since \(k(N-k)=N^2p(1-p)\), the positive contribution is at most
\begin{equation}\label{eq:positive-contribution}
    \frac{r}{Np(1-p)}
    \left(
        4\beta p(1-p)+O_B(\varepsilon)
    \right).
\end{equation}

Combining \eqref{eq:negative-contribution},
\eqref{eq:positive-contribution}, and
\eqref{eq:Dt-contraction-start}, we get
\[
\begin{aligned}
    \mathbb E\bigl[|D'|\mid x,y\bigr]
    &\le
    r
    -
    \frac{r}{Np(1-p)}
    \left(
        \frac12-4\beta p(1-p)-C_B(\varepsilon+n^{-1})
    \right).
\end{aligned}
\]
By \cref{eq:kappaB},
\[
    \frac12-4\beta p(1-p)\ge\kappa_B.
\]
Thus, by choosing \(\varepsilon=\varepsilon_0(B)\) sufficiently small
and then \(n\) sufficiently large,
\[
    \frac12-4\beta p(1-p)-C_B(\varepsilon+n^{-1})
    \ge \frac{\kappa_B}{2}.
\]
Since \(N\asymp n^2\) and \(p(1-p)\) is uniformly bounded away from zero
on \(B\), there is a constant \(c_B>0\) such that
\[
    \mathbb E\bigl[|D'|\mid x,y\bigr]
    \le
    \left(1-\frac{c_B}{n^2}\right)|D|.
\]
Since \(|D|=\rho(x,y)\) and \(|D'|=\rho(X',Y')\), the desired
contraction follows. 
%Taking \(\varepsilon=\varepsilon_0\) proves the lemma.
\end{proof}

\begin{appendices}
\section{Proof of the threshold \texorpdfstring{$c_p/2$}{c-p/2} for the unconditional model}\label{others}
\begin{proposition}
Recall from \cref{eq:I(W)} that
\[
        I(W):=\frac12\int_{[0,1]^2}
        \bigl[W\log W+(1-W)\log(1-W)\bigr]\dd x\dd y
\]
and define the unconditional variational functional (as in \cref{eq:variationUnconditional})
\[
        \Phi_{\alpha,\beta}(W)
        :=\alpha  e(W)+\beta  t(W)-I(W).
\]
If \(W\equiv p\) is a global maximizer of $\Phi_{\alpha,\beta}$ for some \(\alpha \), then
necessarily
\[
        \alpha =\frac12\log\frac{p}{1-p}-2\beta p.
\]
For this value of \(\alpha \), the following hold:
\begin{enumerate}
\item if \(\beta <c_p/2\), then \(W\equiv p\) is the unique global
maximizer;
\item if \(\beta =c_p/2\), then \(W\equiv p\) is a global maximizer,
and it is unique if and only if \(p=1/2\);
\item if \(\beta >c_p/2\), then \(W\equiv p\) is not a global maximizer.
\end{enumerate}
\end{proposition}

\begin{proof}
By \cite[Theorem~4.1]{chatterjee2013estimating}, every global maximizer is
a constant graphon. Thus it suffices to consider
\[
\phi(q):=\alpha q+\beta q^2
-\frac12\{q\log q+(1-q)\log(1-q)\},\qquad q\in[0,1].
\]
If \(q=p\) is a maximizer, then \(\phi'(p)=0\), which gives
\[
\alpha =\frac12\log\frac{p}{1-p}-2\beta p.
\]
For this value of \(\alpha \),
\[
\phi(q)-\phi(p)
=
\beta (q-p)^2-\frac12D(q\|p).
\]
By \cref{lem:sharp}, we have
\[
D(q\|p)\ge c_p(q-p)^2.
\]
Hence, if \(\beta <c_p/2\), then \(\phi(q)<\phi(p)\) for every
\(q\ne p\), proving uniqueness. If \(\beta =c_p/2\), then \(p\mathbf{1}\) is a
maximizer; the equality cases in \cref{lem:sharp} show
that it is unique exactly when \(p=1/2\).

Finally, if \(\beta >c_p/2\), the definition of \(c_p\) gives some
\(q\ne p\) such that \(D(q\|p)<2\beta (q-p)^2\). Thus
\(\phi(q)>\phi(p)\), so \(W\equiv p\) is not a global maximizer.
\end{proof}

\section{From degree concentration to Lemma~\ref{le-1}}\label{sec:degree-to-local}

We follow the cavity method of \cite{bresler2024metastable}. 
The argument proceeds from macroscopic cut-metric control to uniform local
regularity.  
It suffices to control all vertex degrees.
To obtain the degree control, we decompose the
cut ball into slices according to the degree of one cavity vertex \(u\) and compare their Gibbs weights through fixed-edge cavity switchings.
Lemma~\ref{twostar-shell} removes the negligible outer cut shell, while
Lemma~\ref{lem:twostar-cut-consequences} supplies the bulk estimates needed
to analyze cavity switchings.  Lemmas~\ref{lem:twostar-energy} and
\ref{lem:twostar-slice-ratio} compute, respectively, the energetic and
combinatorial costs of moving an atypical degree \(d\) to the reference
degree \(d_*\).  Their combination in
Lemma~\ref{lem:twostar-switching-ratio} yields the effective penalty
\[
        D(q\|p)-\beta(q-p)^2,
        \qquad q=d/n.
\]
Its positivity in the replica symmetric region then gives uniform degree
concentration.

For a vertex \(u\in[n]\), write
\[
        p_u(X)=\frac{\deg_X(u)}{n}.
\]
For \(d\in\{0,\ldots,n-1\}\), define the degree slice inside the cut ball by
\[
        \mathcal A^{\eta,k}_{u,d}
        =
        \{X\in \square_\eta^k(p):\deg_X(u)=d\},
\]
and its unnormalized weight by (recall the Hamiltonian from \cref{eq:Hamilton})
\[
        \mathcal Z^{\eta,k}_{u,d}
        =
        \sum_{X\in \mathcal A^{\eta,k}_{u,d}}
        e^{\mathcal H_\beta(X)}.
\]
For a general set %\(E\subseteq \Omega_k\), 
$E$, write
\[
        \mathcal Z(E):=\sum_{X\in E\cap\Omega_k}e^{\mathcal H_\beta(X)}.
\]
and let
\[
Bulk=[n]\backslash \{u\},\quad 
        N_{Bulk}=\binom{n-1}{2},
        \quad \text{and}\ 
        d_*=\lfloor p(n-1)\rfloor
\]
be the remaining $n-1$ vertices except for $u$, the total number of possible edges in the bulk, and the reference degree of $u$, respectively.
In the following, we will consider cavity switchings that moves edges connecting to $u$ to the bulk or vise versa. 
A cavity switching may change \(O(n)\) edge indicators and hence perturb the
cut norm by \(O(n^{-1})\) (see \cref{eq:cutchange}).  We first create room for these switchings by
showing that the outer shell between radii \(\eta/2\) and \(\eta\) carries
negligible Gibbs weight.
\begin{lemma}
\label{twostar-shell}
Under the standing assumption \((p,\beta)\in K\), for every sufficiently
small fixed \(\eta>0\), there exist constants \(c_{K,\eta}>0\) and
\(n_{K,\eta}<\infty\) such that, uniformly over \((p,\beta)\in K\), for all
 \(n\ge n_{K,\eta}\),
\[
        \mathcal Z\bigl(\square_\eta^k(p)\setminus\square_{\eta/2}^k(p)\bigr)
        \le
        e^{-c_{K,\eta}n^2}\,
        \mathcal Z\bigl(\square^k_\eta(p)\bigr).
\]
%where \(\mathcal Z(A)=\sum_{x\in A\cap\Omega_k}e^{\mathcal H_\beta(x)}\).
\end{lemma}

\begin{proof}
By Proposition~\ref{theo of appro}, uniformly over \((p,\beta)\in K\), for
every fixed \(r>0\) there exist constants \(C_{K,r},c_{K,r}>0\) such that
\(
        \mu_\beta^k\bigl((\square_r^k(p))^c\bigr)
        \le C_{K,r} e^{-c_{K,r} n^2}
\)
for all sufficiently large  \(n\), with a lower threshold depending
only on \(K\) and \(r\).  Applying this estimate with \(r=\eta/2\) and
\(r=\eta\), we get
\[
\begin{aligned}
\frac{\mathcal Z(\square_\eta^k(p)\setminus\square_{\eta/2}^k(p))}
     {\mathcal Z(\square_\eta^k(p))}
&=
\frac{\mu_\beta^k(\square_\eta^k(p)\setminus\square_{\eta/2}^k(p))}
     {\mu_\beta^k(\square_\eta^k(p))}  \\
&\le
\frac{\mu_\beta^k((\square_{\eta/2}^k(p))^c)}
     {1-\mu_\beta^k((\square_\eta^k(p))^c)}
\le
\frac{C_{K,\eta/2}e^{-c_{K,\eta/2}n^2}}
     {1-C_{K,\eta}e^{-c_{K,\eta}n^2}}.
\end{aligned}
\]
For all sufficiently large \(n\), the denominator is at least \(1/2\) and the claim is proved.
\end{proof}
%We collect the elementary cut-norm consequences needed to control the one-site energy difference in the two star case.

We may therefore restrict atypical degree slices to the inner ball
\(\square_{\eta/2}^k(p)\), while allowing their switching images to lie in
\(\square_\eta^k(p)\).  To compare the corresponding Gibbs weights, we next
translate cut-metric closeness of the bulk graph into estimates for the
two star Hamiltonian.

\begin{lemma}\label{lem:twostar-cut-consequences}
Let $Y$ be the bulk graph with $M:=n-1$ vertices (except for the cavity vertex $u$).
Write
\(
        b_v:=\deg_Y(v), v\in Bulk.
\)
Assume that
\(
        \|Y-p\|_{\square,Bulk}\le\eta,
\)
where
\[
        \|Y-p\|_{\square,Bulk}
        :=
        M^{-2}\sup_{S,T\subseteq Bulk}
        \left|
        \sum_{x\in S,y\in T}(Y_{xy}-p)
        \right|.
\]
Then the following estimates hold.  The implicit constants are uniform whenever
\(p\) ranges over a fixed compact subinterval of \((0,1)\):
\[
    e(Y)
        =
        \frac {p}2M^2+O(\eta M^2+M), \qquad \text{(edge estimate)}
\]
\[
        \sum_{v\in Bulk}b_v(b_v-1)
        =
        p^2 M^3+O(\eta M^3+M^2). \qquad \text{(two star estimate)}
\]
Furthermore, if \(r:Bulk\to\{0,1\}\) and \(d=\sum_{v\in Bulk}r_v\), then
\[
        \sum_{v\in Bulk}r_v b_v
        =
        pMd+O(\eta M^2). \qquad \text{(weighted row bulk estimate)}
\]
Finally, if
\[
        \mathsf N_Y:=\{\{a,b\}\subset Bulk:Y_{ab}=0\},
        \qquad
        \mathsf E_Y:=\{\{a,b\}\subset Bulk:Y_{ab}=1\},
\]
then there exist thresholds \(\eta_0>0\) and \(n_0(\eta)<\infty\), chosen
uniformly whenever \(p\) ranges over the fixed compact subinterval above,
such that, for \(0<\eta\le\eta_0\) and \(n\ge n_0(\eta)\),
\[
        \frac1{|\mathsf N_Y|}
        \sum_{\{a,b\}\in\mathsf N_Y}(b_a+b_b)
        =
        2pM+O(\eta M+1), \qquad \text{(averaged degree estimate)}
\]
and
\[
        \frac1{|\mathsf E_Y|}
        \sum_{\{a,b\}\in\mathsf E_Y}(b_a+b_b)
        =
        2pM+O(\eta M+1).
\]
\end{lemma}

\begin{proof}
Taking \(S=T=Bulk\) in the cut norm gives
\[
        \sum_{x,y\in Bulk}Y_{xy}
        =
        pM^2+O(\eta M^2).
\]
Since the diagonal contributes only \(O(M)\), this implies
\[
        e(Y)
        =
        \frac {p}2M^2+O(\eta M^2+M).
\]

We next prove the two star estimate.  Set
\(
        H_{xy}:=Y_{xy}-p.
\)
%We first record the following weighted consequence of the cut norm.  
For
\(0\le a_x\le 1\) and \(0\le c_y\le 1\), we have
\[
\begin{aligned}
        M^{-2}\sum_{x,y\in Bulk}a_xc_yH_{xy}
        =
        M^{-2}\int_0^1\int_0^1
        \sum_{\substack{x,y\in Bulk:a_x>s\\ x,y\in Bulk:c_y>t}}
        H_{xy}\,ds\,dt .
\end{aligned}
\]
Hence
\begin{equation}\label{eq:weighted-cut-bound}
        \left|
        M^{-2}\sum_{x,y\in Bulk}a_xc_yH_{xy}
        \right|
        \le \eta .
\end{equation}

Now
\[
        \sum_{v\in Bulk}b_v^2
        =
        \sum_{v,u,w\in Bulk}Y_{vu}Y_{vw}.
\]
Using
\[
        Y_{vu}Y_{vw}-p^2
        =
        H_{vu}Y_{vw}+pH_{vw},
\]
we obtain
\[
\begin{aligned}
        M^{-3}\sum_{v,u,w\in Bulk}Y_{vu}Y_{vw}-p^2 =
        M^{-3}\sum_{v,u,w\in Bulk}H_{vu}Y_{vw}
        +
        pM^{-3}\sum_{v,u,w\in Bulk}H_{vw}.
\end{aligned}
\]
For the second term,
\[
        pM^{-3}\sum_{v,u,w\in Bulk}H_{vw}
        =
        pM^{-2}\sum_{v,w\in Bulk}H_{vw}
        =
        O(\eta).
\]
For the first term,
\[
\begin{aligned}
        M^{-3}\sum_{v,u,w\in Bulk}H_{vu}Y_{vw}
        =
        M^{-3}\sum_{v,u\in Bulk}H_{vu}b_v                
        =
        M^{-2}\sum_{v,u\in Bulk}H_{vu}\frac{b_v}{M}.
\end{aligned}
\]
Since \(0\le b_v/M\le 1\), \eqref{eq:weighted-cut-bound} gives
\[
        M^{-3}\sum_{v,u,w\in Bulk}H_{vu}Y_{vw}
        =
        O(\eta).
\]
Therefore
\[
        M^{-3}\sum_{v\in Bulk}b_v^2
        =
        p^2+O(\eta).
\]
Since
\(
        b_v(b_v-1)=b_v^2-b_v
\)
and
\(
        0\le \sum_{v\in Bulk}b_v\le M^2,
\)
we have
\[
        M^{-3}\sum_{v\in Bulk}b_v
        =
        O(M^{-1}).
\]
Consequently,
\[
        M^{-3}\sum_{v\in Bulk}b_v(b_v-1)
        =
        p^2+O(\eta+M^{-1}),
\]
which is equivalent to
\[
        \sum_{v\in Bulk}b_v(b_v-1)
        =
        p^2 M^3+O(\eta M^3+M^2).
\]

For the weighted row-bulk estimate, write
\[
        \sum_{v\in Bulk}r_v b_v
        =
        \sum_{v,w\in Bulk} r_vY_{vw}.
\]
Thus
\[
\begin{aligned}
        \sum_{v,w\in Bulk} r_vY_{vw}
        &=
        p\sum_{v,w\in Bulk}r_v
        +
        \sum_{v,w\in Bulk}r_v(Y_{vw}-p)      =
        pMd
        +
        \sum_{v,w\in Bulk}r_vH_{vw}.
\end{aligned}
\]
The last term is bounded by \(\eta M^2\) by
\eqref{eq:weighted-cut-bound}, since \(r\) is the indicator of a subset of
\textit{Bulk} and the second weight is the constant function \(1\).  Hence
\[
        \sum_{v\in Bulk}r_v b_v
        =
        pMd+O(\eta M^2).
\]

It remains to prove the two averaged degree estimates.  For nonedges,
\[
\begin{aligned}
        \sum_{\{a,b\}\in\mathsf N_Y}(b_a+b_b)
        &=
        \sum_{a\in Bulk}b_a\bigl((M-1)-b_a\bigr)  =
        (M-1)\sum_{a\in Bulk}b_a-\sum_{a\in Bulk}b_a^2       \\
        &=
        (M-2)\sum_{a\in Bulk}b_a-\sum_{a\in Bulk}b_a(b_a-1).
\end{aligned}
\]
Using
\[
        \sum_{a\in Bulk}b_a=2| \mathsf E_Y|=pM^2+O(\eta M^2+M)
\]
and the two star estimate above, we obtain
\[
        \sum_{\{a,b\}\in\mathsf N_Y}(b_a+b_b)
        =
        p(1-p)M^3+O(\eta M^3+M^2).
\]
Also, we have
\[
        |\mathsf N_Y|
        =
        \binom M2-| \mathsf E_Y|
        =
        \frac{1-p}{2}M^2+O(\eta M^2+M).
\]
Dividing the two estimates gives
\[
        \frac1{|\mathsf N_Y|}
        \sum_{\{a,b\}\in\mathsf N_Y}(b_a+b_b)
        =
        2pM+O(\eta M+1),
\]
provided \(\eta\) is sufficiently small.

For edges,
\[
        \sum_{\{a,b\}\in\mathsf E_Y}(b_a+b_b)
        =
        \sum_{a\in Bulk}b_a^2
        =
        \sum_{a\in Bulk}b_a(b_a-1)+\sum_{a\in Bulk}b_a.
\]
Therefore
\[
        \sum_{\{a,b\}\in\mathsf E_Y}(b_a+b_b)
        =
        p^2 M^3+O(\eta M^3+M^2).
\]
Since
\[
        |\mathsf E_Y|
        =
        \frac {p}2M^2+O(\eta M^2+M),
\]
we get
\[
        \frac1{|\mathsf E_Y|}
        \sum_{\{a,b\}\in\mathsf E_Y}(b_a+b_b)
        =
        2pM+O(\eta M+1).
\]
This completes the proof.
\end{proof}

%We now prove the one-site energy estimate.  This is the point where the two star structure gives the exact correction \((q-p)^2\).

%The preceding estimates control both the row--bulk interaction and the average effect of adding or deleting bulk edges.  They now give the energetic contribution to the cavity comparison; the quadratic two star structure produces the precise correction \((q-p)^2\).

For \(d,d'\in\{0,\ldots,n-1\}\), a \emph{single-vertex
fixed-edge switching from \(d\) to \(d'\) at \(u\)} is defined as follows.
If \(d>d'\), it deletes \(d-d'\) edges incident to \(u\) and adds the same
number of nonedges within \(Bulk=[n]\setminus\{u\}\).  If \(d<d'\), it adds
\(d'-d\) nonedges incident to \(u\) and deletes the same number of edges
within \(Bulk\).  In either case the degree of \(u\) changes from \(d\) to
\(d'\), while the total number of edges is preserved.

\begin{lemma}\label{lem:twostar-energy}
Let \(X\in\mathcal A^{\eta/2,k}_{u,d}\), and set
\(
        q=d/n, q_*=d_*/n.
\)
Let \(X^\sigma\) be chosen uniformly among all single-vertex fixed-edge
switchings which change the degree of \(u\) from \(d\) to \(d_*\), while
preserving the total number of edges.  Then there exists
\(n_0=n_0(K,\eta)<\infty\) such that, uniformly over the standing
\((p,\beta)\in K\), for every  \(n\ge n_0\),
\[
        \mathcal H_\beta(X)
        -
        \mathbb E_\sigma\mathcal H_\beta(X^\sigma)
        =
        \beta n(q-p)^2+O(\eta n+1).
\]
\end{lemma}

\begin{proof}
Recall \(Bulk=[n]\setminus\{u\}\), \(M=|Bulk|=n-1\), and let \(Y=X|_{Bulk}\) be the bulk
graph.  For \(v\in Bulk\), write
\(
        r_v=X_{uv}.
\)
Thus
\(
        d=\deg_X(u)=\sum_{v\in Bulk}r_v,
\)
and for \(v\in Bulk\),
\(
        \deg_X(v)=b_v+r_v.
\)

Set
\[
        \mathcal S(X)=\frac1n\sum_{w\in[n]}
        \deg_X(w)(\deg_X(w)-1)
\]
and recall
        $\mathcal H_\beta(X)=\beta\mathcal S(X).$
Then
\[
\begin{aligned}
        \mathcal S(X)
        &=
        \frac1n d(d-1)
        +
        \frac1n\sum_{v\in Bulk}(b_v+r_v)(b_v+r_v-1)  \\
        &=
        \frac1n\sum_{v\in Bulk}b_v(b_v-1)
        +
        \frac1n d(d-1)
        +
        \frac2n\sum_{v\in Bulk}r_vb_v.
\end{aligned}
\]
We call the first term the bulk part and the last two terms the row part.

First consider the row part.  Since \(X\in\square_{\eta/2}^k(p)\), the bulk graph
\(Y\) satisfies
\(
        \|Y-p\|_{\square,Bulk}\le \eta
\)
for all sufficiently large \(n\).  By Lemma~\ref{lem:twostar-cut-consequences},
\[
        \sum_{v\in Bulk}r_vb_v
        =
        pMd+O(\eta n^2).
\]
Hence
\[
\begin{aligned}
        \frac1n d(d-1)+\frac2n\sum_{v\in Bulk}r_vb_v
        &=
        \frac1n d(d-1)
        +
        \frac2n\{pMd+O(\eta n^2)\}        \\
        &=
        nq^2-q+2p(n-1)q+O(\eta n)        \\
        &=
        nq^2+2pnq+O(\eta n+1).
\end{aligned}
\]
For every switching image \(X^\sigma\), the degree of \(u\) is \(d_*\), and
\(X^\sigma\in\square_\eta^k(p)\).  Applying the same estimate to
\(X^\sigma\) gives, uniformly in \(\sigma\),
\[
        \frac1n d_*(d_*-1)
        +
        \frac2n\sum_{v\in Bulk}r_v^\sigma b_v^\sigma
        =
        nq_*^2+2pnq_*+O(\eta n+1).
\]
Therefore the row part contributes
\(
        n(q^2-q_*^2)+2pn(q-q_*)+O(\eta n+1)
\)
to
\(
        \mathcal S(X)
        -
        \mathbb E_\sigma \mathcal S(X^\sigma).
\)

We now estimate the bulk part.  Let
\[
        S_0(Y)=\frac1n\sum_{v\in Bulk}b_v(b_v-1).
\]
Suppose first that \(d>d_*\), and put \(L=d-d_*\).  The switching deletes
\(L\) row edges incident to \(u\) and adds \(L\) bulk nonedges.  If a single
bulk nonedge \(\{a,b\}\) is added, then only the bulk degrees of \(a\) and
\(b\) increase by one, so
\[
        S_0(Y^{+ab})-S_0(Y)
        =
        \frac2n(b_a+b_b).
\]
Averaging over bulk nonedges and using
Lemma~\ref{lem:twostar-cut-consequences}, we obtain
\[
        \frac1{|\mathsf N_Y|}
        \sum_{\{a,b\}\in\mathsf N_Y}
        \bigl(S_0(Y^{+ab})-S_0(Y)\bigr)
        =
        4p+O(\eta+n^{-1}).
\]

The switching adds \(L\le n\) bulk nonedges simultaneously.  If \(c_v\) is the
number of selected added bulk edges incident to \(v\), then
\[
        (b_v+c_v)(b_v+c_v-1)-b_v(b_v-1)
        =
        2c_vb_v+c_v(c_v-1).
\]
The first term is exactly the sum of single-edge increments.  The second term
counts ordered pairs of selected added edges sharing the vertex \(v\).  Since
the number of pairs of bulk pairs sharing an endpoint is \(O(n^3)\), and since
a fixed pair of bulk pairs is selected with probability \(O(L^2/n^4)\), the
expected number of such selected intersecting pairs is
\(
        O(n^3)\cdot O(L^2/n^4)=O(L^2/n)\le O(n).
\)
After the prefactor \(1/n\), the expected nonlinear contribution is \(O(1)\).
Consequently,
\[
        \mathbb E_\sigma\{S_0(Y^\sigma)-S_0(Y)\}
        =
        L\{4p+O(\eta+n^{-1})\}+O(1).
\]
Since \(L=d-d_*=n(q-q_*)\), this becomes
\[
        \mathbb E_\sigma\{S_0(Y^\sigma)-S_0(Y)\}
        =
        4pn(q-q_*)+O(\eta n+1).
\]

If \(d<d_*\), then \(L=d_*-d\), and the switching adds \(L\) row edges and
deletes \(L\) bulk edges.  If a single bulk edge \(\{a,b\}\) is deleted, then
only the bulk degrees of \(a\) and \(b\) decrease by one, and therefore
\[
\begin{aligned}
        S_0(Y)-S_0(Y^{-ab})
        &=
        \frac1n\Big[
        b_a(b_a-1)-(b_a-1)(b_a-2)       \\
        &\qquad\qquad
        +b_b(b_b-1)-(b_b-1)(b_b-2)
        \Big]                                    \\
        &=
        \frac2n(b_a+b_b)-\frac4n                  \\
        &=
        \frac2n(b_a+b_b)+O(n^{-1}).
\end{aligned}
\]
Averaging over bulk edges and using
Lemma~\ref{lem:twostar-cut-consequences} again yields the same single-edge
main term
\[
        4p+O(\eta+n^{-1}).
\]
The nonlinear error from simultaneous deletions is again \(O(1)\).  Since now
\(q-q_*<0\), the same formula holds:
\[
        \mathbb E_\sigma\{S_0(Y^\sigma)-S_0(Y)\}
        =
        4pn(q-q_*)+O(\eta n+1).
\]
Therefore, in all cases,
\[
        S_0(Y)-\mathbb E_\sigma S_0(Y^\sigma)
        =
        -4pn(q-q_*)+O(\eta n+1).
\]

Combining the row part and the bulk part, we obtain
\[
\begin{aligned}
        \mathcal S(X)
        -
        \mathbb E_\sigma \mathcal S(X^\sigma)
        &=
        n(q^2-q_*^2)
        +2pn(q-q_*)
        -4pn(q-q_*)      
        +O(\eta n+1)      \\
        &=
        n(q^2-q_*^2)-2pn(q-q_*)
        +O(\eta n+1).
\end{aligned}
\]
Since
\(
        q_*={d_*}/{n}=p+O(n^{-1}),
\)
we have
\[
        q^2-q_*^2-2p(q-q_*)
        =
        (q-p)^2-(q_*-p)^2
        =
        (q-p)^2+O(n^{-2}).
\]
Hence
\[
        \mathcal S(X)
        -
        \mathbb E_\sigma \mathcal S(X^\sigma)
        =
        n(q-p)^2+O(\eta n+1).
\]
Multiplying by \(\beta\) gives
\[
        \mathcal H_\beta(X)
        -
        \mathbb E_\sigma\mathcal H_\beta(X^\sigma)
        =
        \beta n(q-p)^2+O(\eta n+1),
\]
as claimed.
\end{proof}

%We next compute the combinatorial multiplicity ratio between the degree slice \(d\) and the reference degree slice \(d_*\).

%Lemma~\ref{lem:twostar-energy} determines the energetic cost of the cavity map.  
We next compute the entropic cost of the cavity map by comparing the numbers of forward
and reverse switchings between the degree-\(d\) and degree-\(d_*\) slices.

\begin{lemma}\label{lem:twostar-slice-ratio}
Let \(d\in\{0,\ldots,n-1\}\) be feasible, meaning
\(
        0\le k-d\le N_{Bulk}.
\)
Let \(F_d\) be the number of forward switchings from degree \(d\) to degree
\(d_*\), and let \(R_d\) be the maximal number of reverse switchings from
degree \(d_*\) back to degree \(d\).  Then
\[
        \frac{R_d}{F_d}
        =
        \frac{\binom{n-1}{d}\binom{N_{Bulk}}{k-d}}
             {\binom{n-1}{d_*}\binom{N_{Bulk}}{k-d_*}}.
\]
Moreover, uniformly over all feasible \(d\) and all standing
\((p,\beta)\in K\),
\[
        \log\frac{R_d}{F_d}
        =
        -nD\left(\frac dn\middle\|p\right)+O_K(\log n).
\]
\end{lemma}

\begin{proof}
Suppose first that \(d>d_*\), and put \(L=d-d_*\).  From degree \(d\) to degree
\(d_*\), one deletes \(L\) row edges incident to \(u\) and adds \(L\) bulk
nonedges.  Therefore
\[
        F_d
        =
        \binom dL
        \binom{N_{Bulk}-(k-d)}L.
\]
Conversely, from degree \(d_*\) back to degree \(d\), one adds \(L\) row
nonedges incident to \(u\) and deletes \(L\) bulk edges.  Thus
\[
        R_d
        =
        \binom{n-1-d_*}{L}
        \binom{k-d_*}{L}.
\]
Using the identities
\[
        \frac{\binom{n-1-d_*}{L}}{\binom dL}
        =
        \frac{\binom{n-1}{d}}{\binom{n-1}{d_*}}
\]
and
\[
        \frac{\binom{k-d_*}{L}}
             {\binom{N_{Bulk}-(k-d)}L}
        =
        \frac{\binom{N_{Bulk}}{k-d}}
             {\binom{N_{Bulk}}{k-d_*}},
\]
we obtain the desired product formula.  The case \(d<d_*\) is identical, with
additions and deletions interchanged.  The case \(d=d_*\) is trivial.

It remains to estimate this ratio.  Recall
\(
        M=n-1, q=d/n.
\)
All Stirling and Taylor remainders below are uniform because
\(K\) is a compact subset of \(\mathscr R_{\mathrm{RS}}\).
By uniform Stirling's formula,
\[
        \log\binom M d
        =
        Mh\left(\frac dM\right)+O_K(\log n),
\]
where
\(
        h(x):=-x\log x-(1-x)\log(1-x).
\)
Uniformly in \(d\),
\[
        Mh\left(\frac dM\right)
        =
        nh(q)+O_K(\log n),
\]
and since \(d_*=\lfloor p(n-1)\rfloor\),
\[
        Mh\left(\frac{d_*}{M}\right)
        =
        nh(p)+O_K(1).
\]
Therefore
\[
        \log\binom M d-\log\binom M{d_*}
        =
        n\{h(q)-h(p)\}+O_K(\log n).
\]

Note
\(
        k-d
        =
        p\binom n2-d
        =
        pN_{Bulk}+s_d,\) where 
        \(s_d=p(n-1)-d.
\)
For feasible \(d\), one has \(|s_d|=O_K(n)\).  Hence Stirling's formula and a
Taylor expansion around \(p\) give, uniformly for \(|s|=O_K(n)\),
\[
        \log\binom{N_{Bulk}}{pN_{Bulk}+s}
        =
        N_{Bulk} h(p)+s h'(p)+O_K(\log n),
\]
because \(s^2/N_{Bulk}=O(1)\).  Since
\(
        h'(p)=\log\frac{1-p}{p},
\)
and
\[
        s_d-s_{d_*}=d_*-d=-n(q-p)+O_K(1),
\]
we obtain
\[
        \log\binom{N_{Bulk}}{k-d}
        -
        \log\binom{N_{Bulk}}{k-d_*}
        =
        -n(q-p)\log\frac{1-p}{p}
        +O_K(\log n).
\]
Combining the row and bulk terms,
\[
\begin{aligned}
        \log\frac{R_d}{F_d}
        &=
        n\{h(q)-h(p)\}
        -
        n(q-p)\log\frac{1-p}{p}
        +
        O_K(\log n)       =
        -nD(q\|p)+O_K(\log n).
\end{aligned}
\]
This completes the proof.
\end{proof}

%Now we combine the energy estimate and the multiplicity estimate.

The energy estimate contributes
\(\beta n(q-p)^2\), whereas the switching multiplicity contributes
\(-nD(q\|p)\).  Combining them gives the following comparison between the
Gibbs weights of an atypical degree slice and the reference slice.

\begin{lemma}\label{lem:twostar-switching-ratio}
For every fixed \(\varepsilon>0\), there exist
\(\eta_0=\eta_0(K,\varepsilon)>0\) and \(C=C(K,\varepsilon)<\infty\) such that,
for every fixed \(0<\eta<\eta_0\), there exists
\(n_0=n_0(K,\varepsilon,\eta)<\infty\) such that, uniformly over
\((p,\beta)\in K\), for every  \(n\ge n_0\), if
\(
        \left|d/n-p\right|\ge\varepsilon,
\)
then
\[
        \mathcal Z^{\eta/2,k}_{u,d}
        \le
        \exp\left\{
        -n\left[
        D\left(\frac dn\middle\|p\right)
        -
        \beta\left(\frac dn-p\right)^2
        \right]
        +C\eta n+C\log n
        \right\}
        \mathcal Z^{\eta,k}_{u,d_*}.
\]
\end{lemma}

\begin{proof}
If \(\mathcal A^{\eta/2,k}_{u,d}=\varnothing\), the statement is trivial.
Thus assume that this slice is nonempty.  Fix
\(X\in\mathcal A^{\eta/2,k}_{u,d}\).  Let \(X^\sigma\) be one of the \(F_d\)
forward switchings from degree \(d\) to degree \(d_*\).

Since \(X^\sigma\) is obtained from \(X\) by changing \(2|d-d_*|\le 2n\)
edges, for every cut \(S,T\subset[n]\),
\[
        \left|
        \sum_{i\in S,j\in T}
        (X^\sigma_{ij}-X_{ij})
        \right|
        \le Cn.
\]
Therefore
\ben{\label{eq:cutchange}
        \|X^\sigma-X\|_{\square}
        \le \frac{C}{n}.
}
Since \(X\in\square_{\eta/2}^k(p)\), it follows that
\(
        X^\sigma\in\square_{\eta/2+C/n}^k(p)\subset \square_\eta^k(p)
\)
for all sufficiently large \(n\).  Hence we have
\(
        X^\sigma\in \mathcal A^{\eta,k}_{u,d_*}.
\)

By Jensen's inequality,
\[
        \frac1{F_d}\sum_{\sigma}e^{\mathcal H_\beta(X^\sigma)}
        \ge
        \exp\left\{
        \mathbb E_\sigma\mathcal H_\beta(X^\sigma)
        \right\},
\]
and by Lemma~\ref{lem:twostar-energy},
\[
        \mathbb E_\sigma\mathcal H_\beta(X^\sigma)
        \ge
        \mathcal H_\beta(X)
        -
        \beta n\left(\frac dn-p\right)^2
        -
        C_K\eta n-C_K.
\]
Hence
\[
        \frac1{F_d}\sum_{\sigma}e^{\mathcal H_\beta(X^\sigma)}
        \ge
        \exp\left\{
        \mathcal H_\beta(X)
        -
        \beta n\left(\frac dn-p\right)^2
        -
        C_K\eta n-C_K
        \right\}.
\]
Equivalently,
\[
        F_d e^{\mathcal H_\beta(X)}
        \le
        \exp\left\{
        \beta n\left(\frac dn-p\right)^2
        +
        C_K\eta n+C_K
        \right\}
        \sum_{\sigma}e^{\mathcal H_\beta(X^\sigma)}.
\]

We now sum this inequality over all
\(X\in\mathcal A^{\eta/2,k}_{u,d}\).  The left-hand side becomes
\(
        F_d\mathcal Z^{\eta/2,k}_{u,d}.
\)
The right-hand side contains the double sum
\[
        \sum_{X\in\mathcal A^{\eta/2,k}_{u,d}}
        \sum_{\sigma}e^{\mathcal H_\beta(X^\sigma)}.
\]
We regroup this double sum according to the image \(X'=X^\sigma\).  For
\(X'\in\mathcal A^{\eta,k}_{u,d_*}\), let \(N(X')\) be the number of
preimage-switching pairs \((X,\sigma)\) such that \(X^\sigma=X'\).  By the
definition of \(R_d\), every \(X'\) has at most \(R_d\) such preimage-switching
pairs.  Thus
\[
\begin{aligned}
        \sum_{X\in\mathcal A^{\eta/2,k}_{u,d}}
        \sum_{\sigma}e^{\mathcal H_\beta(X^\sigma)}
        &=
        \sum_{X'\in\mathcal A^{\eta,k}_{u,d_*}}
        N(X')e^{\mathcal H_\beta(X')}         \le
        R_d
        \sum_{X'\in\mathcal A^{\eta,k}_{u,d_*}}
        e^{\mathcal H_\beta(X')}               =
        R_d\mathcal Z^{\eta,k}_{u,d_*}.
\end{aligned}
\]
Therefore
\[
        F_d\mathcal Z^{\eta/2,k}_{u,d}
        \le
        \exp\left\{
        \beta n\left(\frac dn-p\right)^2
        +
        C_K\eta n+C_K
        \right\}
        R_d\mathcal Z^{\eta,k}_{u,d_*}.
\]
Dividing by \(F_d\), we get
\[
        \mathcal Z^{\eta/2,k}_{u,d}
        \le
        \frac{R_d}{F_d}
        \exp\left\{
        \beta n\left(\frac dn-p\right)^2
        +
        C_K\eta n+C_K
        \right\}
        \mathcal Z^{\eta,k}_{u,d_*}.
\]
Finally, Lemma~\ref{lem:twostar-slice-ratio} gives
\[
        \frac{R_d}{F_d}
        =
        \exp\left\{
        -nD\left(\frac dn\middle\|p\right)
        +O_K(\log n)
        \right\}.
\]
Substituting this estimate proves the claim.
\end{proof}

%We can now prove the uniform degree concentration estimate in the two star case.

Lemma~\ref{lem:twostar-switching-ratio} shows that a macroscopic degree
deviation is exponentially suppressed whenever
\(D(q\|p)-\beta(q-p)^2>0\).  Summing this estimate over all bad degree
slices, and using Lemma~\ref{twostar-shell} for the outer shell, yields
uniform degree concentration. Now we turn to completing the proof of \cref{le-1}.
\begin{proof}[Proof of \cref{le-1}]
All choices below are uniform over \((p,\beta)\in K\). Fix \(\varepsilon>0\), and set
\[
        \kappa_{K,\varepsilon}
        =
        \inf_{\substack{(p,\beta)\in K\\
        q\in[0,1],\ |q-p|\ge\varepsilon}}
        \{D(q\|p)-\beta(q-p)^2\}.
\]
By the definition of \(c_{p}\) and compactness of \(K\),
\(
        \kappa_{K,\varepsilon}>0.
\)
Write \(\kappa=\kappa_{K,\varepsilon}\).
Let \(C\) be the constant appearing in
Lemma~\ref{lem:twostar-switching-ratio}, which depends only on \(K\) and
\(\varepsilon\).  Choose \(\eta_0=\eta_0(K,\varepsilon)>0\) sufficiently small
so that
\(
        C\eta_0\le {\kappa}/{4}.
\)
Fix \(0<\eta<\eta_0\).  We then choose
\(n_0=n_0(K,\varepsilon,\eta)\) sufficiently large so that, for all
\(n\ge n_0\),
\(
        C\log n\le \frac{\kappa n}{4},
\)
the shell estimate in \cref{twostar-shell} holds, and
\[
        e^{-c_{K,\eta}n^2}\le e^{-\kappa n/4}.
\]

Fix a vertex \(u\).  Define
\[
        \mathcal D_\varepsilon
        :=
        \left\{
        d\in\{0,\ldots,n-1\}:
        \left|\frac dn-p\right|>\varepsilon
        \right\}.
\]
Inside \(\square_\eta^k(p)\), we split the bad event into the outer shell and
the interior degree slices:
\[
\begin{aligned}
        \{|p_u(X)-p|>\varepsilon\}\cap \square_\eta^k(p)
        \subseteq
        \bigl(\square_\eta^k(p)\setminus \square_{\eta/2}^k(p)\bigr)
        \cup
        \bigcup_{d\in\mathcal D_\varepsilon}
        \mathcal A^{\eta/2,k}_{u,d}.
\end{aligned}
\]
Therefore,
\[
\begin{aligned}
        \mathcal Z\bigl(
        \{|p_u(X)-p|>\varepsilon\}\cap \square_\eta^k(p)
        \bigr)
        &\le
        \mathcal Z\bigl(\square_\eta^k(p)\setminus \square_{\eta/2}^k(p)\bigr)
        \\
        &\quad+
        \sum_{d\in\mathcal D_\varepsilon}
        \mathcal Z^{\eta/2,k}_{u,d}.
\end{aligned}
\]

The shell term is bounded by \cref{twostar-shell}:
\[
        \mathcal Z\bigl(\square_\eta^k(p)\setminus \square_{\eta/2}^k(p)\bigr)
        \le
        e^{-c_{K,\eta}n^2}
        \mathcal Z\bigl(\square_\eta^k(p)\bigr)
        \le
        e^{-\kappa n/4}
        \mathcal Z\bigl(\square_\eta^k(p)\bigr).
\]

For the interior slices, Lemma~\ref{lem:twostar-switching-ratio} gives, for
every \(d\in\mathcal D_\varepsilon\),
\[
\begin{aligned}
        \mathcal Z^{\eta/2,k}_{u,d}
        &\le
        \exp\left\{
        -n\left[
        D\left(\frac dn\middle\|p\right)
        -
        \beta\left(\frac dn-p\right)^2
        \right]
        +C\eta n+C\log n
        \right\}
        \mathcal Z^{\eta,k}_{u,d_*}  \\
        &\le
        \exp\{-\kappa n+C\eta n+C\log n\}
        \mathcal Z\bigl(\square_\eta^k(p)\bigr).
\end{aligned}
\]
By the choices of \(\eta_0\) and \(n_0\),
\(
        -\kappa n+C\eta n+C\log n
        \le
        -{\kappa n}/{2}.
\)
Hence
\[
        \mathcal Z^{\eta/2,k}_{u,d}
        \le
        e^{-\kappa n/2}
        \mathcal Z\bigl(\square_\eta^k(p)\bigr).
\]
Since there are at most \(n\) possible values of \(d\),
\[
        \sum_{d\in\mathcal D_\varepsilon}
        \mathcal Z^{\eta/2,k}_{u,d}
        \le
        n e^{-\kappa n/2}
        \mathcal Z\bigl(\square_\eta^k(p)\bigr)
        \le
        e^{-\kappa n/3}
        \mathcal Z\bigl(\square_\eta^k(p)\bigr)
\]
for all sufficiently large \(n\).

Combining the shell and slice estimates, we obtain
\[
        \mu_\beta^k\left(
        |p_u(X)-p|>\varepsilon
        \,\middle|\,
        \square_\eta^k(p)
        \right)
        \le
        e^{-\kappa n/4}+e^{-\kappa n/3}
        \le
        e^{-\kappa n/5}.
\]
Finally, taking a union bound over \(u=1,\ldots,n\), we get
\[
\begin{aligned}
        \mu_\beta^k\left(
        \sup_{u\in[n]}|p_u(X)-p|>\varepsilon
        \,\middle|\,
        \square_\eta^k(p)
        \right)
        &\le
        n e^{-\kappa n/5}
        \le
        e^{-cn}
\end{aligned}
\]
for some \(c=c(K,\varepsilon)>0\), after increasing \(n_0\) if necessary. Applying the preceding estimate with $\varepsilon/2$ in place of
$\varepsilon$, and using
$\{\sup_{u\in [n]}|p_u(X)-p|<\varepsilon/2\}\subset \Gamma_{p,\varepsilon}$ proves \cref{le-1}.
%$$r(\{u,v\},X)=(\deg_X(u)+\deg_X(v)-2X_{uv})/(2n),$$
%we obtain $X\in\Gamma_{p,\varepsilon}$ for all sufficiently large $n$, which completes the proof.
\end{proof}

\section{A counterexample for the necessity of metastable mixing}\label{sec:GPT1}

Throughout this section, \((p,\beta)\) is fixed, and \(n\) ranges over
admissible values for which \(k=Np\in\mathbb Z\).  
The replica-symmetric estimates used below are applied with the compact
singleton \(K=\{(p,\beta)\}\subset\mathscr R_{\mathrm{RS}}\).

This appendix shows that the restriction on the initial state in
\cref{coupling} cannot in general be removed.  More precisely, even in the
conditional replica-symmetric region and under the contraction condition
\(4\beta p(1-p)<1/2\), a single vertex with anomalously large degree may create
a set of exponentially small conductance.  Consequently, there are initial
configurations outside the regular region whose hitting time of
\(\Gamma_{p,\varepsilon}\), and hence whose total-variation mixing time, is
exponential in \(n\).

We retain the notation
\(\mathcal A_{u,d}^{\eta,k}\), \(\mathcal Z_{u,d}^{\eta,k}\), and
\(\mathcal Z(\cdot)\) from Appendix~\ref{sec:degree-to-local}.    Define the one-row effective potential
\[
        V(q):=\Dkl(q\|p)-\beta(q-p)^2,
        \qquad q\in[0,1].
\]
By the definition of \(c_{p}\), the condition \(\beta<c_{p}\) implies
\(
        V(q)>V(p)=0, q\ne p.
\)

\begin{proposition}
\label{prop:one-vertex-defect}
Assume \(\beta<c_{p}\).  Suppose that there exist
\(
        p<b<q<1
\)
such that
\(
        \Delta:=V(b)-V(q)>0.
\)
Then, for every sufficiently small fixed \(\varepsilon>0\) satisfying
\(
        p+3\varepsilon<b,
\)
there exist constants \(c>0\), \(C<\infty\), and \(n_0<\infty\) such that,
for every admissible \(n\ge n_0\), there is a deterministic
\(x_n\in\Omega_k\setminus\Gamma_{p,\varepsilon}\) for which
\[
        \mathbb P_{x_n}\left(
        \tau_{\Gamma_{p,\varepsilon}}\le e^{cn}
        \right)
        \le e^{-cn},
\]
where
\(
        \tau_{\Gamma_{p,\varepsilon}}
        =\inf\{t\ge0:X_t\in\Gamma_{p,\varepsilon}\}.
\)
Moreover, after decreasing \(c\) if necessary,
\[
        \mathbb E_{x_n}\tau_{\Gamma_{p,\varepsilon}}
        \ge \frac12 e^{cn},
\]
and, for every integer \(0\le t\le e^{cn}\),
\[
        d_{\mathrm{TV}}\bigl(x_nP^t,\mu_\beta^k\bigr)
        \ge 1-Ce^{-cn}.
\]
In particular, the worst total variation mixing time of the full
Kawasaki chain is exponential in \(n\).
\end{proposition}

The proof rests on a two-slice version of the switching comparison established
in Appendix~\ref{sec:degree-to-local}.  The point is that the same
calculation as in Appendix~\ref{sec:degree-to-local} compares two arbitrary macroscopic degrees, rather than only an
atypical degree with the reference degree near \(pn\).
The same arguments leading to \cref{lem:twostar-switching-ratio} gives the following result.

\begin{lemma}
\label{lem:two-slice-comparison}
Fix a compact interval \(J\subset(0,1)\).  There is a constant
\(C=C(J,p,\beta)<\infty\) such that the following holds.  Let
\(q\in[0,1]\), \(q'\in J\), and set
\(
        d=\lfloor qn\rfloor,
        d'=\lfloor q'n\rfloor.
\)
For every sufficiently small fixed \(\eta>0\) and all sufficiently large
 \(n\),
\begin{equation}
\label{eq:two-slice-comparison}
        \mathcal Z_{u,d}^{\eta/2,k}
        \le
        \exp\left\{
        -n\bigl[V(q)-V(q')\bigr]
        +C\eta n+C\log n
        \right\}
        \mathcal Z_{u,d'}^{\eta,k}.
\end{equation}
\end{lemma}

\iffalse
\begin{proof}
For \(X\in\mathcal A_{u,d}^{\eta/2,k}\), let \(X^\sigma\) be chosen uniformly
among the single-vertex fixed-edge switchings which change the degree of \(u\)
from \(d\) to \(d'\).  Repeating the calculation in the proof of
\cref{lem:twostar-energy}, with \(d_*\) replaced by \(d'\), gives
\[
        \mathcal H_\beta(X)
        -\mathbb E_\sigma\mathcal H_\beta(X^\sigma)
        =
        \beta n\left[
        (q-p)^2-(q'-p)^2
        \right]
        +O_{J,p,\beta}(\eta n+1).
\]
Let \(F_{d,d'}\) be the number of such forward switchings and let
\(R_{d,d'}\) be the maximal number of reverse switchings.  The counting
argument in the proof of \cref{lem:twostar-slice-ratio} gives
\[
        \frac{R_{d,d'}}{F_{d,d'}}
        =
        \frac{
        \binom{n-1}{d}\binom{N_{Bulk}}{k-d}
        }{
        \binom{n-1}{d'}\binom{N_{Bulk}}{k-d'}
        },
\]
and the same uniform Stirling estimate yields
\[
        \log\frac{R_{d,d'}}{F_{d,d'}}
        =
        -n\left[
        \Dkl(q\|p)-\Dkl(q'\|p)
        \right]
        +O_J(\log n).
\]
Finally, the Jensen and regrouping argument used in the proof of
\cref{lem:twostar-switching-ratio} applies without change.  Every switching
image lies in \(\square_\eta^k(p)\) for all sufficiently large \(n\), since
at most \(2|d-d'|=O(n)\) edge indicators are changed.  Combining the energy
and multiplicity estimates gives \eqref{eq:two-slice-comparison}.
\end{proof}
\fi

We shall also need a lower bound on the weight of a fixed macroscopic degree
slice.  The estimate is only exponential in \(n\), which is sufficient because
the probability of leaving a fixed cut neighborhood is exponential in
\(n^2\).

\begin{lemma}
\label{lem:fixed-slice-lower-bound}
Assume \(\beta<c_{p}\).  Fix \(q\in(0,1)\) and a sufficiently small fixed
\(\eta>0\), and let \(d_q:=\lfloor qn\rfloor\).  There exists
\(C=C(q,\eta,p,\beta)<\infty\) such that, for all sufficiently large
 \(n\),
\begin{equation}
\label{eq:fixed-slice-lower-bound}
        \frac{\mathcal Z_{u,d_q}^{\eta,k}}
        {\mathcal Z(\Omega_k)}
        \ge e^{-Cn}.
\end{equation}
\end{lemma}

\begin{proof}
Let \(d_*:=\lfloor p(n-1)\rfloor\), as in
Appendix~\ref{sec:degree-to-local}.  Since \(\beta<c_{p}\), we have
\(V(r)\ge0\) for every \(r\in[0,1]\).  Applying
\cref{lem:two-slice-comparison} with target degree \(d_*\), and summing over
all \(d\in\{0,\ldots,n-1\}\), gives
\[
        \mathcal Z\bigl(\square_{\eta/4}^k(p)\bigr)
        \le
        n\exp\{C\eta n+C\log n\}
        \mathcal Z_{u,d_*}^{\eta/2,k}.
\]
By \cref{prop:1,theo of appro},
\(
        \mu_\beta^k\bigl(\square_{\eta/4}^k(p)\bigr)
        \ge \frac12
\)
for all sufficiently large \(n\).  Hence
\[
        \mathcal Z_{u,d_*}^{\eta/2,k}
        \ge e^{-C_1n}\mathcal Z(\Omega_k).
\]
Applying \cref{lem:two-slice-comparison} once more, now comparing \(d_*\) to
\(d_q\), yields
\[
        \mathcal Z_{u,d_*}^{\eta/2,k}
        \le
        \exp\{nV(q)+C\eta n+C\log n\}
        \mathcal Z_{u,d_q}^{\eta,k}.
\]
Combining the last two displays proves \eqref{eq:fixed-slice-lower-bound}.
\end{proof}

Fix a vertex \(u\in[n]\), let
\(
        d_b=\lceil bn\rceil,
        A_n=\{X\in\Omega_k:\deg_X(u)\ge d_b\},
\)
and write
\(
        Q(A,B)=\sum_{x\in A}\mu_\beta^k(x)P(x,B)
\)
for the stationary flow of the full Kawasaki chain.

\begin{lemma}
\label{lem:defect-separation}
If \(p+3\varepsilon<b\), then, for all sufficiently large  \(n\),
\[
        \Gamma_{p,\varepsilon}\subseteq A_n^c.
\]
\end{lemma}

\begin{proof}
Let \(X\in\Gamma_{p,\varepsilon}\).  For every \(w\ne u\), the definition of
\(\Gamma_{p,\varepsilon}\), applied to the edge-position \(\{u,w\}\), gives
\[
        \deg_X(u)+\deg_X(w)-2X_{uw}
        \le 2n(p+\varepsilon).
\]
Consequently,
\[
        \deg_X(u)+\deg_X(w)
        \le 2n(p+\varepsilon)+2.
\]
Summing over \(w\ne u\), and using
\[
        \deg_X(u)+\sum_{w\ne u}\deg_X(w)
        =2k=pn(n-1),
\]
we obtain
\[
        (n-2)\deg_X(u)
        \le n(n-1)(p+2\varepsilon)+O(n).
\]
Thus
\[
        \frac{\deg_X(u)}{n}\le p+3\varepsilon
\]
for all sufficiently large \(n\).  Since \(p+3\varepsilon<b\), this implies
\(\deg_X(u)<d_b\), and hence \(X\notin A_n\).
\end{proof}

\begin{lemma}
\label{lem:defect-conductance}
Under the assumptions of \cref{prop:one-vertex-defect}, there exists
\(c_0>0\) such that, for all sufficiently large  \(n\),
\begin{equation}
\label{eq:defect-conductance}
        \Phi(A_n)
        :=\frac{Q(A_n,A_n^c)}{\mu_\beta^k(A_n)}
        \le e^{-c_0n}.
\end{equation}
\end{lemma}

\begin{proof}
A single Kawasaki move changes \(\deg_X(u)\) by at most one.  Therefore, if
\(x\in A_n\) and \(P(x,A_n^c)>0\), then \(\deg_x(u)=d_b\).  Hence
\[
        Q(A_n,A_n^c)
        \le \mu_\beta^k\bigl(\deg_X(u)=d_b\bigr).
\]
Let \(d_q:=\lfloor qn\rfloor\).  Since \(q>b\), the slice
\(\{\deg_X(u)=d_q\}\) is contained in \(A_n\) for all sufficiently large
\(n\).  For a fixed sufficiently small \(\eta>0\), it follows that
\[
\begin{aligned}
        \Phi(A_n)
        &\le
        \frac{
        \mathcal Z_{u,d_b}^{\eta/2,k}
        +\mathcal Z(\Omega_k)
        \mu_\beta^k\bigl((\square_{\eta/2}^k(p))^c\bigr)
        }{
        \mathcal Z_{u,d_q}^{\eta,k}
        }.
\end{aligned}
\]
By \cref{lem:two-slice-comparison}, and after absorbing the rounding of
\(bn\) and \(qn\) into the error term,
\[
        \frac{\mathcal Z_{u,d_b}^{\eta/2,k}}
        {\mathcal Z_{u,d_q}^{\eta,k}}
        \le
        \exp\{-n\Delta+C\eta n+C\log n\}.
\]
Moreover, by \cref{prop:1,theo of appro,lem:fixed-slice-lower-bound},
\[
        \frac{
        \mathcal Z(\Omega_k)
        \mu_\beta^k\bigl((\square_{\eta/2}^k(p))^c\bigr)
        }{
        \mathcal Z_{u,d_q}^{\eta,k}
        }
        \le
        \exp\{-c_\eta n^2+C_qn\}.
\]
Choose \(\eta>0\) so small that \(C\eta\le\Delta/4\).  The two preceding
bounds then imply \eqref{eq:defect-conductance} after decreasing \(c_0\).
\end{proof}

The following elementary lemma converts small conductance into a lower bound on
the exit time.  We include the short proof for completeness.

\begin{lemma}
\label{lem:exit-small-conductance}
Let \(P\) be a reversible Markov kernel on a finite state space with stationary
law \(\pi\), and let \(A\) satisfy \(\pi(A)>0\).  If the chain starts from
\(\pi_A:=\pi(\,\cdot\mid A)\), then, for every integer \(T\ge1\),
\begin{equation}
\label{eq:exit-small-conductance}
        \mathbb P_{\pi_A}(\tau_{A^c}\le T)
        \le
        T\frac{Q(A,A^c)}{\pi(A)},
\end{equation}
where
\(
        Q(A,A^c)=\sum_{x\in A}\pi(x)P(x,A^c).
\)
\end{lemma}

\begin{proof}
For \(t\ge0\), define
\[
        \nu_t(y):=
        \mathbb P_{\pi_A}(X_t=y,\ \tau_{A^c}>t).
\]
We claim that
\[
        \nu_t(y)
        \le \frac{\pi(y)}{\pi(A)}\mathbf 1_A(y)
\]
for every \(t\ge0\).  This is an equality at \(t=0\).  If it holds at time
\(t\), reversibility gives
\[
\begin{aligned}
        \nu_{t+1}(y)
        &=
        \mathbf 1_A(y)\sum_{x\in A}\nu_t(x)P(x,y)  \le
        \frac{\mathbf 1_A(y)}{\pi(A)}
        \sum_{x\in A}\pi(x)P(x,y) =
        \frac{\pi(y)}{\pi(A)}\mathbf 1_A(y)P(y,A)
        \le
        \frac{\pi(y)}{\pi(A)}\mathbf 1_A(y).
\end{aligned}
\]
Thus the claim follows by induction.  Consequently,
\[
        \mathbb P_{\pi_A}(\tau_{A^c}=t+1)
        =
        \sum_{x\in A}\nu_t(x)P(x,A^c)
        \le
        \frac{Q(A,A^c)}{\pi(A)}.
\]
Summing over \(t=0,\ldots,T-1\) proves
\eqref{eq:exit-small-conductance}.
\end{proof}

\begin{proof}[Proof of \cref{prop:one-vertex-defect}]
By \cref{lem:defect-conductance},
\(
        \Phi(A_n)\le e^{-c_0n}
\)
for all sufficiently large \(n\).  Apply
\cref{lem:exit-small-conductance} with
\(
        T_n=\left\lfloor e^{c_0n/2}\right\rfloor.
\)
Then
\[
        \mathbb P_{\mu_\beta^k(\cdot\mid A_n)}
        (\tau_{A_n^c}\le T_n)
        \le e^{-c_0n/2}.
\]
By \cref{lem:defect-separation},
\(\Gamma_{p,\varepsilon}\subseteq A_n^c\).  Hence
\[
        \{\tau_{\Gamma_{p,\varepsilon}}\le T_n\}
        \subseteq
        \{\tau_{A_n^c}\le T_n\}.
\]
Averaging over the initial state under
\(\mu_\beta^k(\cdot\mid A_n)\), we conclude that there exists a deterministic
\(x_n\in A_n\) such that
\[
        \mathbb P_{x_n}(\tau_{A_n^c}\le T_n)
        \le e^{-c_0n/2}.
\]
Since \(\Gamma_{p,\varepsilon}\subseteq A_n^c\), the same initial state
satisfies
\[
        \mathbb P_{x_n}
        (\tau_{\Gamma_{p,\varepsilon}}\le T_n)
        \le e^{-c_0n/2}.
\]
In particular, \(x_n\notin\Gamma_{p,\varepsilon}\), and, for every
\(0\le t\le T_n\),
\[
        x_nP^t(A_n)
        \ge 1-e^{-c_0n/2}.
\]
On the other hand, since \(A_n\subseteq\Gamma_{p,\varepsilon}^c\),
\cref{prop:1,theo of appro,le-1} gives
\(
        \mu_\beta^k(A_n)
        \le Ce^{-c_1n}.
\)
Therefore
\[
\begin{aligned}
        d_{\mathrm{TV}}(x_nP^t,\mu_\beta^k)
        &\ge
        x_nP^t(A_n)-\mu_\beta^k(A_n) \ge
        1-e^{-c_0n/2}-Ce^{-c_1n}.
\end{aligned}
\]
Finally,
\[
        \mathbb E_{x_n}\tau_{\Gamma_{p,\varepsilon}}
        \ge
        T_n\,
        \mathbb P_{x_n}
        (\tau_{\Gamma_{p,\varepsilon}}>T_n).
\]
After decreasing the constant \(c\), all the assertions follow.
\end{proof}

We finish by giving an explicit parameter choice for which the contraction
inside the regular region and the one-vertex obstruction coexist.

\begin{corollary}
\label{cor:explicit-one-vertex-defect}
Let
\[
        p=0.01,
        \qquad
        \beta=4.
\]
Then the assumptions of \cref{coupling} hold.  Nevertheless, for every
sufficiently small fixed \(\varepsilon>0\), along all sufficiently large
admissible values of \(n\) there are deterministic initial configurations for
which the total-variation mixing time is exponential in \(n\).
\end{corollary}

\begin{proof}
First,
\[
        c_{p}
        =
        \frac{\log(p/(1-p))}{2p-1}
        \approx 4.688897806>4,
\]
so \(\beta<c_{p}\).  Moreover,
\[
        4\beta p(1-p)=0.1584<\frac12.
\]
Thus the conditional model is in the replica-symmetric region and the
regular-region contraction condition in \cref{coupling} is satisfied.

For the effective potential
\(
        V(r)=\Dkl(r\|0.01)-4(r-0.01)^2,
\)
take
\(
        b=0.8,
        q=0.95.
\)
A direct calculation gives
\[
        V(0.8)\approx0.6893437924,
        \qquad
        V(0.95)\approx0.6424989501,
\]
and hence
\[
        V(0.8)-V(0.95)
        \approx0.0468448423>0.
\]
The conclusion follows from \cref{prop:one-vertex-defect}.
\end{proof}

%\begin{remark}
%The obstruction is invisible at the graphon large deviation scale.  Changing the row and column incident to a single vertex modifies only \(O(n)\) edge indicators, and therefore changes the cut distance by at most \(O(n^{-1})\). It is nevertheless detected by \(\Gamma_{p,\varepsilon}\), which imposes a uniform local regularity condition on every edge-position.  Thus \cref{coupling} is genuinely a regular-basin, or metastable, mixing result and cannot in general be upgraded to a polynomial worst-case mixing bound.
%\end{remark}

\end{appendices}

\end{document}